\documentclass[11pt,draft]{amsart}
\usepackage[utf8]{inputenc}
\usepackage[T1]{fontenc}
\usepackage[letterpaper,centering]{geometry}
\usepackage{lmodern}
\usepackage{eucal}
\usepackage{mathrsfs}
\usepackage[shortlabels]{enumitem}
\usepackage{amsmath,amssymb,amsfonts}
\usepackage{xcolor}
\usepackage{graphicx}
\definecolor{darkblue}{rgb}{0,0,0.3}
\definecolor{darkgreen}{rgb}{0,0.4,0}
\usepackage[colorlinks=true,linkcolor=darkblue, citecolor=darkgreen, unicode,pdfborder={0 0 0},final]{hyperref}
\usepackage[ps,arrow,matrix,tips,line]{xy}
{\setbox0\hbox{$ $}}\fontdimen16\textfont2=\fontdimen17\textfont2
\entrymodifiers={+!!<0pt,\the\fontdimen22\textfont2>}
\SelectTips{cm}{11}
\setlist[enumerate]{label={\upshape(\arabic*)},partopsep=.5ex,leftmargin=*}
\setlist[itemize]{leftmargin=*}

\usepackage{ifdraft}
\ifdraft{

}{}

\theoremstyle{plain}
\newtheorem{thm}{Theorem}[section]
\newtheorem{quests}[thm]{Questions}
\newtheorem{quest}[thm]{Question}

\newtheorem{lem}[thm]{Lemma}
\newtheorem{cor}[thm]{Corollary}
\newtheorem{prop}[thm]{Proposition}
\theoremstyle{definition}

\newtheorem{rmk}[thm]{Remark}
\newtheorem{rmks}[thm]{Remarks}

\numberwithin{equation}{section}
\theoremstyle{remark}
\newtheorem{Case}{Case}

\def\clap#1{\hbox to 0pt{\hss#1\hss}}

\def\mathrlap{\mathpalette\mathrlapinternal}

\def\mathrlapinternal#1#2{%
\rlap{$\mathsurround=0pt#1{#2}$}}

\input cyracc.def
\DeclareFontFamily{U}{russian}{}
\DeclareFontShape{U}{russian}{m}{n}
        { <5><6> wncyr5
        <7><8><9> wncyr7
        <10><10.95><12><14.4><17.28><20.74><24.88> wncyr10 }{}
\DeclareSymbolFont{Russian}{U}{russian}{m}{n}
\DeclareSymbolFontAlphabet{\mathcyr}{Russian}
\makeatletter
\let\@math@cyr\mathcyr
\renewcommand{\mathcyr}[1]{\@math@cyr{\cyracc #1}}
\makeatother
\newcommand{\vinbas}{\text{\rotatebox[origin=c]{-90}{$\in$}}}
\newcommand{\vinhaut}{\text{\rotatebox[origin=c]{90}{$\in$}}}
\newcommand{\congru}[3]{#1 \equiv #2 \!\!\mod #3}

\newcommand{\isoto}{\myxrightarrow{\,\sim\,}}
\makeatletter
\def\myrightarrow{{\setbox\z@\hbox{$\rightarrow$}\dimen0\ht\z@\multiply\dimen0 6\divide\dimen0 10\ht\z@\dimen0\box\z@}}
\def\myrightarrowfill@{\arrowfill@\relbar\relbar\myrightarrow}
\newcommand{\myxrightarrow}[2][]{\ext@arrow 0359\myrightarrowfill@{#1}{#2}}
\makeatother

\newcommand{\et}{{\mathrm{\acute et}}}

\newcommand{\oc}{{\bar c}}
\newcommand{\ox}{{\bar x}}
\newcommand{\oy}{{\bar y}}
\newcommand{\oz}{{\bar z}}
\newcommand{\oo}{{\bar \omega}}
\newcommand{\dA}{{\widetilde{\mathcal A}}}
\newcommand{\cA}{{\mathcal A}}

\newcommand{\sC}{{\mathscr C}}

\newcommand{\sE}{{\mathscr E}}
\newcommand{\sF}{{\mathscr F}}
\newcommand{\sH}{{\mathscr H}}
\newcommand{\sK}{{\mathscr K}}
\newcommand{\sO}{{\mathscr O}}
\newcommand{\sQ}{{\mathscr Q}}

\newcommand{\tX}{{\widetilde X}}

\newcommand{\torsion}{{\mathrm{tors}}}

\renewcommand{\L}{{\mathbf L}}

\renewcommand{\P}{{\mathbf P}}

\newcommand{\Q}{{\mathbf Q}}
\newcommand{\R}{{\mathbf R}}
\newcommand{\C}{{\mathbf C}}
\newcommand{\bS}{{\mathbf S}}

\newcommand{\Z}{{\mathbf Z}}

\newcommand{\CH}{\mathrm{CH}}
\newcommand{\nr}{\mathrm{nr}}

\newcommand{\Tor}{\mathrm{Tor}}
\newcommand{\Gal}{\mathrm{Gal}}

\newcommand{\GL}{\mathrm{GL}}

\newcommand{\cl}{\mathrm{cl}}

\newcommand{\rk}{\mathrm{rk}}
\renewcommand{\deg}{\mathrm{deg}}

\newcommand{\Pic}{\mathrm{Pic}}
\newcommand{\Br}{\mathrm{Br}}
\newcommand{\Spec}{\mathrm{Spec}}

\renewcommand{\phi}{\varphi}
\renewcommand{\emptyset}{\varnothing}

\newcommand{\RR}{{\mathrm{R}}}
\newcommand{\OO}{{\mathrm{O}}}
\newcommand{\U}{{\mathrm{U}}}

\newcommand{\BO}{{\mathrm{BO}}}
\newcommand{\MU}{{\mathrm{MU}}}
\newcommand{\BU}{{\mathrm{BU}}}

\newcommand{\KR}{{\mathrm{KR}}}
\newcommand{\KO}{{\mathrm{KO}}}
\newcommand{\KT}{{\mathrm{K3}}}
\newcommand{\ind}{{\mathrm{ind}}}

\newcommand{\alg}{{\mathrm{alg}}}
\newcommand{\ac}{{\mathrm{ac}}}
\newcommand{\an}{{\mathrm{an}}}
\renewcommand{\top}{{\mathrm{top}}}
\newcommand{\Gr}{{\mathrm{Gr}}}
\newcommand{\coind}{{\mathrm{coind}}}
\newcommand{\pt}{{\mathrm{pt}}}
\newcommand{\Sq}{{\mathrm{Sq}}}
\newcommand{\Ker}{{\mathrm{Ker}}}

\newcommand{\Hom}{{\mathrm{Hom}}}
\newcommand{\Ext}{{\mathrm{Ext}}}
\newcommand{\RHom}{R\mathscr{H}\mkern-4muom}

\newcommand{\Homrond}{\mathscr{H}\mkern-4muom}

\newcommand{\ci}{\mathcal{C}^{\infty}}

\newcommand{\Nr}{\mathrm{N}}

\newcommand{\Id}{\mathrm{Id}}

\def\loccit{\emph{loc}.\kern3pt \emph{cit}.{}\xspace}
\def\eg{e.g.\kern.3em}
\def\ie{i.e.\ }
\def\resp {\text{resp.}\ }

 \makeatletter
 \renewcommand{\tocsection}[3]{%
   \indentlabel{\@ifnotempty{#2}{\bfseries\ignorespaces#1 #2\quad}}\bfseries#3}

 \renewcommand{\tocsubsection}[3]{%
   \indentlabel{\@ifnotempty{#2}{\hspace{1.6em}\ignorespaces#1 #2\quad}}#3}
 \makeatother

\makeatletter\let\@wraptoccontribs\wraptoccontribs\makeatother

\date{June 10th, 2024; revised on April 30th, 2026}
\title{The Wu relations in real algebraic geometry}

\author{Olivier Benoist}
\address{D\'epartement de math\'ematiques et applications, \'Ecole normale sup\'erieure et CNRS, 45~rue d'Ulm, 75230 Paris Cedex 05, France}
\email{olivier.benoist@ens.fr}

\author{Olivier Wittenberg}
\address{Institut Galil\'ee, Universit\'e Sorbonne Paris Nord, 99~avenue Jean-Baptiste Cl\'ement, 93430 Villetaneuse, France}
\email{wittenberg@math.univ-paris13.fr}

\begin{document}

\begin{abstract}
We construct and study relations between Chern classes and Galois cohomology classes in the $\Gal(\C/\R)$-equivariant cohomology of real algebraic varieties with no real points. We give applications to the topology of their sets of complex points, and to sums of squares problems. In particular, we show that $-1$ is a sum of $2$ squares in the function field of any smooth projective real algebraic surface with no real points and with vanishing geometric genus, as well as higher-dimensional generalizations of this result.
\end{abstract}

\maketitle

\section{Introduction}
\label{sec:intro}

The aim of this article is to show that tools from classical algebraic topology---namely,
universal relations between characteristic classes of manifolds, via Wu classes and, more generally,
via right
actions of the Steenrod algebra---can be used in an equivariant stably complex
setting (see \S\ref{parintrorelations}) to derive new results in real algebraic geometry.
These concern the topology of real algebraic varieties (see \S\ref{parintrotopology}) and
sums of squares problems (see~\S\ref{parintrosquares}).
In the latter direction, we  answer
an old question of van Hamel
about levels of function fields of real algebraic surfaces,
improve Pfister's bound on the level of function fields of higher-dimensional real algebraic varieties
in the case of geometrically uniruled threefolds and fourfolds,
and obtain the best known quantitative results about Hilbert's 17th problem on sums of squares
in $\R(x_1,\dots,x_n)$ in the case of polynomials of degree~$d\leq n$.

\subsection{Relations between tautological classes}
\label{parintrorelations}

Let~$X$ be a smooth algebraic variety over~$\R$ such that $X(\R)=\emptyset$. Set $n=\dim(X)$.
Let~$\Z(j)$ denote the abelian group~$\Z$ endowed with
the action of $G=\Gal(\C/\R)\simeq\Z/2$
by multiplication by~$(-1)^j$,
and~$X(\C)$ the space of complex points of $X$, endowed with the natural action of~$G$.

\subsubsection{Tautological classes}

We are interested in tautological classes that live in
the Borel $G$\nobreakdash-equivariant cohomology groups
$H^i_G(X(\C),\Z(j))$ (see~\S\ref{notation}).  These come in two sorts.

A first source of tautological classes is the cohomology of the group~$G$
itself: one can consider the image $\omega \in H^1_G(X(\C),\Z(1))$ of the
generator of $H^1(G,\Z(1))\simeq \Z/2$ by the pull-back morphism
$H^1(G,\Z(1))=H^1_G(\pt,\Z(1))\to H^1_G(X(\C),\Z(1))$, as well as its
powers $\omega^i \in H^i_G(X(\C),\Z(i))$ with respect to cup product, for
$i\geq 0$.

Tautological classes also arise from a second source; namely, the tangent
bundle of~$X(\C)$ has $G$\nobreakdash-equivariant Chern classes
$c_i\in H^{2i}_G(X(\C),\Z(i))$ for $i \in \{1,\dots,\dim(X)\}$.  These
were introduced by
Kahn~\cite{kahnchern}.

\subsubsection{The Wu relations}

Determining which of the~$\omega^i$ vanish is a delicate problem.
The following theorem, which we establish
in~\textsection\ref{seccomputations},
addresses this question for surfaces and serves as a first step toward the
substantially more general results discussed below.
It answers a question raised by Krasnov in \mbox{\cite[p.~758]{Krasnov}},
and  has applications both to the topology of real algebraic surfaces
and to sums of squares problems (see
Theorem~\ref{coindexsurfacesintro} and \textsection\ref{subsubsec:levelsurf}).
 
\begin{thm}[Corollary \ref{cor:omega3}]
\label{omega3intro}
Let $X$ be a smooth algebraic surface over~$\R$.
If $X(\R)=\varnothing$, then $\omega^3=0$ in $H^3_G(X(\C),\Z(3))$.
\end{thm}

We refer to Remarks \ref{remsomega3} for more context and for a short direct proof of Theorem \ref{omega3intro}.

The main thrust of the present article is the observation that the relation
$\omega^3=0$ given by Theorem \ref{omega3intro} when $n=2$ fits, in higher
dimensions, into a maze of polynomial relations between $\omega$ and the
$c_i$. To give one example, the identity
\begin{equation}
\label{relationintro}
\omega^5=\omega(c_1^2+c_2)\,\,\textrm{ in }\,\,H^5_G(X(\C),\Z(5))
\end{equation}
always holds when $n=4$ (see Remark \ref{evendimrmks} (iii)). We refer to Proposition \ref{lowrelations} for more examples of such polynomial relations (in equivariant cohomology with~$\Z/2$ coefficients, or, equivalently,
between the reductions modulo~$2$ of~$\omega$ and of the~$c_i$) in low dimensions.

That such relations exist, at least modulo~$2$, can be seen as follows.  It was discovered by Wu \cite{Wu} that some polynomials in the Stiefel--Whitney classes of the tangent bundle vanish for all compact~$\ci$ manifolds of a fixed dimension. These relations were refined and eventually fully understood in successive works of Dold \cite{Dold}, Massey \cite{Massey} and Brown and Peterson \cite{BP, BP2}---as it turns out, they can always be explained by an analysis of the interaction between Poincar\'e duality and Steenrod operations.  Writing out the Stiefel\nobreakdash--Whitney classes of the tangent bundle of
the compact~$\ci$ manifold
$X(\C)/G$ as functions of the reductions modulo $2$ of $\omega$ and of the~$c_i$ (see Proposition~\ref{propequivw}
and Remark~\ref{rk:SetS/G}~(i)) then yields relations of the desired form, at least when~$X$ is proper.

One can further exploit the complex structure of $X(\C)$ and its compatibility with the $G$-action to produce additional relations, much in the same way that Massey \cite{Massey} and Brown and Peterson \cite[Theorem~4.3]{BP} construct additional relations between Stiefel\nobreakdash--Whitney classes for manifolds that are assumed to be orientable. This refinement is already needed to prove that $\omega^3=0$ when $n=2$,
as
 the latter vanishing  does \emph{not}
follow from the universal polynomial
relations satisfied by Stiefel--Whitney
classes of (orientable) compact~$\ci$ fourfolds applied to $X(\C)/G$
(see Remark~\ref{rk:4.1versusBP} below).

The fundamental role played by the complex structure of~$X(\C)$
in all of the theorems stated in~\textsection\ref{sec:intro}
encouraged us to cast the construction of our relations in its natural topological setting: that of stably complex $G$\nobreakdash-equivariant manifolds with no $G$\nobreakdash-fixed points (see \S\ref{parstablycomplex}).  Two benefits of this extended framework are that it allows us to show that our relations also hold in the $G$-equivariant cohomology of $X(\C)\setminus X(\R)$ for any smooth variety over $\R$ (and more generally of complex manifolds endowed with a fixed-point free antiholomorphic involution), and that it is sometimes easier to construct interesting examples in this setting (see \eg Remark~\ref{remexemplessc}~(iii)).

\subsubsection{Relations modulo \texorpdfstring{$2$}{2}}

The next theorem states the universal polynomial relations that we produce
between the reductions modulo~$2$ of~$\omega$ and of the~$c_i$, for an
arbitrary stably complex~$\ci$ $G$\nobreakdash-manifold~$M$ of dimension~$n$
with
$M^G=\emptyset$.
To formulate it, we introduce two more pieces of notation.
First,
following Adams and Brown--Peterson,
we need
a certain right action of the mod~$2$ Steenrod algebra~$\cA$ on the cohomology groups $H^*_G(M,\Z/2)$,
denoted $(x)a \in H^{k+l}_G(M,\Z/2)$ for $x \in H^k_G(M,\Z/2)$ and $a \in \cA$ homogeneous of
degree~$l$.
Its definition
can be found in~\textsection\ref{parrightaction}
or in \cite[\textsection6]{BP} (noting that $H^*_G(M,\Z/2)=H^*(M/G,\Z/2)$).
When~$M$ is compact,
it is characterised by the equality
$x \smile a(y) = (x)a \smile y$ for $y \in H^{2n-k-l}_G(M,\Z/2)$
(see \cite[Corollary~6.3]{BP}).
Secondly, we denote by $R_k^l \subset \cA$ the subgroup spanned by the elements
$\Sq^I$ when~$I$ ranges over the degree~$l$ admissible sequences of integers
satisfying the explicit conditions listed
in \cite[Theorem~4.4~(a)--(d)]{BP} (with $s=k$).

\begin{thm}[Theorem~\ref{reltopo}]
\label{introth:relmod2}
Let~$M$ be a stably complex~$\ci$ $G$-manifold of dimension~$n$ with $M^G=\emptyset$.
Let $\oo \in H^1_G(M,\Z/2)$ and $\oc_i \in H^i_G(M,\Z/2)$
denote the reductions modulo~$2$
of~$\omega$
and of the $G$\nobreakdash-equivariant Chern classes of the stable tangent bundle of~$M$.
Let $k,l,e,b_1,\dots,b_n$ be nonnegative integers such that $e + 2 \sum_{i=1}^n ib_i = 2n-l-k$.
Let $a \in R^l_k$.
\begin{enumerate}[(i)]
\item If $k<l$, then $(\oo^e \oc_1^{b_1} \dots \oc_n^{b_n})\Sq^l=0$
in $H^{2n-k}_G(M,\Z/2)$.
\item If $n+e+\sum_{i=1}^n ib_i$ is even, then\vphantom{\raise 8pt\hbox{x}} $(\oo^e \oc_1^{b_1} \dots \oc_n^{b_n})a=0$
in $H^{2n-k}_G(M,\Z/2)$.
\end{enumerate}
\end{thm}

Even though not immediately apparent,
the relations asserted by Theorem~\ref{introth:relmod2} are indeed polynomial
relations between~$\oo$ and the~$\oc_i$ that do not depend on the
manifold~$M$ (of fixed dimension~$n$).  This is because the classes~$\oo$ and~$\oc_i$
come, by
pull-back, from the $G$\nobreakdash-equivariant cohomology of the
classifying space~$\BU$, and there is a natural right action, depending
only on~$n$, of~$\cA$ on $H^*_G(\BU,\Z/2)=\Z/2[\oo, (\oc_i)_{i\geq 1}]$
such that the pull-back map $H^*_G(\BU,\Z/2) \to H^*_G(M,\Z/2)$ respects
the right actions of~$\cA$
(see~\textsection\textsection\ref{parGeqvb}--\ref{parrightaction}).
In addition, for any given~$n$, this right action of~$\cA$ on
$\Z/2[\oo, (\oc_i)_{i\geq 1}]$ can be made explicit by combining
Proposition~\ref{propequivw},
which expresses
the $G$\nobreakdash-equivariant Stiefel--Whitney classes of~$M$ in terms of~$\oo$
and the~$\oc_i$,
with Wu's formula for the (left) action of~$\cA$ on Stiefel--Whitney classes.

\subsubsection{Integral relations}

There are no nontrivial integral polynomial relations between the Chern classes of all smooth projective complex varieties of a fixed dimension.
(To see it,  reduce to the case of Chern numbers and apply \cite[Proposition 4.3.8]{Kochman} with $\mathfrak{B}=\BU$ and $E=K(\Z)$, taking into account that the Hurewicz map $\pi_*(\MU)\to H_*(\MU,\Z)$ is a rational isomorphism (see \cite[\S II.8]{AdamsMU}) and that $\pi_*(\MU)$ is generated by classes of smooth projective complex varieties (see \cite[Corollary II.10.8]{AdamsMU}).) 
It is therefore a striking feature of the relations we uncover (as opposed to Brown and Peterson's) that some of them hold integrally,  as in Theorem \ref{omega3intro} or in~(\ref{relationintro}).

Let us state the universal polynomial relations that we obtain between~$\omega$ and the~$c_i$.
To this end, consider the group homomorphism
$\beta_j:\Z/2[\oo, (\oc_i)_{i\geq 1}] \to \Z[\omega, (c_i)_{i\geq 1}]/(2\omega)$,
for $j \in \{0,1\}$,
defined by
$\beta_j\big(\oo^e \oc_1^{b_1} \dots \oc_n^{b_n}\big)=(j + e + \sum_{i=1}^n ib_i)\omega^{e+1} c_1^{b_1} \dots c_n^{b_n}$.
We note that for any $P \in \Z/2[\oo, (\oc_i)_{i\geq 1}]$
and any $j \in \{0,1\}$,
evaluating $\beta_j(P)$ on $\omega \in H^1_G(M,\Z(1))$
and on the $c_i \in H^{2i}_G(M,\Z(i))$ yields an element of $H^*_G(M,\Z(j))$.

\begin{thm}[Theorem~\ref{reltopoZ}]
\label{introth:relint}
Let~$M$ be a stably complex~$\ci$ $G$-manifold of dimension~$n$ with $M^G=\emptyset$.
Let $P \in \Z/2[\oo, (\oc_i)_{i\geq 1}]$.
Let $k,l,e,b_1,\dots,b_n$ be nonnegative integers
such that $e + 2 \sum_{i=1}^n ib_i = 2n-2^{l+1}k$
and that $n+e+\sum_{i=1}^n ib_i$ is even.
\begin{enumerate}[(i)]
\item
If~$P$ is any of the polynomials that Theorem~\ref{introth:relmod2} asserts to vanish in $H^*_G(M,\Z/2)$,
then $\beta_0(P)$ vanishes in $H^*_G(M,\Z)$ and~$\beta_1(P)$ vanishes in $H^*_G(M,\Z(1))$.

\item
If~$P$ is the polynomial obtained by writing out
$(\oo^e \oc_1^{b_1} \dots \oc_n^{b_n}) \Sq^{2^lk}\cdots\Sq^{2k}\Sq^k$\vphantom{\raise 8pt\hbox{x}}
in terms of~$\oo$ and the~$\oc_i$ as indicated in the paragraph that follows Theorem~\ref{introth:relmod2},
then $\beta_0(P)$ vanishes in $H^*_G(M,\Z)$ and~$\beta_1(P)$ vanishes in $H^*_G(M,\Z(1))$.
\end{enumerate}
\end{thm}

\subsubsection{Proving and using the relations}

To prove Theorems \ref{introth:relmod2} and \ref{introth:relint}, we first reduce to the case where $M$ is compact.  With integral coefficients, a source of significant difficulties for this reduction
is the possible failure of the Mittag--Leffler condition---see~\S\ref{parZpolynomial} and~\S\ref{parintrelations}. In the compact case, the theorems are proved by the Brown--Peterson method of exploiting Poincar\'e duality.
For Theorem \ref{introth:relint}, we make use of an integral version of Poincar\'e duality that pairs cohomology groups with coefficients in $\Z$ and $\Q/\Z$; see~\S\ref{parPoincare}.

Despite the explicit computability of the right actions of the Steenrod algebra appearing in the
statements of Theorems~\ref{introth:relmod2} and~\ref{introth:relint},
it is not easy to determine precisely which
relations follow from those output by these two theorems.  This is a major
obstacle one must overcome when applying them to concrete problems, as we
do in \S\ref{parintrotopology} and~\S\ref{parintrosquares} below.

\subsection{Applications to the topology of \texorpdfstring{$X(\C)$}{X(C)}}
\label{parintrotopology}

Among the relations between $\omega$ and the~$c_i$ that we construct, some do not involve the $c_i$ (as for example in Theorem~\ref{omega3intro}).  It turns out to be difficult even to pinpoint exactly which of our relations only involve~$\omega$.  Our best result in this direction is the following generalization of Theorem \ref{omega3intro}.

\begin{thm}[see Corollary \ref{th:omega2n-1ouvert}]
\label{omega2n-1intro}
Let $X$ be a smooth variety of dimension $n$ over $\R$ with~$X(\R)=\varnothing$.
If $n$ is not of the form $2^k-1$, then $\omega^{2n-1}=0$ in $H^{2n-1}_G(X(\C),\Z(2n-1))$.
\end{thm}

The hypotheses that $n$ is not of the form $2^k-1$, and that $X$ is smooth, cannot be dispensed with (see Propositions~\ref{prop:omegai quadriques} and \ref{singularomega3}).
The exponent $2n-1$ in Theorem \ref{omega2n-1intro} (and in its extension to stably complex $G$-equivariant manifolds) is optimal for infinitely many values of $n$ not of the form $2^k-1$ (see Proposition \ref{prop:omegai quadriques} and Theorem \ref{nonalgexamples}),  but not for all of them
(see Proposition \ref{omegalow} for an improvement when $n=6$, and see Question~\ref{questvanomega}).

The question of the vanishing of the $\omega^e$ (and of their reductions modulo $2$) on real algebraic varieties with no real points was raised and studied by Krasnov in \cite{Krasnov}.  His nontrivial results in this direction concern the vanishing of these classes on particular kinds of varieties (mostly surfaces),  and not on all varieties of a fixed dimension as in Theorem~\ref{omega3intro} and in its generalization Theorem \ref{omega2n-1intro}.

When $n$ is not of the form $2^k-1$, the vanishing of $\omega^{2n-1}$ established in Theorem \ref{omega2n-1intro} is a topological restriction on the $G$-space (\ie on the topological space with involution)~$X(\C)$. One can use this vanishing to control a more classical invariant of this $G$-space: its coindex.
The \textit{coindex} $\coind(S)$ of a $G$-space $S$ with $S^G=\varnothing$ was defined by
Yang~\cite{YangII} (under the name ``$B$\nobreakdash-index'') and by Conner and Floyd~\cite{CF1,CF2} to be the smallest nonnegative integer~$m$ such that there exists a $G$-equivariant map $f\colon S\to\bS^m$, where the sphere $\bS^m$ is endowed with the antipodal involution
(or $\coind(S)=\infty$ if no such~$m$ exists).
 It is a measure of the topological complexity of $S$.  Theorem \ref{omega2n-1intro} implies:

\begin{thm}[see Theorems \ref{coindexn} and \ref{coindexn2}]
\label{coindexnintro}
Let $X$ be a variety of dimension $n$ over $\R$ with $X(\R)=\varnothing$. If $n$ is not of the form $2^k-1$, then 
$\coind(X(\C))\leq 2n-1$.  If moreover $X$ is smooth with no proper irreducible component, then $\coind(X(\C))\leq 2n-2$. 
\end{thm}

Our investigation of the coindex of $G$-spaces of the form $X(\C)$ also yields a complete understanding of this invariant when $X$ has dimension at most $2$ (see Proposition \ref{coindexcurves} for the easy case of curves and Proposition~\ref{prop2obs} and Lemma \ref{coind0ou1} for surfaces). In particular, relying on Theorem \ref{omega3intro}, we prove the following result:

\begin{thm}[Theorems \ref{coindexsurfaces} and \ref{coindexblowup}]
\label{coindexsurfacesintro}
The possible values of $\coind(X(\C))$ for an irreducible smooth proper surface $X$ over $\R$ with $X(\R)=\varnothing$ are $0$, $1$, $2$ and $3$.  For all such~$X$, 
there exists a projective birational morphism $\pi\colon Y\to X$ with $\coind(Y(\C))\leq 2$.
\end{thm}

\subsection{Applications to sums of squares}
\label{parintrosquares}

Let $F$ be a field.  Pfister defined the \textit{level}~$s(F)$ of $F$ to be the smallest integer $s$ such that $-1$ is a sum of $s$ squares in $F$, if such an integer exists,  or $+\infty$ otherwise. When $s(F)$ is finite, it is a power of~$2$, by \cite[Satz~4]{PfisterStufe}. If~$X$ is an irreducible smooth variety of dimension $n$ over $\R$, it is a consequence of Artin's solution to Hilbert's 17th problem \cite{Artin17} that $s(\R(X))$ is finite if and only if~$X(\R)=\varnothing$. In this case, Pfister \cite[Theorem~2]{Pfister2} has shown that $s(\R(X))\leq 2^n$.  It is not known if this bound is optimal when $n\geq 3$.  It is indeed the best possible for~$n=0$ (take~$X=\Spec(\C)$), for $n=1$ (take $X$ to be the anisotropic conic over $\R$), and for $n=2$ by the Cassels--Ellison--Pfister theorem (see \cite{cep} or \cite[Theorem]{ctnoetherlefschetz}, and combine it with \cite[Lemma~1.2]{ctnoetherlefschetz}).

\subsubsection{The level of real algebraic varieties}

In contrast with the classical works of Brown and Peterson \cite{BP,BP2},
the relations that we produce involve cohomology classes of very different
natures.
On the one hand, the class $\omega$ comes from the Galois cohomology of the base field $\R$, and influences the arithmetic of the function field~$\R(X)$ of~$X$. On the other hand, the classes $c_i$ come from algebraic cycles and in particular they vanish on a dense Zariski open subset of $X$ (for~$i\geq 1$).

The different roles played by these two kinds of classes can be leveraged to study the level of real functions fields.  Indeed, it follows from the Milnor conjecture, proved by Voevodsky~\cite{Voevodsky}, that the level of $\R(X)$ is $\leq 2^{k-1}$ if and only if $\omega^k$ vanishes on some dense Zariski open subset of $X$ (see Lemma \ref{lowlevelcrit}).  Consequently, a relation between~$\omega$ and the $c_i$ nontrivially involving the monomial $\omega^k$ would imply that $s(\R(X))\leq 2^{k-1}$.

This remark does not immediately lead to improvements to Pfister's $s(\R(X))\leq 2^n$ bound, because all our relations live in degree $\geq n+1$.  However,  under appropriate geometric hypotheses on $X$,  one can combine this strategy with additional arguments (notably, when $n$ is odd, 
Bloch--Ogus theory applied as in \cite{Hilbert17}),  to improve Pfister's~$2^n$ bound on the level (see Theorems \ref{evendim}, \ref{odddim} and \ref{conicbundles}). Let us only state here the following concrete consequence (which also relies, when $n=3$, on Voisin's proof of the integral Hodge conjecture for uniruled complex threefolds \cite{voisinthreefolds}).

\begin{thm}[Corollaries \ref{coreven} and \ref{cor3}]
\label{uniruledintro}
Let $X$ be an irreducible smooth proper variety of dimension $n\geq 1$ over~$\R$ with $X(\R)=\varnothing$.  Assume that $X_{\C}$ is uniruled, and that $n$ is even or that $n\leq 3$. Then $s(\R(X))\leq 2^{n-1}$.
\end{thm}

When $n\geq 5$ is odd,  the conclusion of Theorem \ref{uniruledintro} is still valid under the additional hypothesis that the group $H^{n+1}(X(\C),\Z)/N^2H^{n+1}(X(\C),\Z)$ has no $2$-torsion, where $N^{\bullet}$ denotes the coniveau filtration (see Remark \ref{remsodd} (ii)). This additional hypothesis might always hold when $X_{\C}$ is uniruled (see Question \ref{q:torsionfree}).

\subsubsection{The level of real algebraic surfaces}
\label{subsubsec:levelsurf}

In the particular case of real algebraic surfaces, the relation on which Theorem \ref{uniruledintro} relies is precisely $\omega^3=0$ (see Theorem \ref{omega3intro}). In this case, the proof of Theorem \ref{uniruledintro} can be combined with the Lefschetz $(1,1)$ theorem to improve Pfister's bound for real surfaces with vanishing geometric genus, thereby answering in the negative a question of van Hamel \cite[Remark~3.5~(ii)]{vanhameldivisors}.

\begin{thm}[Corollary \ref{2dim}]
\label{levelsurfaceintro}
Let $X$ be an irreducible smooth proper surface over $\R$ with~$X(\R)=\varnothing$ and $H^2(X,\sO_X)=0$. 
Then $s(\R(X))\leq 2$.
\end{thm}

Theorem \ref{levelsurfaceintro} was known to hold in some particular cases: for geometrically rational surfaces (Parimala and Sujatha \cite[Theorem 5]{parimalasujatha})
and more generally for surfaces with~$H^1(X(\C),\Z/2)=0$ and $H^2(X,\sO_X)=0$ (Silhol, see \cite[Corollary 2.4]{vanhameldivisors}, see also Krasnov \cite[Corollary 2.4 (ii)]{Krasnov}), as well as for Enriques surfaces (Krasnov \cite[Corollary 2.4 (iii) and Theorem~2.1]{Krasnov}, Sujatha and van Hamel \cite[Theorem 3.2]{sujathavanhamel};
these surfaces satisfy~$H^1(X(\C),\Z/2)=\Z/2$ and~$H^2(X,\sO_X)=0$).  Theorem \ref{levelsurfaceintro} encompasses more examples, such as Campedelli surfaces. 
We note that the arguments of~\cite{Krasnov} and~\cite{sujathavanhamel} about Enriques surfaces
rely on specific geometric properties of these surfaces that are not susceptible to generalization.

Although there do exist surfaces with $H^2(X,\sO_X)\neq 0$ and $s(\R(X))=2$ (for instance, products of two curves of genus $\geq 1$ with no real points), the hypothesis that~${H^2(X,\sO_X)=0}$ in Theorem \ref{levelsurfaceintro} is essentially optimal.  This is apparent in Colliot-Th\'el\`ene's construction of real algebraic surfaces with $s(\R(X))=4$ based on the Noether--Lefschetz theorem \cite{ctnoetherlefschetz}, which makes essential use of the nontriviality of the Hodge structure on the degree $2$ Betti cohomology group $H^2(X(\C),\Z)$ (see Remark \ref{evendimrmks} (i)).

\subsubsection{Hilbert's 17th problem in low degree}

The quantitative version of Hilbert's 17th problem aims at writing a nonnegative polynomial $f\in\R[x_1,\dots,x_n]$ as a sum of few squares of rational functions. Pfister's celebrated theorem \cite[Theorem~1]{Pfister2} shows that $2^n$ squares always suffice.  Better bounds on the level of real function fields give rise to improvements to Pfister's theorem (see \S\ref{parH17} for more context and more details). To demonstrate the applicability of our results on the level, we use them to significantly lower Pfister's bounds for polynomials of low degree (in the spirit of \cite{Hilbert17}, with much better bounds under a stronger hypothesis on the degree).

\begin{thm}[Theorem \ref{thlowerdegree}]
\label{th:intro17}
Let $f\in\R[x_1,\dots,x_n]$ be a nonnegative polynomial of degree~$d$. If $2\leq d\leq n$ and $(d,n)\neq(2,2)$, then $f$ is a sum of $2^{n-1}$ squares in $\R(x_1,\dots, x_n)$. 
\end{thm}

\subsection{Organization of the text}

Section \ref{secprelim} gathers generalities on $G$-equivariant algebraic topology that are used throughout the text.
Section \ref{secrelations} is devoted to the construction of the relations between the Galois cohomology class~$\omega$ and the Chern classes $c_i$ which form the heart of this article. We work in the appropriate topological generality, that of stably complex $\ci$ $G$-manifolds.  In Section \ref{seccomputations}, we explain how to use the formalism developed in Section \ref{secrelations} to produce concrete relations, and we put forward those that will be useful in the remainder of the text. Section~\ref{secomegai} focuses on relations of the form~$\omega^e=0$, and investigates the optimality of our results in this direction. 

Applications to the coindex of real algebraic varieties with no real points, and to the level of real function fields (as well as to Hilbert's 17th problem), are respectively given in Sections~\ref{seccoindex} and \ref{seclevel}. In Section~\ref{seclevel}, we work over an arbitrary real closed field, which is the natural generality for sums of squares problems.

\subsection{Acknowledgements}

We thank Bruno Kahn for a suggestion that led to an improvement to Theorem~\ref{odddim}, Slava Kharlamov for drawing our attention to Krasnov's article~\cite{Krasnov}, and the referee for their careful work and thoughtful comments.

\subsection{Notation and conventions}
\label{notation}

If~$A$ is an abelian group and~$m$ is an integer, we denote by $A[m]$ the $m$\nobreakdash-torsion subgroup of~$A$
and write $A/m$ for $A/mA$.  We write $A_{\torsion}$ for the torsion subgroup of~$A$, and we write $A/\mathrm{tors}$
for $A/A_{\torsion}$.

A \textit{variety} $X$ over a field $k$ is a separated scheme of finite type over $k$.  If $l/k$ is a field extension, we let $X(l)$ denote the set of $l$-points of $X$.  The variety $X$ is said to be a \textit{curve} (\resp a \textit{surface}, \resp a \textit{threefold}) if it has pure dimension $1$ (\resp $2$, \resp $3$).

Manifolds are Hausdorff and second-countable.
A \textit{space} is a Hausdorff compactly generated topological space (see \cite[Chapter 5]{Mayconcise}). 
By the cohomology of a space with values in an abelian group (or a local system, see \cite[\S 3.H]{Hatcher}), we always mean singular cohomology.  
We only use sheaf cohomology for locally contractible spaces (such as CW complexes), 
in which case it recovers singular cohomology (see \cite[Theorem~1.2]{Petersen}).

Let~$\R$ and $\C$ be the fields of real and complex numbers.
Let $G:=\Gal(\C/\R)\simeq \Z/2$ be generated by the complex conjugation $\sigma\in G$.  A \textit{$G$\nobreakdash-space} (\resp a \textit{$\ci$ $G$-manifold}) is a space (\resp a $\ci$ manifold) endowed with a continuous (\resp $\ci$) $G$-action. 
If $X$ is a variety over $\R$, then $X(\C)$ is naturally a $G$-space with $X(\C)^G=X(\R)$.
The quotient space of the antipodal action of $G$ on the contractible space $EG:=\bS^{\infty}$ is the classifying space  $BG:=EG/G=\P^{\infty}(\R)$ of $G$.

Let $S$ be a $G$-space, and let $T:=(S\times EG)/G$ be the space associated with it by the Borel construction. A $G$-module~$A$ induces a local system on $T$,  which we still denote by $A$.  We consider the $G$-equivariant cohomology groups $H^k_G(S,A):=H^k(T,A)$ in the sense of Borel.  These groups encompass the cohomology groups $H^k(G,A):=H^k_G(\pt, A)=H^k(BG,A)$ of~$G$. 

Let $\Z(j)$ be the $G$-module $\Z$ on which $\sigma$ act by $(-1)^j$ (it only depends on the parity of~$j$).  
If $A$ is a $G$-module, we set $A(j):=A\otimes_{\Z}\Z(j)$.
Let $\omega\in H^1(G,\Z(1))\simeq \Z/2$ be the generator and $\oo\in H^1(G,\Z/2)=\Z/2$ be its reduction modulo $2$. For $e\geq 1$, the $e$th cup powers $\omega^e\in H^e(G,\Z(e))\simeq \Z/2$ and $\oo^e\in H^e(G,\Z/2)\simeq \Z/2$  are the nonzero classes. If~$S$ is a $G$-space, we let $\omega^e_S$ and $\oo^e_S$ denote the pull-backs of
$\omega^e$ and $\oo^e$ in $H^e_G(S,\Z(e))$ and in $H^e_G(S,\Z/2)$ by the second projection $T \to BG$, where~$T$ is as above.
 If $X$ is a variety over~$\R$, we write $\omega^e_X$ and $\oo^e_X$ instead of $\omega^e_{X(\C)}$ and $\oo^e_{X(\C)}$. When no confusion is possible, we still denote these classes by $\omega^e$ and $\oo^e$.

We use $G$\nobreakdash-equivariant sheaf cohomology only for locally contractible $G$-spaces~$S$.  With values in a $G$-module $A$, this recovers Borel $G$-equivariant cohomology,  by the $G$\nobreakdash-equivariant sheaf cohomology isomorphisms $H^k_G(S,A)\isoto H^k_G(S\times EG,A)= H^k((S\times EG)/G,A)$ resulting from the Leray spectral sequence for $S\times EG\to S$ and from the first spectral sequence of \cite[Th\'eor\`eme~5.2.1]{tohoku}.
Let $\sF$ be a $G$-equivariant sheaf on a $G$-space $S$.  The second spectral sequence of \cite[Th\'eor\`eme~5.2.1]{tohoku}, called \textit{the Hochschild--Serre spectral sequence}, reads
\begin{equation}
\label{HS}
E_2^{p,q}=H^p(G,H^q(S,\sF))\Rightarrow H^{p+q}_G(S,\sF).
\end{equation}
Set $\sF(j):=\sF\otimes_\Z\Z(j)$ for $j\in\Z$ and $\sF[G]:=\sF\otimes_\Z\Z[G]$.
The natural short exact sequence $0\to \sF(1)\to \sF[G]\to \sF\to 0$ gives rise to \textit{the real-complex long exact sequence}
\begin{equation}
\label{rc}
\cdots\to H^k_G(S,\sF(1))\to H^k(S,\sF)\xrightarrow{\Nr_{\sF}} H^k_G(S,\sF)\xrightarrow{\omega} H^{k+1}_G(S,\sF(1))\to\cdots,
\end{equation}
where the right arrow is the cup product with the class $\omega\in H^1(G,\Z(1))$ of the extension $0\to \Z(1)\to\Z[G]\to \Z\to 0$,  and where $\Nr_{\sF}$ is called the \textit{norm map} (see \cite[(1.7)]{bw1}).

Given a graded ring $H^*=(H^i)_{i \geq 0}$, we shall freely use expressions of the form $h=\sum_{i \geq 0} h_i$
to denote elements of $\prod_{i\geq 0}H^i$, and we shall refer
to equalities between such expressions as equalities \emph{in the graded ring $H^*$}.

\section{\texorpdfstring{$G$}{G}-equivariant preliminaries}
\label{secprelim}

\subsection{\texorpdfstring{$G$}{G}-equivariant vector bundles}
\label{parGeqvb}

Recall that $G=\Gal(\C/\R)$ and that $\sigma\in G$ denotes complex conjugation.  
A $G$\textit{-equivariant real} (\resp \textit{complex}) \textit{vector bundle} on a $G$\nobreakdash-space~$S$ 
is a real (\resp complex) vector bundle~$E$ on~$S$ endowed with a continuous action of~$G$ on~$E$ such that
the projection map $E\to S$ is $G$\nobreakdash-equivariant and such that for any $s \in S$,
the map $E_s \to E_{\sigma(s)}$ induced by~$\sigma$
is $\R$-linear (\resp $\C$\nobreakdash-antilinear).
The $G$-equivariant complex vector bundles are exactly, under another name, the Real vector bundles introduced by Atiyah in~\cite{Atiyah}.
Our terminology is justified by the fact that we view~$G$ as a group endowed with an action on~$\C$ rather than as an abstract group.
Similarly, we define \textit{$G$\nobreakdash-equivariant~$\ci$ real} (\resp \textit{complex}) \textit{vector bundles} on $\ci$ $G$\nobreakdash-manifolds by requiring that the action of $G$ on~$E$ be~$\ci$.

Let $E$ be a $G$-equivariant real vector bundle on $S$. A \textit{stable complex structure} on $E$ is
an equivalence class of pairs $(k,J)$ consisting of an integer $k\geq 0$ and a structure $J$ of $G$\nobreakdash-equivariant complex vector bundle on the $G$-equivariant real vector bundle $E\oplus (S \times \C^{k})$ (i.e.\ $J$ is an endomorphism of the real vector bundle $E \oplus (S \times \C^k)$
such that $J^2=-1$ and $\sigma J=-J \sigma$), where the equivalence relation
is generated by the condition
that $(k,J)$ and $(k+1,(J,i))$ are equivalent.
If a
$G$\nobreakdash-equivariant real vector bundle~$E$ on~$S$ admits a stable complex structure,
then the rank~$r$ of~$E$ is even, and any stable complex structure on~$E$
determines a $G$\nobreakdash-equivariant isomorphism between the orientation sheaf of~$E$ and $\Z(r/2)$.

Let $\Gr_r(\C^N)$ denote the Grassmannian of complex $r$-planes in $\C^N$ and
set $${\BU(r):=\bigcup_{N\geq 0}\Gr_r(\C^N)}\rlap.$$
The natural action of~$G$ on~$\C^N$ induces actions of~$G$
on~$\Gr_r(\C^N)$ and on~$\BU(r)$, and turns the tautological rank~$r$ complex
vector bundles on these spaces into
$G$\nobreakdash-equivariant complex vector bundles.
Denote by $[S,S']_G$ the set of $G$\nobreakdash-equivariant homotopy classes of $G$-equivariant continuous maps between two $G$-spaces $S$ and $S'$. If $S$ is a paracompact $G$\nobreakdash-space, pulling back
the tautological bundle induces an identification between $[S,\BU(r)]_G$ and the set of isomorphism classes of $G$\nobreakdash-equivariant complex vector bundles of rank~$r$ on~$S$ (if $S$ is compact, see \cite[Proposition II.1]{Edelson}; in general, run the proofs of \cite[Chap.~4,  \S 12]
{Husemoller} $G$\nobreakdash-equivariantly).

Viewing $\BU(r)$ as a subspace of~$\BU(r+1)$ via the map $V \mapsto \C \oplus V$,
we set $$\BU:=\bigcup_{r\geq 0}\BU(r)\rlap.$$
The \textit{$\KR$-theory} of a $G$-space $S$ is by definition $\KR(S):=[S,\BU\times\Z]_G$ endowed with its natural group structure (see \eg \cite[p.~215]{Dugger}).  In view of the above description of~$[S,\BU(r)]_G$, a $G$-equivariant complex vector bundle~$E$ on a paracompact $G$-space $S$ induces a class $[E]\in\KR(S)$,  the factor $\Z$ corresponding to the rank of the bundle.  When~$S$ is moreover compact,  the group $\KR(S)$ can be identified---this is the original point of view taken in \cite{Atiyah}---with the Grothendieck group of the category of $G$-equivariant complex vector bundles on $S$ (the non-equivariant proof in \cite[Chapter II, Theorem 1.33]{karoubibook} can be adapted to the $G$\nobreakdash-equivariant setting).

\subsection{Stably complex \texorpdfstring{$G$}{G}-manifolds}
\label{parstablycomplex}

A \textit{stably complex $\ci$ $G$-manifold of dimension $n$} is  a $\ci$ $G$-manifold~$M$
of dimension~$2n$ equipped with a stable complex structure on the $G$-equivariant real vector bundle~$TM$;
we write $(M,k,J)$ instead of~$M$
when we need to refer to a representative $(k,J)$ of the given stable complex structure on~$TM$.

These are particular cases (minor details put aside) of the $R$-manifolds considered in \cite[\S 2]{Fujiibordism}
or of the Real manifolds of \mbox{\cite[\S2]{Hu}} (which may have dimension $(p,q)$ with $p,q\in\Z$; the ones we consider correspond to the case $p=q=n$). 
When we allow $M$ to have a boundary (\eg in \S\ref{parstablycxconstr}), we always write so explicitly.

The orientation sheaf of a stably complex $\ci$ $G$-manifold $M$ of dimension $n$ is canonically and $G$\nobreakdash-equivariantly isomorphic to~$\Z(n)$.
When $M^G=\varnothing$,  we will use this observation
when applying Poincar\'e duality on the $\ci$ manifold $M/G$;
it implies, for instance, that the group
$H^{2n}_G(M,\Z(n))$ is torsion-free.
We define the stable tangent bundle of
the stably complex~$\ci$ $G$-manifold $(M,k,J)$ to be $\tau_M:=[(TM\oplus(M \times \C^{k}),J)]-[(M \times \C^{k},i)]\in\KR(M)$.

An \textit{almost complex $\ci$ $G$\nobreakdash-manifold of dimension $n$} is
a $\ci$ $G$\nobreakdash-manifold~$M$ of dimension~$2n$ equipped with a structure $J$ of $G$-equivariant complex vector bundle on the $G$\nobreakdash-equivariant real vector bundle $TM$.
An almost complex $\ci$ $G$\nobreakdash-manifold $(M,J)$ gives rise
to the stably complex $\ci$ $G$\nobreakdash-manifold $(M,0,J)$.
The corresponding class~$\tau_M$ is represented by the $G$-equivariant $\ci$ complex vector bundle~$(TM,J)$. Natural examples include $n$\nobreakdash-dimensional complex manifolds endowed with an antiholomorphic involution, and more particularly sets of complex points of smooth varieties of dimension $n$ over $\R$.

\subsection{Construction of stably complex \texorpdfstring{$G$}{G}-manifolds}
\label{parstablycxconstr}

We present, for later use in~\S\ref{parrelations}, in \S\ref{parintrelations} and in \S\ref{parnonalgexamples}, a few techniques to construct stably complex $\ci$ $G$-manifolds.

\vspace{.5em}

Let $(M,k,J)$ be a compact  
stably complex $\ci$ $G$-manifold of dimension $n$ as in~\S\ref{parstablycomplex}. Let $E$ be a $G$\nobreakdash-equivariant $\ci$ complex vector bundle of rank $r\leq n$ on $M$.  
Let $N\subset M$ be the zero locus of a $G$-equivariant section of $E$ that is transverse to the zero section.
One can endow the compact~$\ci$ $G$\nobreakdash-manifold~$N$ with the structure of a stably complex~$\ci$ $G$\nobreakdash-manifold of dimension $n-r$ as follows.  Let~$F$ be a $G$\nobreakdash-equivariant~$\ci$ complex vector bundle of rank $s$ on $M$ that is a stable opposite of $E$, \ie such that there exists an isomorphism of $G$\nobreakdash-equivariant complex vector bundles $E\oplus F\simeq M \times \C^{r+s}$ (use \cite[Lemma II.2]{Edelson}).
The $G$\nobreakdash-equivariant real vector bundle
$TN\oplus (N \times \C^{r+s+k})$ can then be identified with $TN\oplus E|_N\oplus F|_N\oplus (N \times \C^{k})\simeq TM|_N\oplus(N \times \C^{k})\oplus F|_N$, which has a natural structure of $G$-equivariant complex vector bundle.

\vspace{.5em}

Let $(M,k,J)$ be a compact stably complex $\ci$ $G$-manifold of dimension $n$ with boundary~$\partial M$.
We define the \textit{opposite} $-M$ of $M$ to be the stably complex $\ci$ $G$\nobreakdash-manifold $(M,k+1,(J,-i))$. We now explain how to construct the \textit{double} $S$ of $M$ by appropriately gluing $M$ and $-M$ together along their boundary.

The $G$-equivariant real vector bundle $(TM\oplus(M \times \C^{k}))|_{\partial M}$ on $\partial M$ splits as a direct sum $E\oplus L\oplus L'$,  where $E\subset (TM\oplus (M \times \C^{k}))|_{\partial M}$ is the largest real subbundle
 of $T(\partial M)\oplus(\partial M \times \C^{k})$ that is stable under~$J$, where $L'$ is a supplement of $E$ in $T(\partial M)\oplus(\partial M \times \C^{k})$ that is stable under~$G$, and where $L:=J(L')$ (so that $L'=J(L)$).
We see that~$L$ can be identified with the normal bundle of $\partial M$ in~$M$.  As a consequence,
it is a trivial $G$-equivariant real line bundle on~$\partial M$.

Consider the following two complex structures on the $G$-equivariant real vector bundle $E\oplus L\oplus L'\oplus (\partial M \times \C)$ on $\partial M$. The first one is~$(J,i)$.  The second one is the conjugate of $(J,-i)$ by $(\Id,-\Id,\Id,\Id)$.
As $L \oplus L'$ is isomorphic, as a $G$\nobreakdash-equivariant
real line bundle on~$\partial M$, to $\partial M \times \C$,
Lemma~\ref{homotopyJ} below implies the existence
of a homotopy (constant on the factor~$E$) between these two complex structures such that, keeping the $G$-action constant, the bundle $E\oplus L\oplus L'\oplus (\partial M \times \C)$ remains a $G$-equivariant complex vector bundle along the whole homotopy.

Define $S$ to be the compact $\ci$ $G$-manifold obtained by gluing $M$ and $-M$ along their common boundary $\partial M$ by inserting between them a cylinder $\partial M\times [0,1]$ (to perform the gluing,  make use of $G$-equivariant collars).
 To endow $TS$ with a stable complex structure $(k+1,J_S)$, we use the complex structure $(J,i)$ on the piece $M$, the complex structure~$(J,-i)$ on the piece $-M$, and the above homotopy on the cylinder $\partial M\times [0,1]$ joining them. These complex structures do glue because $(\Id,-\Id,\Id,\Id)$ is induced by the differential of the gluing map. It only remains to smooth the resulting complex structure.

\begin{lem}
\label{homotopyJ}
There is a continuous family $(J_t)_{t\in [0,1]}$ of $\R$-linear automorphisms of $\C^2$~with 
\begin{enumerate}[(i)]
\item $J_0=(i,i)$ and  $J_1=(-i,-i)$;
\item $J_t^2=-\Id$ and $J_t(\bar z_1,\bar z_2)=-\overline{J_t(z_1,z_2)}$ for all $t\in[0,1]$ and $(z_1,z_2)\in\C^2$.
\end{enumerate}
\end{lem}

\begin{proof}
In the $\R$-basis $((1,0), (0,1),(i,0),(0,i))$ of $\C^2$, choose $J_t:=
\begin{pmatrix} 
0 & -(K_t)^{-1} \\
K_t& 0
\end{pmatrix},
$ where $(K_t)_{t\in [0,1]}$ is a path joining $\Id$ to $-\Id$ in $\GL_2(\R)$.
\end{proof}

\subsection{The Steenrod algebra}
\label{parSteenalg}

Let $\cA$ be the mod $2$ Steenrod algebra (see~\cite[Chapter~I, Chapter~II]{Steenrod}).
It is a graded $\Z/2$-algebra generated by the Steenrod squares $(\Sq^i)_{i\geq 1}$ subject to the Adem relations.  We let $\Sq^0\in\cA$ be the unit, and we set $\Sq:=\sum_{i\geq 0}\Sq^i$.
 The algebra~$\cA$ functorially acts on the mod $2$ relative cohomology of all pairs of topological spaces. 
The element $\Sq^1$ acts as the Bockstein. For any $i\geq 0$, the element $\Sq^i$ acts as the squaring map in degree $i$, and as zero in degree $> i$.
 The coproduct $\psi\colon\cA\to\cA\otimes_{\Z/2} \cA$ given by $\psi(\Sq^i)=\sum_{j+k=i}\Sq^j\otimes\Sq^k$ reflects the Cartan formula 
\begin{equation}
\label{Cartan}
\Sq^i(x\smile y)=\sum_{j+k=i}\Sq^j(x)\smile\Sq^k(y) 
\end{equation}
describing the compatibility of Steenrod squares with the cup product.
The antipode $\chi\colon\cA\to\cA$,  characterized by $\chi(\Sq)\Sq=\Sq\chi(\Sq)=1$, is an involutive anti-automorphism of the graded algebra $\cA$. 
These maps turn $\cA$ into a cocommutative Hopf algebra.

As $\Sq(\oo^j)=\Sq(\oo)^j=(\oo+\oo^2)^j$,  the $\cA$-action on $H^*(G,\Z/2)=\Z/2[\oo]$ is given by 
\begin{equation}
\label{Aomega}
\Sq^i(\oo^j)=\binom{j}{i}\,\oo^{i+j}.
\end{equation}  
In addition,  one computes that
\begin{equation}
\label{chiSqomega}
\Sq\bigg(\sum_{l\geq 0}\oo^{2^l}\bigg)=\sum_{l\geq 0}\Sq\big(\oo^{2^l}\big)=\sum_{l\geq 0}\big(\oo^{2^l}+\oo^{2^{l+1}}\big)=\oo
\textrm{ \hspace{.2em}and hence\hspace{.2em} }
\chi(\Sq)(\oo)=\sum_{l\geq 0}\oo^{2^l}.
\end{equation}

Define $\dA:=H^*(G,\Z/2)\otimes_{\Z/2} \cA$ endowed with the twisted product
$$(b\otimes a)(b'\otimes a'):=\sum_i (b \smile a^{(1)}_i(b'))\otimes a^{(2)}_ia',$$
where $\psi(a)=\sum_i a^{(1)}_i\otimes a^{(2)}_i$. 
 Using \cite[Theorem 2.13]{Molnar} and the Hopf algebra structure on $H^*(G,\Z/2)$
(see \cite[Theorem~4.41]{mccleary}), one can give $\dA$ a natural structure of cocommutative Hopf algebra,  with antipode $\chi\colon\dA\to\dA$ satisfying $\chi(\oo)=\oo$.  

Let $S$ be a $G$-space. Letting $\cA$ act on $H^*_G(S,\Z/2)$ by means of the Borel construction
and letting $H^*(G,\Z/2)$ act on $H^*_G(S,\Z/2)$  in the natural way
induces an action of the algebra~$\dA$ on $H^*_G(S,\Z/2)$. (To see this, use (\ref{Cartan}).)
 In fact, Greenlees showed in~\cite[Theorem~2.7]{Greenlees} that~$\dA$ can be identified with the algebra of stable operations in mod $2$ Borel equivariant cohomology.

\subsection{Characteristic classes}
\label{parcharclasses}

A $G$-equivariant real vector bundle $E$ on a $G$\nobreakdash-space $S$ induces a real vector bundle $F$ on the space $T:=(S\times EG)/G$ associated with it by the Borel construction.
  We define the \textit{$G$\nobreakdash-equivariant Stiefel--Whitney classes} $w_{G,i}(E)\in H^i_G(S,\Z/2)=H^i(T,\Z/2)$ of $E$ to be the usual Stiefel\nobreakdash--Whitney classes of~$F$ (see \cite{milnorstasheff}). The $G$-equivariant Stiefel--Whitney classes of a $G$\nobreakdash-equivariant complex vector bundle are by definition those of the underlying $G$-equivariant real vector bundle. 
As usual, we set $w_G(E):=\sum_{i\geq 0}w_{G,i}(E)$.

The \textit{Wu classes} $u_i(E)\in H^i(S,\Z/2)$ of a real vector bundle $E$ on a space $S$ are defined by the equation $u(E)=\sum_{i\geq 0} u_i(E):=\chi(\Sq)(w(E))$. The $G$-equivariant Wu classes $u_{G,i}(E)\in H^i_G(S,\Z/2)$ of a $G$\nobreakdash-equivariant real vector bundle $E$ on a $G$\nobreakdash-space~$S$ are defined by the analogous formula $u_G(E)=\sum_{i\geq 0}u_{G,i}(E):=\chi(\Sq)(w_G(E))$.

Stiefel--Whitney classes and Wu classes,
and their $G$-equivariant versions,
can be associated not only with vector bundles but more generally with classes in appropriate $K$\nobreakdash-theory groups. (Use \cite[Lemma 10.3]{milnorstasheff} to reduce to the case of compact spaces, represent $K$\nobreakdash-theory classes by differences of bundles, and use the Whitney formula.)

Let $E$ be a $G$-equivariant complex vector bundle on a $G$\nobreakdash-space $S$.
Kahn has defined functorial \textit{$G$\nobreakdash-equivariant Chern classes} $c_i(E)\in H^{2i}_G(S,\Z(i))$.
They satisfy a Whitney formula, refine the non-equivariant Chern classes and are characterized by
similar axioms (see \cite[Th\'eor\`eme~2, Proposition~2]{kahnchern}; see also \cite[\textsection4]{pitschscherer}).
 We let $c(E)=\sum_{i\geq 0} c_i(E)$ and  $\oc(E)=\sum_{i\geq 0} \oc_i(E)$ denote the total equivariant Chern class  and its reduction mod~$2$.
The $G$-equivariant Chern classes are stable in the sense that $c(E\oplus (S \times \C))=c(E)$,  by the Whitney formula and the vanishing of $H^2(G,\Z(1))$. 

Kahn computed in \cite[Th\'eor\`eme 3]{kahnchern} that 
\begin{equation}
\label{cohoBUn}
\bigoplus_{j\in\Z/2,\, k\geq 0}H^k_G(\BU(r),\Z(j))=\Z[\omega, c_1,\dots, c_r]/(2\omega)\rlap,
\end{equation}
where the $c_i$ are the $G$-equivariant Chern classes of the tautological bundle.
This is an isomorphism of $(\Z/2 \times \Z)$\nobreakdash-graded rings.  The ring structure of the left-hand side
is induced by the ring structure
on $\Z\oplus\Z(1)$ for which the square of a generator of $\Z(1)$ is~$1\in\Z$.
In the right-hand side, the grading is defined by $\deg(\omega)=(1, 1)$ and $\deg(c_i)=(i,2i)$ for all~$i \geq 1$.
It follows from (\ref{cohoBUn}) (using \cite[Lemma 2]{Milnorlim} and noting that the derived inverse limit there vanishes in our situation as the Mittag--Leffler condition is satisfied) that
\begin{equation}
\label{cohoBU}
\bigoplus_{j\in\Z/2,\, k\geq 0}H^k_G(\BU,\Z(j))=\Z[\omega, (c_i)_{i\geq 1}]/(2\omega)\rlap.
\end{equation}
Using the short exact sequence
$0 \to \Z \to \Z \to \Z/2 \to 0$, it can be deduced
that
\begin{equation}
\label{cohoBU2}
H^*_G(\BU,\Z/2)=\Z/2[\oo, (\oc_i)_{i\geq 1}]\rlap.
\end{equation}

Formula (\ref{cohoBU}) allows us to define $G$-equivariant Chern classes for arbitrary classes in $\KR$\nobreakdash-theory.
One can therefore speak of the Chern classes $c_i(M):=c_i(\tau_M)\in H^{2i}_G(M,\Z(i))$ and $\oc_i(M):=\oc_{i}(\tau_M)\in H^{2i}_G(M,\Z/2)$ of a stably complex $\ci$ $G$-manifold~$M$.  When no confusion is possible, we write~$c_i$ and $\oc_i$ instead of $c_i(M)$ and $\oc_i(M)$. 

\begin{prop}
\label{propequivw}
Let $S$ be a $G$-space. For $\kappa\in\KR(S)$, one has
\begin{align}
\label{cw}
w_G(\kappa)=\sum_{i\geq 0}\oc_i(\kappa)(1+\oo)^{\rk(\kappa)-i}
\end{align}
in the graded ring $H^*_G(S,\Z/2)$.
\end{prop}

\begin{proof}
It is a consequence of \cite[Lemma 10.3]{milnorstasheff} that $\varprojlim_{C\subset S}H_G^*(C,\Z/2)=H_G^*(S,\Z/2)$, where $C$ runs over all
 compact subsets of $S$ that are stable under~$G$. One can therefore assume that~$S$ is compact, and hence that $\kappa$ is represented by a difference of $G$-equivariant complex vector bundles. 
By the splitting principle for $G$-equivariant complex vector bundles (see \cite[Th\'eor\`eme 1]{kahnchern}) and the Whitney formula,
noting that the right-hand side of~\eqref{cw} is additive in~$\kappa$,
  we may assume that $\kappa$ is the class of a line bundle~$L$.
Equation~(\ref{cw}) then reduces to the two identities $w_{G,1}(L)=\oo$ and $w_{G,2}(L)=\oc_1(L)$.  

To prove the first one,  let $F$ be the real vector bundle of rank $2$ on $T:=(S\times EG)/G$ associated with~$L$. Its pull-back to $S\times EG$ is orientable (as is any complex line bundle),  and the action of the deck transformation $\sigma\in G$ reverses its orientation. We deduce that the double cover associated with $w_1(F)\in H^1(T,\Z/2)$ is exactly $S\times EG\to T$,  whose class is $\oo\in H^1(T,\Z/2)$. It follows that $w_{G,1}(L)=\oo$. 

To prove the second one, we note that $\oc_1(L)$ equals the $G$-equivariant Euler class mod~$2$ of $L$ by \cite[Th\'eor\`eme 5]{kahnchern}, and that the latter equals $w_{G,2}(L)$ by \cite[Property 9.5]{milnorstasheff} (applied to the vector bundle $F$ on $T$).
\end{proof}

\begin{rmks}
\label{rk:SetS/G}
Let~$S$ be a $G$\nobreakdash-space with $S^G=\emptyset$ and $\pi\colon S \to S/G$ be the quotient map.
Assume that~$G$ acts properly discontinuously on~$S$ (i.e.\ that~$\pi$ is a covering space).

(i)
The categories of $G$\nobreakdash-equivariant real vector bundles on~$S$ and
 real vector bundles on~$S/G$ are equivalent, via the functors
$E \mapsto (\pi_*E)^G$ and $F \mapsto \pi^*F$.
The $G$\nobreakdash-equivariant
Stiefel--Whitney classes of a $G$\nobreakdash-equivariant real vector bundle on~$S$ are nothing
but the non-equivariant
Stiefel--Whitney classes of the corresponding real vector bundle on~$S/G$,
through the natural isomorphisms $H^i_G(S,\Z/2)=H^i(S/G,\Z/2)$.

(ii)
In contrast, the complex structure on a $G$\nobreakdash-equivariant complex vector bundle on~$S$
does not descend to a complex structure on the corresponding real vector bundle on~$S/G$.
For this reason $G$\nobreakdash-equivariant Chern classes cannot simply be defined in terms of
non-equivariant Chern classes in the cohomology of~$S/G$.
\end{rmks}

\subsection{Right actions of the Steenrod algebra}
\label{parrightaction}

One must credit Adams \cite{Adams} for first introducing right $\cA$-actions to study characteristic classes. The following general construction is due to Brown and Peterson \cite[\S 6]{BP}.  Let $F$ be a rank $r$ real vector bundle on a space $T$.  Let $F^*\subset F$ be the complement of the zero section. The cup product with the Thom class of $F$ (an element of $H^r(F,F^*,\Z/2)$) gives rise to the Thom isomorphism $\phi_F\colon H^*(T,\Z/2)\isoto H^{*+r}(F,F^*,\Z/2)$ (see \cite[Theorem~10.2]{milnorstasheff}).
For $x\in H^k(T,\Z/2)$ and~$a\in\cA$ of degree $l$, we define $(x)a\in H^{k+l}(T,\Z/2)$ by the formula
\begin{equation}
\label{rightaction}
\phi_F((x)a)=\chi(a)(\phi_F(x)).
\end{equation}
This right action of $\cA$ on $H^*(T,\Z/2)$ depends on $F$.
We stress that it does not commute, in general, with the left action of~$\cA$
on $H^*(T,\Z/2)$.

From (\ref{rightaction}) and from Thom's definition of the Stiefel--Whitney classes \cite[p.~91]{milnorstasheff}, it follows at once that 
\begin{equation}
\label{rightaction1}
(x)\chi(\Sq)=w(F)\smile\Sq(x)
\end{equation}
and that
\begin{equation}
\label{rightaction2}
(x)\Sq=u(-[F])\smile\chi(\Sq)(x).
\end{equation}
In addition,  a computation based on (\ref{rightaction2}) and on the Cartan formula (\ref{Cartan}) shows that
\begin{equation}
\label{Cartanr}
(x)\Sq^i\smile y=\sum_{j+k=i}(x\smile \Sq^j(y))\Sq^k
\end{equation}
for all $x,y\in H^*(T,\Z/2)$.

The vector bundles~$F$ and $F\oplus (T\times \R)$ induce the same right action, by multiplicativity of Thom classes 
and since $\cA$ acts trivially (on the left) on the Thom class of a trivial bundle.  
Define $\KO'(T):=\varprojlim_{C\subset T}\KO(C)$, where $\KO$ denotes real $K$-theory
(see~\cite[Chapter II, Example~1.11]{karoubibook})
and $C$ runs over all compact subsets of $T$.
Fix $\kappa\in\KO'(T)$.
The restriction of~$\kappa$ to any compact subset $C\subset T$
is stably represented by a vector bundle (argue as in \cite[Corollary 1.4.14]{AtiyahKTheory}), which gives rise to a well-defined right action of $\cA$ on $H^*(C,\Z/2)$. We therefore obtain a right action of $\cA$ on the graded ring $\varprojlim_{C\subset T}H^*(C,\Z/2)=H^*(T,\Z/2)$ (see \cite[Lemma 10.3]{milnorstasheff}), depending (only) on~$\kappa$.

Now, let $S$ be a $G$-space,  and set $T:=(S\times EG)/G$. Fix an element $\kappa\in\KO'(T)$.
The above construction yields a right action of $\cA$ on $H^*(T,\Z/2)=H^*_G(S,\Z/2)$.  One can promote this action to a right action of $\dA$ on $H^*_G(S,\Z/2)$ by using the exact same formula~(\ref{rightaction}) (after having stably represented $\kappa$ by the class of a vector bundle on a compact subset of $T$). Since $\chi(\oo)=\oo$, the induced action of $H^*(G,\Z/2)$ on $H^*_G(S,\Z/2)$ is the natural one. The resulting action of $\dA$ depends on the choice of $\kappa$.

\begin{prop}
\label{droitegauche}
Let $M$ be a compact $\ci$ $G$-manifold with $M^G=\varnothing$.  Let $\dA$ act on  $H^*_G(M,\Z/2)$ on the right as above,  by means of the element $\kappa\in \KO((M\times EG)/G)$ induced by $-[TM]$.
For all $x,y\in H^*_G(M,\Z/2)$ and $a\in\dA$ homogeneous of respective degrees $\deg(x)$, $\deg(y)$, $\deg(a)$
with $\deg(x)+\deg(y)+\deg(a)=\dim(M)$, one has
\begin{equation}
\label{Wuidentity}
x\smile a(y)=(x)a\smile y.
\end{equation}
\end{prop}

\begin{proof}
It suffices to verify the validity of (\ref{Wuidentity}) for generators of $\dA$. When $a=\oo$, the assertion is clear. When $a\in\cA$, it is \cite[Corollary 6.3]{BP} applied to the compact~$\ci$ manifold $M/G$, in view of the identification $H^*_G(M,\Z/2)=H^*(M/G,\Z/2)$.
\end{proof}

\section{Relations between tautological classes}
\label{secrelations}

In this section, we construct many relations between powers of $\oo$ (or $\omega$) and Chern classes, in the spirit of \cite{BP,BP2}. We work in the topological setting of stably complex $G$\nobreakdash-manifolds, which is the natural generality in which our relations are valid. We prove, in particular, Theorems~\ref{introth:relmod2} and~\ref{introth:relint} stated in the introduction.

\subsection{Mod \texorpdfstring{$2$}{2} relations}
\label{parrelations}

Fix $n\geq 0$.  Let $\Theta_n^{\top}$ (\resp $\Theta_n^{\ac}$, \resp $\Theta_n^{\an}$, \resp $\Theta_n^{\alg}$) denote the collection of all stably complex~$\ci$ $G$-manifolds of dimension $n$ with no $G$-fixed points (\resp of all almost complex~$\ci$ $G$-manifolds of dimension~$n$ with no $G$\nobreakdash-fixed points, \resp of all $n$\nobreakdash-dimensional complex manifolds endowed with a fixed-point free antiholomorphic involution, \resp of all sets of complex points of smooth algebraic varieties of dimension~$n$ over $\R$ with no real points).
Fix $?\in\{\top, \ac,\an,\alg\}$.  The evaluation of $P\in \Z/2[\oo,(\oc_i)_{i\geq 1}]$ at $M\in \Theta_n^?$ is
by definition the cohomology class
  $P(\oo,\oc(M))\in H^*_G(M,\Z/2)$,
where
$P(\oo,\oc(M))$ stands for
 $P(\oo,\oc_1(M),\oc_2(M),\dots)=P(\oo,\oc_1(\tau_M),\oc_2(\tau_M),\dots)$
(see \textsection\ref{parcharclasses}).
We define $J_n^{?}\subset\Z/2[\oo,(\oc_i)_{i\geq 1}]$ to be the ideal of those polynomials that vanish on all $M\in\Theta_n^?$.

\begin{quest}
\label{qtopalg}
Are the inclusions $J_n^{\top}\subset J_n^{\ac}\subset J_n^{\an}\subset J_n^{\alg}$ equalities?
\end{quest}

 In other words, are all relations that hold in the algebraic setting also valid in the analytic setting, and furthermore also in the almost complex or stably complex settings? We shall prove in Theorem \ref{thtopac} that $J_n^{\top}=J_n^{\ac}$, but the other equalities are open. We do not even know if the powers of $\oo$ that belong to $J_n^{\alg}$ also belong to $J_n^{\top}$ (see Remark \ref{remexemplessc} (iii) below).

We now explain how to produce elements of $J_n^{\top}$.  Over each $G$-invariant compact subset of $\BU\times\Z\times EG$, the universal $\KR$-theory element in $\KR(\BU\times \Z\times EG)$ can be represented as a difference of $G$-equivariant complex vector bundles.
 Forgetting the complex structure and quotienting by the action of $G$ yields an element $\kappa\in \KO'((\BU\times\Z\times EG)/G)$. 
Let~$\kappa_{n}\in \KO'((\BU\times EG)/G)$ be its restriction to the connected component corresponding to~$n\in\Z$.  We let $\dA$ act on the right on $H^*_G(\BU,\Z/2)$ as in \S\ref{parrightaction}, by means of $-\kappa_{n}$. We insist that this right action depends on $n$.

Let $R^l_k\subset\cA$ be the set of degree $l$ elements of the Steenrod algebra which, when applied to the universal degree $k$ class in $H^k(K(\Z/2,k),\Z/2)$, give rise to a class which lifts to $H^{k+l}(K(\Z/2,k),\Z)_{\torsion}$.  A complete description of the $R^l_k$ appears in \cite[Theorem~4.4]{BP}.  In addition, we identify $H^{i}_G(\BU,\Z(n))/2$ with the image of the reduction mod $2$ map $H^{i}_G(\BU,\Z(n))\to H^{i}_G(\BU,\Z/2)$. We then define
\begin{equation}
\label{defKn}
K_n=\sum_{k<l}(H^{2n-l-k}_G(\BU,\Z/2))\Sq^l+\sum_{k,l}(H^{2n-l-k}_G(\BU,\Z(n))/2)R_k^l,
\end{equation}
which we view as a subgroup of $H^*_G(\BU,\Z/2)=\Z/2[\oo, (\oc_i)_{i\geq 1}]$ (see (\ref{cohoBU2})).

The first summand
in~\eqref{defKn} (resp.\ the group~$K_n$)
is the stably complex $G$\nobreakdash-equivariant analogue
of the relations exhibited by Brown and Peterson in \cite[Theorem~3.5]{BP}
(resp.\ in \cite[Theorem~4.3]{BP}) between the Stiefel--Whitney
classes of compact $\ci$ manifolds (resp.\ of compact oriented $\ci$ manifolds)
of dimension~$2n$ (see also Remark~\ref{rk:4.1versusBP}).

\begin{thm}
\label{reltopo}
One has $K_n\subset J_n^{\top}$.
\end{thm}

\begin{proof}
Let $M$ be a stably complex $\ci$ $G$-manifold of dimension~$n$ with $M^G=\varnothing$.  Let \mbox{$\tau\colon M\to \BU\times \Z$} be a $G$-equivariant continuous map classifying the stable tangent bundle $\tau_M\in\KR(M)$.
  Let $\dA$ act on the right on $H^*_G(M,\Z/2)$ by means of the class $-\tau^*\kappa\in\KO'((M\times EG)/G)$ induced by $-[TM]$.

\begin{Case} 
\label{Case1mod2}
$M$ is compact.
\end{Case}

Fix $x\in H^{2n-l-k}_G(\BU,\Z/2)$. Then
for all $z\in H^k_G(M,\Z/2)$,
 Proposition~\ref{droitegauche} shows that
$$\tau^*((x)\Sq^l)\smile z=(\tau^*x)\Sq^l\smile z=\tau^*x\smile\Sq^l(z)\rlap.$$
If $l>k$, then $\Sq^l(z)=0$ and hence
$\tau^*((x)\Sq^l)\smile z=0$ for all
 $z\in H^k_G(M,\Z/2)$.  By Poincar\'e duality on $M/G$, it follows that $\tau^*((x)\Sq^l)=0$, i.e.\ $(x)\Sq^l$ vanishes on $M$.

Now choose $y\in H^{2n-l-k}_G(\BU,\Z(n))$ and let $\bar y\in H^{2n-l-k}_G(\BU,\Z/2)$ be its reduction mod~$2$.  For $a\in R^l_k$ and $z\in  H^k_G(M,\Z/2)$, Proposition~\ref{droitegauche} shows that
\begin{equation}
\label{equalideg}
\tau^*((\bar y)a)\smile z=(\tau^*\bar y)a\smile z=\tau^*\bar y\smile a(z).
\end{equation}
By the choice of $a$, there exists $t\in H^{k+l}_G(M,\Z)_{\torsion}$ lifting $a(z)$.  As $H^{2n}_G(M,\Z(n))$ is torsion-free (see \S\ref{parstablycomplex}),  the class $\tau^*y\smile t$ vanishes. So does $\tau^*\bar y\smile a(z)$, which is its reduction mod~$2$.  As this holds for all~$z$, we deduce, thanks to~(\ref{equalideg}) and to Poincar\'e duality on~$M/G$, that~$\tau^*((\bar y)a)=0$, i.e.\ that $(\bar y)a$ vanishes on $M$.

\begin{Case} 
\label{Case2mod2}
$M$ is not compact.
\end{Case}

Fix $\alpha\in K_n$. Choose a $G$\nobreakdash-invariant proper $\ci$ map $f\colon M\to \R$.
  By Sard's theorem,  one can exhaust~$M$ by sublevel sets $(M_i)_{i\geq 1}$ of $f$ that are compact stably complex~$\ci$ $G$\nobreakdash-manifolds with boundary.  Let~$S_i$ be the double of $M_i$ (as in~\S\ref{parstablycxconstr}). The relation $\alpha=0$ holds on the $S_i$ by the compact case, hence on the $M_i$ by restriction, and hence on $M$ because $H^*_G(M,\Z/2)=\varprojlim_i H^*_G(M_i,\Z/2)$ (the Mittag--Leffler condition is satisfied because the $H^*_G(M_i,\Z/2)$ are finite).
\end{proof}

\begin{rmk}[comparison with the Brown--Peterson relations]
\label{rk:4.1versusBP}
If~$M$ is a compact stably complex~$\ci$ $G$\nobreakdash-manifold of dimension~$n$ with $M^G=\emptyset$,
any polynomial relation between the Stiefel--Whitney classes of the compact~$\ci$ manifold $M/G$ of
dimension~$2n$ induces,
by Proposition~\ref{propequivw} and
Remark~\ref{rk:SetS/G}~(i),
a polynomial relation between~$\oo$ and the reductions modulo~$2$ of the
$G$\nobreakdash-equivariant Chern classes of~$M$.
In this way,  the
relations universally satisfied by the
Stiefel--Whitney classes of compact~$\ci$ manifolds of dimension~$2n$
(resp.\ of oriented compact~$\ci$ manifolds of dimension~$2n$, when~$n$ is even),
described in \cite[Theorem~3.5]{BP} (resp.\ in \cite[Theorem~4.3]{BP}),
give rise to an ideal $K_n^{\mathrm{BP}} \subset\Z/2[\oo,(\oc_i)_{i\geq 1}]$
(resp.\ an ideal $K_n^{\mathrm{BPo}} \subset\Z/2[\oo,(\oc_i)_{i\geq 1}]$,
noting that~$M/G$ is orientable if~$n$ is even).
From the definition of~$K_n$ and the explicit descriptions of $K_n^{\mathrm{BP}}$
and $K_n^{\mathrm{BPo}}$,
it is apparent that $K_n^{\mathrm{BP}} \subset K_n$
and, if $n$ is even, that $K_n^{\mathrm{BP}} \subset K_n^{\mathrm{BPo}} \subset K_n$.
We point out, however, that these inclusions are strict.
For example, one computes that $\oo^3 \in K_2$ but that $\oo^3 \notin K_2^{\mathrm{BP}}$
and in fact $\oo^3\notin K_2^{\mathrm{BPo}}$.
In particular, the vanishing of~$\oo^3$
when $n=2$
does not follow from the well-known fact that the Wu classes~$u_i$ of~$M/G$
vanish for $i>n$.
\end{rmk}

We do not know whether the relations that we have just constructed are the only ones.

\begin{quest}
\label{qKJ}
Is the inclusion $K_n\subset J_n^{\top}$ an equality?
\end{quest}

\begin{rmks}
\label{remrelmod2}
(i)
Positive answers to non-equivariant versions of Questions \ref{qtopalg} and~\ref{qKJ} were obtained by Brown and Peterson in \cite[Theorems 1.2 and 1.3]{BP2}.

(ii)
We do not even know if $K_n$ is an ideal of $\Z/2[\oo, (\oc_i)_{i\geq 1}]$. 
Whether the ideal of $\Z/2[\oo, (\oc_i)_{i\geq 1}]$ generated by $K_n$ is equal to $J_n^{\top}$ is a natural weakening of Question~\ref{qKJ}.

(iii)
Let us briefly describe an even more general way of producing elements of $J_n^{\top}$.
Fix $0\leq k\leq 2n$. Let $G$ act trivially on the Eilenberg--MacLane space $K(\Z/2,k)$, and denote by $z\in H^k(K(\Z/2,k),\Z/2)$ the tautological class. Consider the cohomology group 
$$B_k:=H^*_G(\BU\times K(\Z/2,k),\Z/2)=H^*_G(\BU,\Z/2)\otimes_{\Z/2}H^*(K(\Z/2,k),\Z/2),$$
and let $\dA$ act on $B_k$ on the right as in \S\ref{parrightaction}, by means of $pr_1^*(-\kappa_{n})$. 
Let $C_k\subset B_k$ be the subgroup generated by all the elements of degree $2n$ of $B_k$ of the form $(b)a\smile b'-b\smile a(b')$, where $a\in\dA$ and $b,b'\in B_k$ are homogeneous.  Finally, define $K'_n\subset H^{*}_G(\BU,\Z/2)$ to be the graded subgroup whose degree $2n-k$ component is 
$\{c\in H^{2n-k}_G(\BU,\Z/2)\mid c\smile z\in C_k\}$.

One can verify that $K'_n\subset J_n^{\top}$ (reduce to the case of compact manifolds as in Case \ref{Case2mod2} of the proof of Theorem \ref{reltopo}, and apply Proposition \ref{droitegauche} and Poincar\'e duality as in Case \ref{Case1mod2} of the same proof), that $K'_n\subset H^{*}_G(\BU,\Z/2)$ is an ideal,  and that this procedure recovers all the relations of Theorem~\ref{reltopo}, in the sense that $K_n\subset K'_n$. 
(We will not use these facts.)

The question whether $K'_n=J_n^{\top}$ is an even weaker version of Question \ref{qKJ} than the one considered in (ii). We do not know if this mechanism leads to more relations than those in~$K_n$, \ie if the inclusion $K_n\subset K'_n$ is strict. However, its principle will be put to use in the proof of Proposition \ref{o8QZ} to construct an interesting relation with $\Q/\Z$ coefficients.
\end{rmks}

\subsection{Poincar\'e duality}
\label{parPoincare}

To construct integral relations between $\omega$ and the $c_i$ in \S\ref{parintrelations}, 
we rely on Poincar\'e duality, as we did in the proof of Theorem~\ref{reltopo}.
As we could not locate, in the existing literature,
the exact version of Poincar\'e duality that we need (for local systems of finitely generated
abelian groups on
possibly non-orientable manifolds with boundary), we derive it below.
In its statement, we use $\widetilde{A}$ as a shorthand for $A \otimes_{\Z}\widetilde{\Z}$
and write $H^k_c$
to denote cohomology with compact support.

\begin{prop}
\label{Poincareduality}
Let $M$ be a compact topological manifold of dimension~$n$ with boundary~$\partial M$.
Set $\mathring M = M \setminus \partial M$.
Let~$\widetilde \Z$ be the orientation sheaf of~$\mathring M$.
Let~$\L$ be a locally constant sheaf of abelian groups on~$M$, with finitely generated stalks.
For any abelian group~$A$,
there is a canonical ``trace'' homomorphism
$H^n_c(\mathring M,\widetilde A) \to A$.
\begin{enumerate}
\item Together with cup product, the trace homomorphism induces, for any $k \in \Z$,
a pairing
\begin{align*}
H^k(M,\L)\times H^{n-k}_c(\mathring M,\Homrond(\L|_{\mathring M},\widetilde{\Q/\Z}))\to \Q/\Z
\end{align*}
which is non-degenerate on both sides and identifies 
$H^{n-k}_c(\mathring M,\Homrond(\L|_{\mathring M},\widetilde{\Q/\Z}))$
with the Pontrjagin dual
of the finitely generated abelian group $H^k(M,\L)$.
\item Assume that~$\L$ has torsion-free stalks.
Together with cup product, the trace homomorphism induces, for any $k \in \Z$,
a perfect (unimodular) pairing
\begin{align*}
H^k(M,\L)/\mathrm{tors} \times H^{n-k}_c(\mathring M,\Homrond(\L|_{\mathring M},\widetilde \Z))/\mathrm{tors}\to \Z\rlap.
\end{align*}
It identifies each of the 
two
 finitely generated torsion-free abelian groups
$H^k(M,\L)/\mathrm{tors}$
and
$H^{n-k}_c(\mathring M,\Homrond(\L|_{\mathring M},\widetilde \Z))/\mathrm{tors}$
with the $\Z$\nobreakdash-linear dual of the other one.
\end{enumerate}
\end{prop}

\begin{proof}
Let $D^+(V)$ denote the bounded below derived category of sheaves of abelian groups on a space~$V$.
We recall that Verdier duality furnishes
a functor $f^!\colon D^+(\pt) \to D^+(M)$ and,
for any sheaf of abelian groups~$\sF$ on~$M$
and any abelian group~$A$,
an isomorphism
\begin{align}
\label{eq:verdier}
\Hom_{D^+(\pt)}(R\Gamma(M,\sF)[k],A)=H^{-k}(M,\RHom(\sF,f^!A))
\end{align}
for any $k\in\Z$
(see \cite[Theorem~3.1.5, Proposition~3.1.10]{kashiwaraschapira}).
Let $i\colon\partial M \hookrightarrow M$ and $j\colon\mathring M \hookrightarrow M$ denote the inclusions.
One has $j^*f^!A=\widetilde A[n]$
(see \cite[Proposition~3.3.2, Proposition~3.3.6]{kashiwaraschapira})
and $i^*f^!A=0$
(see \cite[Proposition~3.1.12]{kashiwaraschapira}, to be applied locally),
In view of this and of the isomorphism~\eqref{eq:verdier} for $\sF=\Z$ and $k=0$,
the natural morphism $\Z \to R\Gamma(M,\Z)$ induces the desired trace map
$H^n_c(\mathring M,\widetilde A) \to A$.

The finiteness assertions of Proposition~\ref{Poincareduality}
follow from \cite[Chapter~III, Proposition~10.2]{iversen}.
To complete the proof of Proposition~\ref{Poincareduality}, it remains to apply the results recalled above,
with $A=\Q/\Z$ for part~(1) and with $A=\Z$ for part~(2),
and to make the following two observations.
First, one has
$\RHom(\L,j_!\widetilde{A}) = j_! \Homrond(j^*\L, \widetilde{A})$
if either $A$ is a divisible abelian group or if~$\L$ has torsion-free stalks.
Secondly,
if~$\sC$ is a bounded complex of abelian groups,
one has
$\Hom_{D^+(\pt)}(\sC,\Q/\Z) = \Hom(H^0(\sC),\Q/\Z)$,
and if $H^1(\sC)$ is finitely generated,
then also
 $\Hom_{D^+(\pt)}(\sC,\Z)/\mathrm{tors} = \Hom(H^0(\sC),\Z)$.
These last two canonical isomorphisms result from \cite[Example~3.2]{kellerderived}.
\end{proof}

\subsection{\texorpdfstring{$\Z$}{Z}-polynomial maps on abelian groups}
\label{parZpolynomial}

The construction of integral relations between $\omega$ and the $c_i$ runs into difficulties related to the possible failure of the Mittag\nobreakdash--Leffler condition (on noncompact manifolds).  Proposition \ref{liftlin}
is used in \S\ref{parintrelations} to overcome them.

Let $A$ be an abelian group. A map $f\colon A\to\Z$ is said to be $\Z$-\textit{polynomial} if for all $N\geq 1$ and all $a_1,\dots,a_N\in A$, there exists $P\in \Z[x_1,\dots,x_N]$ such that
for all $(x_1,\dots, x_N)\in\Z^N$,
the equality $f(\sum_{i=1}^N x_ia_i)=P(x_1,\dots,x_N)$ holds. 
This definition may be viewed as a particular case of Roby's polynomial laws \cite{Roby}.  
We say that $f$ is of degree $d$ (\resp is homogeneous of degree $d$) if so is $P$ for all choices of $(a_i)_{1\leq i\leq N}$. 

\begin{prop}
\label{liftlin}
Let $A$ be a countable abelian group. Fix $m\geq 1$.  If a group morphism $f\colon A\to\Z/m$ admits a $\Z$-polynomial lift $g\colon A\to\Z$, it also admits a $\Z$-linear lift $h\colon A\to \Z$.
\end{prop}

\begin{proof}
As $A$ is countable, one can write $A=A' \oplus A''$, where $A'$ is a free abelian group and $\Hom(A'',\Z)=0$
(see \cite[Lemma 7]{NuRo}).  We may suppose that $A=A'$ or that $A=A''$.  As the conclusion of the proposition always holds if $A$ is free, we may therefore assume that $\Hom(A,\Z)=0$.  To conclude the proof, it now suffices to show that every $\Z$-polynomial map $g\colon A \to \Z$ is constant.

Assume first that $g$ has degree $d$ for some $d\geq 0$.  We argue by induction on~$d$. If~$d=0$, the conclusion is clear, so we suppose that $d>0$. For all $b\in A$, the formula $a\mapsto g(a+b)-g(a)$ defines a $\Z$-polynomial map of degree $\leq d-1$ on $A$, which is constant by the induction hypothesis. It follows that the map $a\mapsto g(a)-g(0)$ is $\Z$-linear on $A$, and hence identically zero since $\Hom(A,\Z)=0$. This shows that $g$ is constant.

In the general case, we let $g_d\colon A\to \Z$ be the map associating with $a$ the coefficient of~$x^d$ in the polynomial map $x\mapsto g(xa)$. The map $g_d$ is $\Z$-polynomial and homogeneous of degree $d$,  hence identically zero if $d>0$.  It follows that $x\mapsto g(xa)$ is constant, from which we deduce that $g(a)=g(0)$.  We have proved that $g$ is constant.
\end{proof}

\subsection{Integral relations}
\label{parintrelations}

We now construct relations with integral coefficients between~$\omega$ and the~$c_i$.  
Keep the notation of \S \ref{parrelations}.  For all $j\in\Z$, identify $H^*_G(\BU,\Z(j))$ with a subgroup of $\Z[\omega,(c_i)_{i\geq 1}]/(2\omega)$ as in (\ref{cohoBU}).
Fix $?\in\{\top, \ac,\an,\alg\}$. As in \S\ref{parrelations}, one can evaluate an element $P\in H^*_G(\BU,\Z(j))$ on $M\in \Theta^?_n$ to get a class $P(\omega, c(M))\in H^*_G(M,\Z(j))$.  Let~$J_{\Z(j),n}^{?}$ be the subset of $H^*_G(\BU,\Z(j))\subset\Z[\omega,(c_i)_{i\geq 1}]/(2\omega)$ consisting of those classes that vanish on all $M\in\Theta^?_n$.

Denote by $\beta_{\Z(j)}$ the boundary maps associated with the short exact sequence of $G$\nobreakdash-modules $0\to \Z(j)\xrightarrow{2}\Z(j)\to\Z/2\to 0$. Then define
\begin{equation}
\label{defKnZ}
K_{\Z(j),n}=\beta_{\Z(j)}\bigg(K_{n}+\sum_{k,l\geq 0} \Big(H^{2n-2^{l+1}k}_G(\BU,\Z(n))/2\Big)\Sq^{2^lk}\cdots\Sq^{2k}\Sq^k\bigg),
\end{equation}
which we view as a subgroup of $H^*_G(\BU,\Z(j))\subset\Z[\omega, (c_i)_{i\geq 1}]/(2\omega)$ (see (\ref{cohoBU2})).

\begin{thm}
\label{reltopoZ}
For $j\in\Z$, one has $K_{\Z(j),n}\subset J_{\Z(j),n}^{\top}$.
\end{thm}

\begin{proof}
That $\beta_{\Z(j)}(K_{n})\subset J_{\Z(j),n}^{\top}$ is a consequence of Theorem \ref{reltopo}.  We now fix $k,l\geq 0$ and a class $x\in H^{2n-2^{l+1}k}_G(\BU,\Z(n))$. Let $\bar x\in H^{2n-2^{l+1}k}_G(\BU,\Z/2)$ be the reduction of $x$ mod~$2$.
Our goal is to show that $\beta_{\Z(j)}((\bar x)\Sq^{2^lk}\cdots\Sq^{2k}\Sq^k)\in J_{\Z(j),n}^{\top}$.
To do so, we fix a stably complex $\ci$ $G$-manifold $M$ of dimension $n$ with $M^G=\varnothing$ and we assume, as we may, that $M/G$ is connected.  Consider a $G$-equivariant map $\tau\colon M\to \BU\times \Z$ classifying $\tau_M\in\KR(M)$, and the right action of $\dA$ on $H^*_G(M,\Z/2)$ induced by $-\tau^*\kappa=-[TM]$. 

\setcounter{Case}{0}
\begin{Case} 
$M$ is compact.
\end{Case}

Define $y:=(\tau^*\bar x)\Sq^{2^lk}\cdots\Sq^{2k}\Sq^k\in H^{2n-k}_G(M,\Z/2)$. By Proposition~\ref{Poincareduality}~(1) applied on $M/G$, the class~$\beta_{\Z(j)}(y)$ vanishes if and only if $\beta_{\Z(j)}(y)\smile t=0$ for all classes $t\in H^{k-1}_G(M,\Q/\Z(j'))$,
where $j':=n-j$.  Let $\delta$ and $\bar\delta$ be the boundary maps of the short exact sequences of $G$\nobreakdash-modules \begin{equation}
\label{2ses}
0\to \Z(j')\to \Q(j')\to \Q/\Z(j')\to 0\textrm{ and }0\to\Z/2\to\Q/\Z(j')\xrightarrow{2}\Q/\Z(j')\to 0.
\end{equation}
The class $\beta_{\Z(j)}(y)\smile t\in H^{2n}_G(M,\Q/\Z(n))$ is the image of $y\smile \bar\delta(t)\in H^{2n}_G(M,\Z/2)$ by the injective morphism $\Z/2=H^{2n}_G(M,\Z/2)\to H^{2n}_G(M,\Q/\Z(n))=\Q/\Z$ induced by the injection $\Z/2\hookrightarrow\Q/\Z(n)$ (apply \cite[Lemma 07MC]{SP}). 
We deduce that $\beta_{\Z(j)}(y)=0$ if and only if $y\smile \bar\delta(t)=0$ for all $t\in H^{k-1}_G(M,\Q/\Z(j'))$.

Fix $z\in H^{k}_G(M,\Z(j'))$ and let $\bar z$ be its reduction mod $2$.
Proposition~\ref{droitegauche} yields
\begin{equation}
\label{powerof2}
y\smile\bar z=(\tau^*\bar x)\Sq^{2^lk}\cdots\Sq^{2k}\Sq^k\smile \bar z=\tau^*\bar x\smile\Sq^{2^lk}\cdots\Sq^{2k}\Sq^k(\bar z)=\tau^*\bar x\smile\bar z\,^{2^{l+1}}.
\end{equation}
The obvious morphism from the first to the second short exact sequence of (\ref{2ses}) shows that $\bar\delta(t)$ is the reduction mod $2$ of $\delta(t)$.  Applying (\ref{powerof2}) with $z=\delta(t)$ and $\bar z=\bar\delta(t)$ shows that $y\smile\bar\delta(t)$ is the reduction mod $2$ of
$\tau^*x\smile \delta(t)^{2^{l+1}}\in H^{2n}_G(M,\Z(n))$. As $\delta(t)$ is torsion and $H^{2n}_G(M,\Z(n))$ is torsion-free (see \S\ref{parstablycomplex}), we conclude that $y\smile\bar\delta(t)=0$ and hence that~$\beta_{\Z(j)}(y)=0$. We have shown that $\beta_{\Z(j)}((\bar x)\Sq^{2^lk}\cdots\Sq^{2k}\Sq^k)$ vanishes on~$M$.

\begin{Case} 
$M$ is not compact.
\end{Case}

Exhaust~$M$ by an increasing sequence $(M_i)_{i\geq 1}$ of compact $G$\nobreakdash-submanifolds of dimension~$n$ with boundary,  as in the proof of Theorem \ref{reltopo}.  We may suppose that the~$M_i/G$ are connected.  Let~$S_i$ be the double of $M_i$ (as in~\S\ref{parstablycxconstr}).  By the compact case treated above, the class $\beta_{\Z(j)}((\bar x)\Sq^{2^lk}\cdots\Sq^{2k}\Sq^k)$ vanishes on $S_i$, hence on $M_i$.  We deduce the existence of $\gamma_i\in H^{2n-k}_G(M_i,\Z(j))$ whose reduction mod $2$ is $\bar\gamma_i:=((\tau^*\bar x)\Sq^{2^lk}\cdots\Sq^{2k}\Sq^k)|_{M_i}$.

Set $A_i:=H^{k}_G(M_i,\partial M_i,\Z(n-j))$ and $A:=\varinjlim_i A_i$.
 Define $g_i\colon A_i\to \Z$ by
$$g_i(z):=(\tau^*x)|_{M_i}\smile z^{2^{l+1}}\in H^{2n}_G(M_i,\partial M_i,\Z(n))=\Z$$
for $z\in A_i$. If $\oz\in H^{k}_G(M_i,\partial M_i,\Z/2)$ denotes the reduction of $z$ mod~$2$, then
the reduction of $g_i(z)$ mod~$2$ is equal to $\bar\gamma_i\smile \bar z$ in $H^{2n}_G(M_i,\partial M_i,\Z/2)=\Z/2$. 
 Indeed, this identity can be checked after composition with the isomorphism $H^{2n}_G(M_i,\partial M_i,\Z/2)\isoto H^{2n}_G(S_i,\Z/2)$, where it results from~(\ref{powerof2}) applied on $S_i$ (to the image of $z$ in $H^{k}_G(S_i,\Z(n-j))$).

The $g_i$ fit together to give rise to a $\Z$-polynomial map $g\colon A\to \Z$.  As the $A_i$ are countable because $M_i$ is compact, so is $A$. We can therefore apply Proposition \ref{liftlin} to find a $\Z$-linear map $h\colon A\to \Z$ such that the reductions mod~$2$ of $g$ and $h$ coincide.

By Proposition \ref{Poincareduality}~(2) applied on $M_i/G$,  the morphism 
\begin{equation}
\label{cuppmap}
H^{2n-k}_G(M_i,\Z(j))\to\Hom(A_i,\Z)
\end{equation}
induced by cup product is surjective and its kernel is $H^{2n-k}_G(M_i,\Z(j))_{\torsion}$.
We can therefore find $\delta_i\in H^{2n-k}_G(M_i,\Z(j))$ whose image by (\ref{cuppmap}) is $h|_{A_i}$.  As $h|_{A_i}$ mod~$2$ is given by $z\mapsto \bar\gamma_i\smile \bar z$,  the image of $\delta_i-\gamma_i$ by (\ref{cuppmap}) is divisible by $2$. It follows that there exist
$\varepsilon_i\in H^{2n-k}_G(M_i,\Z(j))$ and $\zeta_i\in H^{2n-k}_G(M_i,\Z(j))_{\torsion}$ with $\delta_i+\zeta_i=\gamma_i+2\varepsilon_i$.

Let $F_i\subset H^{2n-k}_G(M_i,\Z(j))$ be the subset of those $\xi_i\in H^{2n-k}_G(M_i,\Z(j))$ whose reduction mod $2$ is $\bar\gamma_i$ and whose image by (\ref{cuppmap}) is $h|_{A_i}$. This set is nonempty (it contains $\delta_i+\zeta_i=\gamma_i+2\varepsilon_i$)  and finite (any two of its members differ by an element of the finite set $H^{2n-k}_G(M_i,\Z(j))_{\torsion}$). It is therefore possible to pick elements $\xi_i\in F_i$ in a compatible way, so $(\xi_i)_{i\geq 1}\in \varprojlim_{i}H^{2n-k}_G(M_i,\Z(j))$.  Let $\xi\in H^{2n-k}_G(M,\Z(j))$ be an element inducing $(\xi_i)_{i\geq 1}$. 

The reduction mod $2$ of $\xi$ is equal to $(\bar\gamma_i)_{i\geq 1}$ in $\varprojlim_{i}H^{2n-k}_G(M_i,\Z/2)=H^{2n-k}_G(M,\Z/2)$ (the Mittag--Leffler condition is satisfied because the groups $H^{*}_G(M_i,\Z/2)$ are finite), and hence equals $(\tau^*\bar x)\Sq^{2^lk}\cdots\Sq^{2k}\Sq^k$.  We deduce that $\beta_{\Z(j)}((\tau^*\bar x)\Sq^{2^lk}\cdots\Sq^{2k}\Sq^k)=0$.  The class $\beta_{\Z(j)}((\bar x)\Sq^{2^lk}\cdots\Sq^{2k}\Sq^k)$ therefore vanishes on~$M$.
\end{proof}

We record the following integral counterparts of Questions \ref{qtopalg} and \ref{qKJ}.

\begin{quest}
\label{qtopalgZ}
Are the inclusions $J_{\Z(j),n}^{\top}\hspace{-.1em}\subset\hspace{-.1em} J_{\Z(j),n}^{\ac}\hspace{-.1em}\subset \hspace{-.1em}J_{\Z(j),n}^{\an}\hspace{-.1em}\subset \hspace{-.1em}J_{\Z(j),n}^{\alg}$ 
equalities for $j\in\Z$?
\end{quest}

\begin{quest}
\label{qKJZ}
Is the inclusion $K_{\Z(j),n}\subset J_{\Z(j),n}^{\top}$ an equality for $j\in\Z$?
\end{quest}

\begin{rmks}
(i) 
The equality $J_{\Z(j),n}^{\top}=J_{\Z(j),n}^{\ac}$ is proved in Theorem \ref{thtopac} below, but the other cases of Question \ref{qtopalgZ} are open.

(ii) 
One could also consider weakenings of Question \ref{qKJZ}, in the spirit of Remarks \ref{remrelmod2} (ii) and~(iii).

(ii)
Similar questions with $\Q/\Z$ coefficients may also be of interest (see~\S\ref{parQZ}). We do not develop them here.
\end{rmks}

\subsection{Stably complex versus almost complex \texorpdfstring{$G$}{G}-manifolds}
\label{stacpar}

In Theorem \ref{thtopac}, we show that exactly the same relations hold in the stably complex and in the almost complex cases.

The next proposition is a $G$-equivariant analogue of (a well-known improvement of) \cite[Theorem 1.7]{Thomas}.
Let $E$ be a $G$-equivariant real vector bundle of rank $r$ on a $G$\nobreakdash-space~$S$.
Letting~$\widetilde{\Z}$ denote its $G$\nobreakdash-equivariant orientation sheaf,
we define the \textit{$G$-equivariant Euler class} $e(E)\in H^r_G(S,\widetilde{\Z})$
of $E$ to be the usual (twisted) Euler class of the real vector bundle on $(S \times EG)/G$ induced by~$E$.
If~$E$ is endowed with a stable complex structure $(k,J)$, so that $r=2n$ for an integer~$n$,
we view~$e(E)$ as an element of
$H^{2n}_G(S,\Z(n))$ via the $G$\nobreakdash-equivariant identification $\widetilde{\Z} = \Z(n)$ determined by~$(k,J)$,
and we denote by $c_n(E) \in H^{2n}_G(S,\Z(n))$
the~$n$th Chern class of $(E,k,J)$.

\begin{prop}
\label{stunst}
Let $S$ be a $G$-space with $S^G=\varnothing$, such  that
$S/G$ has the homotopy type of a $CW$ complex of dimension $\leq 2n$.  Let $(k,J)$ be a stable complex structure 
on a $G$-equivariant real vector bundle $E$ of rank $2n$ on $S$. Then~$E$ admits a compatible structure of $G$\nobreakdash-equivariant complex vector bundle if and only if $c_n(E)=e(E)$ in~$H^{2n}_G(S,\Z(n))$.
\end{prop}

\begin{proof}
The direct implication results from \cite[Th\'eor\`eme 5]{kahnchern}. We prove the converse.

Let $\phi\colon S\to\BU$ be a $G$-equivariant map classifying~$(E,k,J)$. 
The fiber sequences $\bS^{2m+1}=\U(m+1)/\U(m)\to\BU(m)\to\BU(m+1)$ show that the homotopy fiber of $\BU(n)\to\BU$ is $2n$-connected. As $\dim(S)\leq 2n$, it therefore follows from obstruction theory (see \cite[Corollary 34.3]{Steenrod}) applied $G$-equivariantly (\ie on the space~$S/G$) that~$\phi$ lifts (up to $G$-equivariant homotopy) to a $G$-equivariant map $\psi\colon S\to \BU(n)$.

Let $F$ be the $G$-equivariant complex vector bundle of rank $n$ on $S$ classified by $\psi$.  To prove the proposition, we will show that $E$ and $F$ are isomorphic $G$\nobreakdash-equivariant real vector bundles. By the construction of $F$, we know that $E\oplus(S \times \C^N)$ and $F\oplus(S\times\C^N)$ are isomorphic $G$\nobreakdash-equivariant real vector bundles for $N\gg 0$. 
Applying obstruction theory on~$S/G$ as above (using the fiber sequences $\bS^{m}=\OO(m+1)/\OO(m)\to\BO(m)\to\BO(m+1)$) shows that $E\oplus(S \times \R)\simeq F\oplus (S \times \R)$ as $G$\nobreakdash-equivariant real vector bundles, since all obstructions vanish for dimension reasons.  Moreover, the only obstruction to $E$ being $G$-equivariantly isomorphic to~$F$ is then equal to $e(F)-e(E)\in H^{2n}_G(S,\Z(n))$ (while we could not find a reference stating this classical fact in this exact form, it follows from the obstruction computation \mbox{\cite[(12.5)]{Liao}}).
It remains to use \cite[Th\'eor\`eme 5]{kahnchern}, the stability of $G$-equivariant Chern classes (see~\S\ref{parcharclasses}) and the hypothesis that $c_n(E)=e(E)$ to compute that
$e(F)=c_n(F)=c_n(E)=e(E)\textrm{ in }H^{2n}_G(S,\Z(n))$.
\end{proof}

We deduce at once the following $G$-equivariant extension of \cite[Theorem 1.1]{Sutherland}.

\begin{cor}
\label{corstac}
Let $M$ be a compact stably complex $\ci$ $G$-manifold of dimension~$n$ with $M^G=\varnothing$ and $M/G$ connected.  Then $M$ has a compatible structure of almost complex~$\ci$ $G$-manifold if and only if  $c_n(M)=\chi_{\top}(M)/2$ in $H^{2n}_G(M,\Z(n))=\Z$.
\end{cor}

\begin{proof}
By Proposition \ref{stunst},  we must show that $e(TM)=\chi_{\top}(M)/2$ in \mbox{$H^{2n}_G(M,\Z(n))=\Z$}.
The composition of
the natural map $\Z=H^{2n}_G(M,\Z(n))\to H^{2n}(M,\Z)$
with the degree map $\deg\colon H^{2n}(M,\Z) = \Z^{\pi_0(M)} \to \Z$
is given by multiplication by~$2$.
Thus, the result follows from the equality $\deg(e)=\chi_{\top}(M)$,
where $e \in H^{2n}(M,\Z)$ denotes the (non-$G$-equivariant) Euler class of~$E$ on~$M$,
for which see \cite[Corollary~11.12]{milnorstasheff}.
\end{proof}

The next lemma is the heart of the proof of Theorem \ref{thtopac}.

\begin{lem}
\label{lemconsum}
Fix $n\geq 2$.
Let $M$ be a compact stably complex $\ci$ $G$\nobreakdash-manifold of dimension~$n$ with~$M^G=\varnothing$ and $M/G$ connected.  Then there exist a compact almost complex~$\ci$ $G$\nobreakdash-manifold~$\widetilde{M}$ of dimension~$n$ and a $G$\nobreakdash-equivariant $\ci$ map $\tilde{\pi}\colon \widetilde{M}\to M$ of degree~$1$ such that
$T\widetilde{M}$ and $\tilde{\pi}^*TM$
are stably isomorphic
as $G$\nobreakdash-equivariant real vector bundles
endowed with stable complex structures.
\end{lem}

\begin{proof}
Let $N$ be a connected compact $\ci$ manifold of dimension~$2n$ with stably trivial tangent bundle (to be specified later).
Fix $x\in M$ and $y\in N$. Let $B_x\subset M$ (\resp $B_y\subset N$) be a small closed ball in a coordinate chart around $x$ (\resp around $y$). Define $M^0:=M\setminus(\mathring B_x\cup\sigma(\mathring B_x))$ and~$N^0:=N\setminus \mathring B_y$. Gluing $M^0$ and two copies of $N^0$ appropriately along their boundaries yields a compact $\ci$ $G$-manifold $M'$: the connected sum of $M$ and two copies of $N$.
Let $\pi\colon M'\to M$ be a $G$-equivariant $\ci$ map which is the identity on $M^0$ and which sends the two copies of $N^0$ into $B_x$ and $\sigma(B_x)$ respectively.

We claim that the two $G$\nobreakdash-equivariant real vector bundles $\pi^*(TM)$ and $TM'$ on $M'$ are stably isomorphic.  More precisely, we will show that the tautological $G$-equivariant isomorphism between them on $M^0$ extends to a stable $G$-equivariant isomorphism on~$M'$. To do so, it suffices to work non-$G$-equivariantly over a small neighborhood~$U$ of~$B_x$ in~$M$. 
The stable trivializations of $\pi^*(TM)$ on $\pi^{-1}(U)\cap M^0$ induced by a trivialization of~$TM$ on~$U$ and of $TM'$ on $N^0$ induced by a stable trivialization of $TN$ differ on the sphere~$\partial N^0\simeq\bS^{2n-1}$ by a map $\phi\colon \bS^{2n-1}\to\OO:=\cup_{r\geq 0}\OO(r)$.  The map $\phi$ is induced by the clutching (or gluing) function $\bS^{2n-1}\to\OO(2n)$ for the tangent bundle of~$\bS^{2n}$ (as shown by an analysis of the connected sum construction),  and hence is homotopically trivial because~$T\bS^{2n}$ is stably trivial. This concludes the proof of the claim.

We can therefore endow~$M'$ with a structure of stably complex $\ci$ $G$-manifold such that
 $TM'$ and $\pi^*(TM)$
are stably isomorphic
as $G$\nobreakdash-equivariant real vector bundles
endowed with stable complex structures.
The existence of such a stable isomorphism implies that
$c_n(M')=\pi^*c_n(M)$. If $N=\bS^1\times \bS^{2n-1}$, one computes that $\chi_{\top}(M')/2=\chi_{\top}(M)/2-2$, and if $N=\bS^2\times \bS^{2n-2}$, that $\chi_{\top}(M')/2=\chi_{\top}(M)/2+2$.  

As the reductions mod $2$ of $c_n(M)$ and of $\chi_{\top}(M)/2=\chi_{\top}(M/G)$ are both equal to $w_{G,2n}(M)=w_{2n}(M/G)$ in $H^{2n}_G(M,\Z/2)=H^{2n}(M/G,\Z/2)=\Z/2$ (see Proposition~\ref{propequivw} and \cite[Corollary 11.12]{milnorstasheff}),  it is possible to kill the quantity $c_n(M)-\chi_{\top}(M)/2$ after performing such operations finitely many times.  The resulting $\ci$ $G$-manifold $\widetilde{M}$ and $G$\nobreakdash-equivariant $\ci$ map $\tilde{\pi}\colon \widetilde{M}\to M$ have the required properties by Corollary~\ref{corstac}.
\end{proof}

\begin{thm}
\label{thtopac}
For all $n\geq 0$ and all $j\in\Z$, one has $J^{\top}_{n}=J^{\ac}_{n}$ and $J^{\top}_{\Z(j),n}=J^{\ac}_{\Z(j),n}$.
\end{thm}

\begin{proof}
As the cases $n=0$ and $n=1$ are easily dealt with by hand, we assume that~$n\geq 2$.
Fix $\alpha\in J^{\ac}_{n}$ (or $\alpha\in J^{\ac}_{\Z(j),n}$). Let $M$ be a stably complex $\ci$ $G$\nobreakdash-manifold of dimension~$n$ with~$M^G=\varnothing$. To show that $\alpha|_M=0$, we may assume that $M/G$ is connected.  

If $M$ is not compact, then $H^{2n}_G(M,\Z(n))=0$. It follows from Proposition \ref{stunst} that~$M$ admits a structure of almost complex $\ci$ $G$-manifold, and hence that $\alpha|_M=0$.
If $M$ is compact, let $\tilde{\pi}\colon\widetilde{M}\to M$ be as in Lemma \ref{lemconsum}.  One has $\alpha|_{\widetilde{M}}=0$ because $\widetilde{M}$ is an almost complex $\ci$ $G$-manifold of dimension $n$ with $\widetilde{M}^G=\varnothing$. 
Since $T\widetilde{M}\simeq \tilde{\pi}^*TM$,  we deduce that $\tilde{\pi}^*(\alpha|_M)=\alpha|_{\widetilde{M}}=0$,  and the projection formula implies that $\alpha|_M=\tilde{\pi}_*\tilde{\pi}^*(\alpha|_M)=0$.
\end{proof}

\section{Computation and examples of relations}
\label{seccomputations}

In this section, we fix an integer $n\geq 0$ and we let the algebra $\dA$ act on the right on $H^*_G(\BU,\Z/2)$ by means of the class $-\kappa_{n}$, as explained in \S\ref{parrelations}. Using the results presented in \S\ref{parrelations} and \S\ref{parintrelations}, we give concrete examples of relations satisfied on all stably complex $\ci$ $G$-manifolds of dimension~$n$ with no $G$-fixed points.

\subsection{A formula for the total Wu class}

Our main computational tool is formula~(\ref{rightaction2}), applied on $(\BU\times EG)/G$.  Letting $v:=u_G(\kappa_n)$
denote the total $G$\nobreakdash-equivariant Wu class of the tautological bundle $\kappa_n$,
so that $v_i\in H^i_G(\BU,\Z/2)$ for each~$i$, this formula reads
\begin{equation}
\label{rightactionwu}
(x)\Sq=v\smile\chi(\Sq)(x)
\end{equation}
for all $x\in H^*_G(\BU,\Z/2)$.  In particular,  one has
\begin{equation}
\label{wur}
v=(1)\Sq.
\end{equation}
We view $v$ as an element of $\Z/2[[\oo,(\oc_i)_{i\geq 1}]]$ depending implicitly on~$n$. Let $M$ be a stably complex~$\ci$ $G$-manifold of dimension~$n$ with $M^G=\varnothing$. If~$\tau\colon M\to BU$ is a $G$\nobreakdash-equivariant map classifying $\tau_M$, we also denote by $v\in H^*_G(M,\Z/2)$ the class $\tau^*v=u_G(\tau_M)$. 

In order to use (\ref{rightactionwu}) or (\ref{wur}), we need a formula for $v\in \Z/2[[\oo,(\oc_i)_{i\geq 1}]]$ which is given in Corollary~\ref{Wukappa} below.  
For $m\geq 0$,  let $T_m\in\Q[c_1,\dots,c_m]$ be the $m$th Todd polynomial (see \cite[\S 1.7]{Hirzebruch}). 
By definition, the $(T_m)_{m\geq 0}$ form a multiplicative sequence of polynomials with characteristic series $\frac{x}{1-e^{-x}}$ (in the sense of \cite[\S 1.2]{Hirzebruch}); in particular~$T_0=1$ and each~$T_m$ is homogeneous of degree~$m$
with respect to the grading of $\Q[c_1,\dots,c_m]$ for which $\deg(c_i)=i$ for all~$i$.
As indicated in \mbox{\cite[\S 2.11]{atiyahhirzebruchrrdiff}}, the coefficients of $2^mT_m$ have nonnegative $2$-adic valuation. We can therefore define $t_m\in\Z/2[\oc_1,\dots,\oc_m]$ to be the reduction mod $2$ of $2^mT_m$.
 In view of \mbox{\cite[\S 2.9, (13)]{atiyahhirzebruchrrdiff}}, the~$t_m$ form a multiplicative sequence with characteristic series $1+\sum_{l\geq 0}x^{2^l}$.
Below, we shall write $t_m(\bar c)$ as shorthand for $t_m(\bar c_1, \bar c_2, \dots)$.

\begin{thm}
\label{Wuvb}
Let $E$ be a $G$-equivariant complex vector bundle of rank $r$ on a $G$-space $S$.
The following identity holds in the graded ring $H^*_G(S,\Z/2)$:
\begin{equation}
\label{uE}
u_G(E)=\sum_{m\geq 0} \bigg(1+\sum_{l\geq 0}\oo^{2^l}\bigg)^{r-2m}t_m(\oc(E)).
\end{equation}
\end{thm}

\begin{proof}
Arguing as at the beginning of the proof of Proposition~\ref{propequivw}, we may assume that $S$ is compact. We prove the theorem by induction on $r\geq 1$ (the case $r=0$ being trivial). 

Assume first that $r=1$. In this case, $\oc_i(E)=0$ for $i\geq 2$. Set $\oc_1:=\oc_1(E)$. By the definition and the computation (recalled above) of the characteristic series of~$(t_m)_{m\geq 0}$, one has $t_m(\oc(E))=\oc_1^{\, m}$ if $m=0$ or $m=2^k$ for some $k\geq 0$, and $t_m(\oc(E))=0$ otherwise. 
Multiplying~(\ref{uE}) with $\big(1+\sum_{l\geq 0}\oo^{2^l}\big)^{-1}$ and applying $\Sq$ therefore yields the equivalent identity
\begin{equation}
\label{uEbis}
\Sq\mkern1.5mu\bigg(\mkern-1.5mu\Big(1+\sum_{l\geq 0}\oo^{2^l}\Big)^{-1}\smile u_G(E)\mkern-1.5mu\bigg)=\Sq\mkern1.5mu\bigg(\mkern-1.5mu1+\sum_{k\geq 0} \Big(1+\sum_{l\geq 0}\oo^{2^l}\Big)^{-2^{k+1}}\oc_1^{\,2^k}\mkern-1.5mu\bigg)\rlap,
\end{equation}
which we shall now prove.
Identity (\ref{chiSqomega}) implies that 
\begin{equation}
\label{identiSqomega}
\Sq\mkern1.5mu\bigg(\mkern-1.5mu\Big(1+\sum_{l\geq 0}\oo^{2^l}\Big)^{-1}\bigg)=(1+\oo)^{-1}=\sum_{j\geq 0}\oo^j.
\end{equation}
Applying the Cartan formula (\ref{Cartan}), the identity (\ref{identiSqomega}), the definition of the $G$-equivariant Wu class given in \S\ref{parSteenalg}, and Proposition~\ref{propequivw}, one sees that the left side of (\ref{uEbis}) is equal to
\begin{equation}
\label{lswu}
\Sq\mkern1.5mu\bigg(\mkern-1.5mu\Big(1+\sum_{l\geq 0}\oo^{2^l}\Big)^{-1}\bigg)\smile w_G(E)=\sum_{j\geq 0}\oo^j\smile(1+\oo+\oc_1)=1+\bigg(\sum_{j\geq 0}\oo^j\smile\oc_1\bigg)\rlap.
\end{equation}
The Cartan formula (\ref{Cartan}), the identity (\ref{identiSqomega}) and the computation $\Sq(\oc_1)=\oc_1+\oo\oc_1+\oc_1^2$ (where the equality $\Sq^1 \oc_1 = \oo\oc_1$ follows from the commutativity of the diagram
\begin{align}
\begin{aligned}
\label{cd:bocktwist}
\xymatrix@R=3ex{
0 \ar[r] & \Z \ar[r]\ar[d] & \Z[G]\ar[d] \ar[r] & \Z(1)  \ar[d]\ar[r] & 0\\
0 \ar[r] & \Z/2 \ar[r] & \Z/4 \ar[r] & \Z/2 \ar[r] & 0
}
\end{aligned}
\end{align}
(see~\eqref{rc})) together show that the right side of (\ref{uEbis}) equals
\begin{equation}
\label{rswu}
1+\sum_{k\geq 0}\big(\sum_{j\geq 0}\oo^{2^{k+1}j}\smile(\oc_1^{\,2^k}+\oo^{2^k}\oc_1^{\,2^k}+\oc_1^{\,2^{k+1}})\big)=1+\big(\sum_{j\geq 0}\oo^j\smile\oc_1\big)
\end{equation}
as all terms involving higher powers of $\oc_1$ cancel out. Comparing (\ref{lswu}) and (\ref{rswu}) concludes the proof of~\eqref{uE} in this case.

Assume now that $r>1$.  By the splitting principle for $G$-equivariant complex vector bundles (see \cite[Th\'eor\`eme 1]{kahnchern}),
 we may assume that $E=E'\oplus E''$ with $E'$ and $E''$ of positive ranks $r'$ and $r''$.
One then computes
\begin{equation*}
\begin{alignedat}{5}
u_G(E)&=u_G(E')\smile u_G(E'')\\
&=\sum_{m\geq 0} \bigg(1+\sum_{l\geq 0}\oo^{2^l}\bigg)^{r'-2m}t_m(\oc(E'))\smile\sum_{m\geq 0} \bigg(1+\sum_{l\geq 0}\oo^{2^l}\bigg)^{r''-2m}t_m(\oc(E''))\\
&=\sum_{m\geq 0} \bigg(1+\sum_{l\geq 0}\oo^{2^l}\bigg)^{r-2m}t_m(\oc(E)),
\end{alignedat}
\end{equation*}
where the first equality combines the Whitney formula for $G$-equivariant Stiefel\nobreakdash--Whitney classes and the Cartan formula (\ref{Cartan}),  the second equality is the induction hypothesis, and the third equality holds by the Whitney formula for $G$-equivariant Chern classes as the~$(t_m)_{m\geq 0}$ form a multiplicative sequence.
\end{proof}

\begin{cor}
\label{Wukappa}
The following identity holds in $H^*_G(\BU,\Z/2)$:
\begin{equation}
\label{v}
v=\sum_{m\geq 0} \bigg(1+\sum_{l\geq 0}\oo^{2^l}\bigg)^{n-2m}t_m(\oc).
\end{equation}
\end{cor}

\begin{proof}
It suffices to prove (\ref{v}) after restriction to each $G$-equivariant compact subset of~$\BU$ (argue  as at the beginning of the proof of Proposition~\ref{propequivw}).  On such a subset, the universal bundle is represented by a difference of $G$-equivariant complex vector bundles (see \S\ref{parGeqvb}). The statement therefore follows from Theorem \ref{Wuvb} and from the multiplicativity of the formula~(\ref{uE}).
\end{proof}

\subsection{Coefficients of Wu classes}

We now draw a few consequences from Corollary~\ref{Wukappa}. 
Let $N(n,k)\in\Z/2$ be the coefficient of $x^k$ in $\big(1+\sum_{l\geq 0}x^{2^l}\big)^{n}\in\Z/2[[x]]$. In view of Corollary~\ref{Wukappa},  it is the coefficient of the monomial~$\oo^k$ in the expression (\ref{v}) for the~$k$th Wu class $v_k$ in dimension $n$. For later use in the proofs of Theorem~\ref{omega2n-1} and Proposition~\ref{reln+1},  we gather a few properties of these numbers.

\begin{lem}
\label{lemNnk}
The following identities hold for all $n,k\geq 0$.
\begin{enumerate}[(i)]
\item $N(2n,2k)=N(n,k)$ and $N(2n,2k+1)=0$.
\item $N(2n+1,2k+1)=N(n,k)$ and $N(2n+1,2k)=N(n+1,k)$.
\item $N(n,n)=1$.
\item $N(2n+1, 2n+2)=1$.
\item $N(n+1,n)=0$ if $n>0$.
\end{enumerate}
\end{lem}

\begin{proof}
Set $q(x):=1+\sum_{l\geq 0}x^{2^l}\in\Z/2[[x]]$. 
The relation $q(x)^{2n}=q(x^2)^n$ implies at once that $N(2n,2k)=N(n,k)$ and $N(2n,2k+1)=0$,  proving (i). 
Similarly,  the identity $q(x)^{2n+1}=q(x)q(x^2)^n$ shows that $N(2n+1, 2k+1)=N(n,k)$. It also implies that
\begin{equation}
\label{reccheloue}
N(2n+1,2k)=N(n,k)+\sum_{l\geq 0}N(n,k-2^l)
\end{equation}
(taking the convention that $N(n,j)=0$ for $j<0$).
Furthermore, we deduce from the relation $q(x)^{n+1}=q(x)q(x)^n$ the equality $N(n+1,k)=N(n,k)+\sum_{l\geq 0}N(n,k-2^l)$. Combining it with (\ref{reccheloue}) finishes the proof of~(ii).  Assertion (iii) follows from (i) and (ii) by induction on $n$ (with base case $N(0,0)=1$). To prove (iv), use (ii) and (iii) to write $N(2n+1, 2n+2)=N(n+1,n+1)=1$.
Finally, we prove (v) by induction on $n\geq 1$. If~$n$ is odd, then (v) follows from (i).  If $n\geq 2$ is even, then $N(n+1,n)=N(\frac{n}{2}+1,\frac{n}{2})=0$ by~(ii) and the induction hypothesis.
\end{proof}

The next lemma will be used in Case \ref{case2} of the proof of Theorem \ref{omega2n-1}.

\begin{lem}
\label{wuexact}
Let $n\geq 1$ be an odd integer. Then $v_{n-1}=t_{\frac{n-1}{2}}(\oc)$ in $\Z/2[\oo,(\oc_i)_{i\geq 1}]$.
\end{lem}

\begin{proof}
Corollary \ref{Wukappa} shows that 
$$v_{n-1}=\sum_{m=0}^{\frac{n-1}{2}}N(n-2m,n-2m-1)\,\oo^{n-2m-1}t_m(\oc).$$
By Lemma \ref{lemNnk} (v), the coefficient $N(n-2m,n-2m-1)$ vanishes unless $m=\frac{n-1}{2}$. 
\end{proof}

\subsection{Low-dimensional relations}

Here are some examples of relations with $\Z/2$ coefficients obtained by our method in low dimensions. This list is by no means exhaustive.

\begin{prop}
\label{lowrelations}
Let $M$ be a stably complex $\ci$ $G$-manifold of dimension~$n$ with $M^G=\varnothing$. 
\begin{enumerate}[(i)]
\item If $n=1$, then $\oo^2=\oc_1$ in $H^2_G(M,\Z/2)$.
\item If $n=2$, then $\oo^3=\oo\oc_1=0$ in $H^3_G(M,\Z/2)$.
\item If $n=3$, then $\oo^4+\oo^2\oc_1+\oc_1^2+\oc_2=0$ in $H^4_G(M,\Z/2)$ and $\oo^3\oc_1=0$ in $H^5_G(M,\Z/2)$. In addition, $\oo^2\oc_2=\oc_3$,  $\oo^2\oc_1^2=\oc_1^3$ and $\oo^6=\oc_1^3+\oc_3$ in $H^6_G(M,\Z/2)$.
\item If $n=4$, then $\oo^3\oc_1=0$ in $H^5_G(M,\Z/2)$.
\end{enumerate} 
\end{prop}

\begin{proof}
All these relations are proved in two steps.  First, an identity in $H^*_G(\BU,\Z/2)$ is verified by applying (\ref{rightactionwu}) and Corollary \ref{Wukappa} blindly.  In doing so, one may need to use formulae for $\Sq^j(\oc_i)$ (such as $\Sq^2(\oc_2)=\oo^2\oc_2+\oc_1\oc_2+\oc_3$), which are deduced from  the Wu formula \cite[Problem 8-A]{milnorstasheff} thanks to Proposition~\ref{propequivw}. Secondly, this identity is seen to imply the desired relation by Theorem \ref{reltopo}. We only state the identities that we use:

When $n=1$,  we use the identity $(1)\Sq^2=\oo^2+\oc_1$.

When $n=2$, we rely on the identities $(\oo)\Sq^2=\oo^3+\oo\oc_1$ and $(1)\Sq^2\Sq^1=\oo\oc_1$.

When $n=3$, we exploit the identities $(1)\Sq^4=\oo^4+\oo^2\oc_1+\oc_1^2+\oc_2$ and $(\oo^2)\Sq^3=\oo^3\oc_1$, as well as $(\oc_2)\Sq^2=\oo^2\oc_2+\oc_3$,  $(\oc_1^2)\Sq^2=\oo^2\oc_1^2+\oc_1^3$ and $(\oo^2)\Sq^4+(\oc_1^2+\oc_2)\Sq^2=\oo^6+\oc_1^3+\oc_3$.

When $n=4$, we only need to compute that $(1)\Sq^4\Sq^1=\oo^3\oc_1$.
\end{proof}

\begin{rmks}
\label{remslowrelations}
Let us now comment on the significance of some of the relations of Proposition~\ref{lowrelations}.
We assume that $M$ is compact and connected. 

(i) When $n=1$, the relation $\oo^2=\oc_1$ in $H^2_G(M,\Z/2)=\Z/2$ means that $\oo^2_M$ vanishes if and only if $c_1\in H^2_G(M,\Z(1))=\Z$ is divisible by $2$,  or equivalently if and only if the (non-$G$-equivariant) Chern number $c_1(M)\in H^2(M,\Z)=\Z$ is divisible by $4$. 
This happens if and only if then genus~$g(M)$ of $M$ is even, as $c_1(M)=2-2g(M)$.
This fact was used in \cite{bw1} to detect the genus of algebraic curves lying on a smooth proper variety $X$ over~$\R$ with $X(\R)=\varnothing$, by means of their $G$-equivariant cycle class (see \cite[Theorem~3.6]{bw1}).

(ii) Similarly, when $n=3$, the relation $\oo^6=\oc_1^3+\oc_3$ means that $\oo^6_M$ vanishes if and only if the (non-$G$-equivariant) Chern number $(c_1^3+c_3)(M)\in H^6(M,\Z)=\Z$ is divisible by $4$.  It is remarkable that the vanishing of $\oo^6_M$, which a priori is an invariant of the $G$-space~$M$, can be read off from the space $M$ without its $G$-action (but using its stable complex structure).

(iii) The relation $\oo^3=0$ when $n=2$ is the reduction mod $2$ of the (stronger) integral relation $\omega^3=0$ (see Proposition \ref{omegalow}).

(iv) To comment on the relation $\oo^2\oc_2=\oc_3$ when $n=3$, let us suppose that $M=X(\C)$ for some irreducible smooth proper threefold $X$ over $\R$ with $X(\R)=\varnothing$. Consider Krasnov's $G$-equivariant cycle class map for $1$-cycles $\cl\colon\CH^2(X)\to H^4_G(X(\C),\Z(2))$ (see \cite[\S 1.6.1]{bw1}).  Let $\psi\colon H^4_G(X(\C),\Z(2))\to H^6_G(X(\C),\Z/2)=\Z/2$ be the map given by cup product with~$\oo^2$. It is shown in \cite[Theorem~3.6]{bw1} that if $B\subset X$ is a closed integral curve, then $\psi(\cl([B]))\neq 0$ if and only if $B$ is geometrically irreducible of even geometric genus.  One can therefore think of the relation $\oo^2\oc_2=\oc_3$
as computing the parity of the genus of a real algebraic curve representing $c_2(X)\in\CH^2(X)$.  Indeed, it shows that $\psi(\cl(c_2(X)))=0$ if and only the topological Euler characteristic ${\chi_{\top}(X(\C))=c_3(X(\C))\in H^6(X(\C),\Z)=\Z}$ is divisible by~$4$. 

(v) The discussion in (iv) also applies to the relation $\oo^2\oc_1^2=\oc_1^3$ when $n=3$.  It shows that the non-$G$-equivariant Chern number $c_1^3(X(\C))\in H^6(X(\C),\Z)=\Z$ controls the parity of the genus of any real algebraic curve representing $c_1(X)^2\in\CH^2(X)$, in the sense that $\psi(\cl(c_1(X)^2))=0$ if and only if $K_{X_{\C}}^3=-c_1^3(X(\C))$
is divisible by~$4$.
\end{rmks}

Let us record the following curious consequences of Remarks \ref{remslowrelations} (iv) and (v).  Let $X$ be a proper variety over $\R$. For $i\geq 0$, Koll\'ar \cite[Definition 1]{kollarelw} 
defined, following~\cite{elw}, the $i$th
\textit{intermediate index} $\ind_i(X)\in\Z$ of $X$ to be the gcd of the integers $\chi(X,\sF)$ when~$\sF$ ranges over all coherent sheaves on~$X$ whose support has dimension $\leq i$.

\begin{prop}
Let $X$ be a geometrically irreducible smooth proper threefold over $\R$ with $X(\R)=\varnothing$.  
 If either $\chi_{\top}(X(\C))$ or~$K_{X_{\C}}^3$ is not divisible by $4$,  then $\ind_1(X)=1$.

If moreover $X_{\C}$ is rationally connected or simply connected Calabi--Yau, then~$X$ satisfies the real integral Hodge conjecture for $1$-cycles in the sense of \cite[Definition 2.2]{bw1}.
\end{prop}

\begin{proof}
By Remarks \ref{remslowrelations} (iv) and (v),  if either $\chi_{\top}(X(\C))$ or $K_{X_{\C}}^3$ is not divisible by $4$, then~$X$ contains a geometrically irreducible curve of even geometric genus. The proposition therefore follows from \cite[Corollaries 3.11 and 3.23]{bw1}.
\end{proof}

\subsection{Vanishing of powers of \texorpdfstring{$\omega$}{omega}}
\label{parvanomegai}

We turn to relations of the form $\omega^e=0$. Let us first prove some identities of this form in low dimensions.

\begin{prop}
\label{omegalow}
Fix $n\in\{2,4,5,6\}$. Set $e(2)=3$, $e(4)=7$, $e(5)=9$ and $e(6)=10$. If~$M$ is a stably complex $\ci$ $G$\nobreakdash-manifold of dimension~$n$ with $M^G=\varnothing$, then~$\omega^{e(n)}_M=0$.
\end{prop}

\begin{proof}
A direct computation based on (\ref{rightactionwu}) and Corollary \ref{Wukappa} shows the identity
\begin{equation}
\begin{alignedat}{5}
\label{identi2456}
(1)\Sq^2&=\oo^2+\oc_1 &&\textrm{\hspace{.2em} if }n=2,\\
(\oo)\Sq^5+(1)\Sq^4\Sq^2&=\oo^6+\oc_1^3+\oc_3&&\textrm{\hspace{.2em} if }n=4,\\
(\oo^4+\oc_1^2+\oc_2)\Sq^4+(\oc_1\oc_2)\Sq^2&=\oo^8+\oo^2(\oc_1^3+\oc_1\oc_2+\oc_3)&&\textrm{\hspace{.2em} if }n=5,\\
 (\oo^2\oc_1)\Sq^5\hspace{-.1em}+\hspace{-.1em}(\oo^5+\oo(\oc_1^2+\oc_2))\Sq^4\hspace{-.1em}+\hspace{-.1em}(1)\Sq^6\Sq^3&=\oo^9+\oo^7\oc_1+\oo^3(\oc_1^3+\oc_3)&&\textrm{\hspace{.2em} if }n=6.
\end{alignedat}
\end{equation}
In view of (\ref{defKnZ}), the image by $\beta_{\Z(e(n))}$ of the left side of (\ref{identi2456}) belongs to $K_{\Z(e(n)),n}$, 
hence so does the image by $\beta_{\Z(e(n))}$ of the right side of (\ref{identi2456}).  Since the latter equals $\omega^{e(n)}$, applying Theorem~\ref{reltopoZ} concludes the proof.
\end{proof}

Theorem \ref{omega2n-1} below provides identities of this kind for many more values of $n$.
We first prove a result which will allow us to argue by induction on the dimension.  We grade the ring $\Z[c_1,\dots,c_n]$ by setting $\deg(c_i)=i$.

\begin{lem}
\label{lemchernstab}
Let $M$ be a stably complex $\ci$ $G$-manifold of dimension~$n$ with $M^G=\varnothing$. 
Let $P\in\Z[c_1,\dots, c_n]$ be homogeneous of degree $d$. Then there exist a stably complex~$\ci$ $G$-manifold $M'$ of  dimension $n-d$ and a proper $G$-equivariant map $\pi\colon M'\to M$ such that $\pi_*1=P(c(M))$ in $H^{2d}_G(M,\Z(d))$.
\end{lem}

\begin{proof}
Let $\tau\colon M\to\BU$ be a $G$-equivariant map classifying $\tau_M$ (see \S\S\ref{parGeqvb}-\ref{parstablycomplex}).  If $r,N\geq 0$ are big enough, the homotopy fiber of the inclusion $\Gr_r(\C^N)\to \BU$ is $2n$-connected, and it follows from obstruction theory applied $G$-equivariantly (\ie on the space $M/G$) that 
the map $\tau$ factorizes (up to $G$-homotopy) through $\tau\colon M\to \Gr_r(\C^N)$.  After possibly increasing $r$ and $N$, we may moreover ensure that $2n<r(N-r)$.

 By \cite[Theorem~1.4]{Bierstone}, we may assume that $\tau$ is in general position with respect to~$\Gr_r(\C^N)^G$ in the sense of \cite[Definition~1.2]{Bierstone}.  Since $M^G=\varnothing$, this implies that $\tau$ is transverse to $\Gr_r(\C^N)^G$,  and hence that $\tau(M)\cap \Gr_r(\C^N)^G=\varnothing$ 
as $2n<r(N-r)$. 
Choose a $G$\nobreakdash-invariant proper $\ci$ map $f\colon M\to \C$ (there even exists one with values in $\R\subset\C$).  Applying \cite[Chap. 2, Theorem~2.13]{Hirsch} to the proper map $$M/G\to \Big((\Gr_r(\C^N)\setminus\Gr_r(\C^N)^G)\times\C\Big)/G$$ induced by $(\tau,f)$ shows that we may assume,  after perturbing $\tau$ and $f$,  that the $\ci$ map $(\tau,f)\colon M\to\Gr_r(\C^N)\times\C$ is a closed embedding.

Since it suffices to prove the lemma for generators of the $\Z$-module $\Z[c_1,\dots,c_n]$ (making use of disjoint unions, and of the opposite construction of \S\ref{parstablycxconstr} to change orientations), and since the Chow ring of the real algebraic variety $\Gr_r(\C^N)$ is generated by the Chern classes of the tautological bundle (see \eg \cite[Proposition 14.6.5]{Fulton}),  we may assume that~$P(c(M))=\tau^*[Z]$ for some codimension $d$ real algebraic subvariety $Z\subset \Gr_r(\C^N)$. Let $W\to Z$ be a resolution of singularities. Let $g\colon W\to \Gr_r(\C^N)$ be the induced map.  By \cite[Theorem~1.4]{Bierstone},  a small $G$-equivariant perturbation $h$ of the map ${g\times\Id\colon W\times \C\to \Gr_r(\C^N)\times\C}$ is in general position with respect to $M\subset \Gr_r(\C^N)\times \C$,  and hence is transversal to $M$ as $h^{-1}(M)^G\subset h^{-1}(M^G)=\varnothing$.

Set $M':=h^{-1}(M)$ and let $\pi\colon M'\to M$ be the induced map. By transversality,  the stable complex structures on $M$, on $\Gr_r(\C^N)\times\C$, and on $W\times\C$ endow~$M'$ with a structure of stably complex $\ci$ $G$-manifold of dimension $n-d$.  Transversality also implies that $\pi_*1=(\tau,f)^*h_*1=(\tau,f)^*(g\times\Id)_*1=\tau^*g_*1=P(c(M))$ in $H^{2d}_G(M,\Z(d))$.
\end{proof}

\begin{cor}
\label{lemmarec}
Fix $0\leq d\leq n$ and let $P\in\Z[c_1,\dots, c_n]_d$ be homogeneous of degree $d$.  \linebreak
If $\omega^e\in J_{\Z(e),n-d}^{\top}$, then $\omega^e P(c)\in J_{\Z(e+d),n}^{\top}$ (where $J_{\Z(j),n}^{\top}$ is defined in \S \ref{parintrelations}).
\end{cor}

\begin{proof}
Let $M$ be a stably complex $\ci$ $G$-manifold of dimension~$n$ with $M^G=\varnothing$.  
Let $\pi\colon M'\to M$ be as in Lemma \ref{lemchernstab}. As $\omega^e\in J_{\Z(e),n-d}^{\top}$, it follows from the projection formula that $\omega^e_MP(c(M))=\omega^e_M\smile\pi_*1=\pi_*\pi^*\omega^e_M=\pi_*\omega^e_{M'}=0$ in $H^{e+2d}_G(M,\Z(e+d))$.
\end{proof}

\begin{thm}
\label{omega2n-1}
Let $n\geq 0$ be an integer not of the form $2^k-1$.
Let $M$ be a stably complex~$\ci$ $G$-manifold of dimension~$n$ with $M^G=\varnothing$. Then $\omega_M^{2n-1}=0$.
\end{thm}

\begin{proof}
We prove, by induction on~$n$, the equivalent statement that $\beta_{\Z(1)}(\oo_M^{2n-2})=0$. The induction hypothesis combined with Corollary \ref{lemmarec} implies the following assertion.
\begin{equation}
\label{hyprec}
\parbox{35em}{If $0<d<n$ is even with $n-d$ not of the form $2^k-1$, and if $P\in\Z[c_1,\dots, c_n]$ is homogeneous of degree $d$,  then  $\beta_{\Z(1)}(\oo^{2n-2d-2}_MP(\oc(M)))=0$.
}
\end{equation}

We successively analyse four cases, exhausting all possible values of $n$.

\setcounter{Case}{0}
\begin{Case}
$n\equiv 0\pmod 2$ and $n\neq 0$.
\end{Case}

In $H^{2n}_G(M,\Z/2)$, consider the trivial identity
\begin{equation}
\label{eqcase1}
\oo^{2n-2}=v_n\,\oo^{n-2}-(v_n-\oo^n)\,\oo^{n-2}.
\end{equation}

Recall that $v_n=(1)\Sq^n$ (see (\ref{wur})).
Applying (\ref{Cartanr})
and noting that $(\oo^{2n-2-k})\Sq^k=0$ for $k>2$ by Theorem~\ref{reltopo},
the first term of the right-hand side of~(\ref{eqcase1}) is therefore equal to $(1)\Sq^n\smile \oo^{n-2}=(\oo^{2n-4})\Sq^2$.
It is annihilated by $\beta_{\Z(1)}$ by Theorem \ref{reltopoZ}.

The monomial~$\oo^n$ does not appear in $v_n-\oo^n\in\Z/2[\oo, (\oc_i)_{i\geq 1}]$, by Lemma~\ref{lemNnk}~(iii)
and Corollary~\ref{Wukappa}.
The monomials appearing in $(v_n-\oo^n)\,\oo^{n-2}$ are therefore all of the form~$\oo^{2n-2d-2}P(\oc)$, for $P\in \Z[c_1,\dots,c_n]$ homogeneous of degree $d$ with $0<d<n$. They belong to the kernel of~$\beta_{\Z(1)}$, by~(\ref{hyprec}) for $d$ even, and for $d$~odd because they are the reductions mod~$2$ of $\omega^{2n-2d-2}P(c)$.

\begin{Case}
\label{case2}
$n\equiv 1 \pmod 8$ and $n\neq 1$.
\end{Case}

In this case, we compute that
\begin{equation}
\label{eqcase2}
\oo^{2n-2}=(1)\Sq^{n-1}\smile \oo^{n-1}=\oo^{n-1}v_{n-1}=\oo^{n-1}t_{\frac{n-1}{2}}(\oc).
\end{equation}
The first equality in (\ref{eqcase2}) follows from (\ref{Cartanr}), because most terms vanish by Theorem~\ref{reltopo} and the classes
$\Sq^{n-2}(\oo^{n-1})=(n-1)\oo^{2n-3}$
and $\Sq^{n-3}(\oo^{n-1})=\frac{(n-1)(n-2)}{2}\oo^{2n-4}$ (see~(\ref{Aomega})) also vanish, as $n\equiv 1\pmod 4$. 
The second equality is (\ref{wur}) and the third is Lemma \ref{wuexact}.  Finally, the right-hand side of (\ref{eqcase2}) is annihilated by $\beta_{\Z(1)}$ thanks to (\ref{hyprec}) applied with $d=\frac{n-1}{2}$.  (Note that $n-d$ is never of the form~$2^k-1$ because $n\equiv 1 \pmod 8$ and $n\neq 1$.)

\begin{Case}
$n\equiv 5\pmod 8$.
\end{Case}

When $n=5$,  the theorem has already been proved in Proposition \ref{omegalow}. We therefore assume that $n\neq 5$ and we compute in $H^{2n}_G(M,\Z/2)$ that
\begin{equation}
\label{eqcase3}
\oo^{2n-2}=(1)\Sq^4\smile\oo^{2n-6}=\oo^{2n-6}v_4=\oo^{2n-6}(\oc_1^2+\oc_2),
\end{equation}
where the first equality follows from (\ref{Cartanr}) (a term vanishes by Theorem~\ref{reltopo} and the others can be computed using (\ref{Aomega})),  the second is (\ref{wur}), and the third is an instance of~(\ref{v}).  Since $n\equiv 5\pmod 8$ and~$n\neq 5$, the right side of (\ref{eqcase3}) is annihilated by $\beta_{\Z(1)}$ thanks to~(\ref{hyprec}) applied with $d=2$.

\begin{Case}
$n\equiv 2^l-1\pmod {2^{l+1}}$ and $n\neq 2^l-1$, for some $l\geq 2$.
\end{Case}

In $H^{2n}_G(M,\Z/2)$, consider the trivial identity
\begin{equation}
\label{eqcase4}
\oo^{2n-2}=v_{2^l}\,\oo^{2n-2^l-2}-(v_{2^l}-\oo^{2^l})\,\oo^{2n-2^l-2}.
\end{equation}
The monomial~$\oo^{2^l}$ does not appear in $v_{2^l}-\oo^{2^l}\in\Z/2[\oo, (\oc_i)_{i\geq 1}]$, by Lemma \ref{lemNnk} (iii)
and Corollary~\ref{Wukappa}.
All the monomials appearing in the second term of the right-hand side of~(\ref{eqcase4}) are therefore annihilated by~$\beta_{\Z(1)}$,  either because they have obvious lifts to $H^{2n-2}_G(M,\Z(1))$, or by (\ref{hyprec}); at least one of these two arguments works for each monomial, as $n\equiv 2^l-1\pmod {2^{l+1}}$ and $n\neq 2^l-1$.
To conclude the proof, we will show that the first term of  the right-hand side of (\ref{eqcase4}) vanishes on the nose. As $v_{2^l}=(1)\Sq^{2^l}$ by (\ref{wur}),  we may apply (\ref{Cartanr}) to compute it.  Most terms that arise vanish by Theorem~\ref{reltopo},  and we deduce that it is equal to
\begin{equation}
\label{eqcase4bis}
\binom{2n-2^l-2}{2^l}\oo^{2n-2}+\binom{2n-2^l-2}{2^l-1}(\oo^{2n-3})\Sq^1+\binom{2n-2^l-2}{2^l-2}(\oo^{2n-4})\Sq^2.
\end{equation}
Using (\ref{rightactionwu}) and (\ref{v}), one computes that $(\oo^{2n-3})\Sq^1=0$. 
The middle term of (\ref{eqcase4bis}) therefore vanishes. We finally verify that the other two terms of (\ref{eqcase4bis}) also vanish by noting that the binomial coefficients that appear are even,
as a straightforward application of Lucas' theorem \cite[Theorem~1]{finebinomial}.
\end{proof}

We note the following consequence of Theorem \ref{omega2n-1}.

\begin{cor}
\label{th:omega2n-1ouvert}
Let $n\geq 0$ be an integer not of the form $2^k-1$.
Let~$X$ be a smooth variety of dimension $n$ over~$\R$. Then $\omega^{2n-1}_{X(\C)\setminus X(\R)}\in H^{2n-1}_G(X(\C)\setminus X(\R),\Z(2n-1))$ vanishes.
\end{cor}

We also put forward the particular case of Corollary \ref{th:omega2n-1ouvert} where $n=2$ and $X(\R)=\varnothing$.

\begin{cor}
\label{cor:omega3}
Let $X$ be a smooth surface over $\R$ with $X(\R)=\varnothing$. Then $\omega^3_{X}=0$ in~$H^3_G(X(\C),\Z(3))$.
\end{cor}

\begin{rmks}
\label{rmk:ane pas soif}
(i)
It is necessary,
in Theorem~\ref{omega2n-1},
to assume
that~$n$ is not of the form $2^k-1$.
Indeed, in Proposition~\ref{prop:omegai quadriques} below,
we will see that
for any integer~$k \geq 1$,
if $Q$ denotes the real anisotropic quadric of dimension $n=2^k-1$ and $M=Q(\C)$,
then $\omega_M^{2n}\neq\emptyset$.

(ii)
When $n>0$ is even, the vanishing of $\omega_M^{2n}$ in Theorem \ref{omega2n-1} is obvious: it is a torsion class in the torsion-free group $H^{2n}_G(M,\Z(2n))$
(see~\textsection\ref{parstablycomplex}).
In contrast, the vanishing of~$\omega_M^{2n}$ when $n$ is an odd integer not of the form $2^k-1$ is a nontrivial statement.

(iii)
We do not know how to prove Theorem \ref{omega2n-1} for all values of $n$ by producing identities similar to~(\ref{identi2456}) in a systematic way, without an induction based on 
Lemma~\ref{lemchernstab}.

(iv)
Theorem \ref{omega2n-1} is not optimal when $n=6$, as Proposition \ref{omegalow} shows. We investigate the optimality of our results concerning identities of the form $\omega^e=0$ in Section~\ref{secomegai}.
\end{rmks}

\begin{rmks}
\label{remsomega3}
(i)
In \cite[p.~758]{Krasnov}, Krasnov writes that he has not succeeded in finding a smooth surface $X$ over $\R$ with $X(\R)=\varnothing$ and $\oo^3_{X}$ nonzero.
Corollary~\ref{cor:omega3} explains why: in fact, no such surfaces exist.

(ii)
For particular classes of surfaces, Corollary~\ref{cor:omega3} (or a closely related statement)
appears in \cite[Lemma 4]{parimalasujatha},  \cite[Lemma 2.3]{vanhameldivisors} and
in \cite[Proposition~1.4 and Theorem~1.6]{Krasnov}; see also the assertion that the pull-back map~$\Br(\R)\to\Br(X)$ vanishes in the proof of \cite[Theorem~3.2]{sujathavanhamel}. The main tool used in most of these articles is the Hochschild--Serre spectral sequence, which is difficult to analyse when $H^1(X(\C),\Z/2)$ is nonzero. Our proof of Corollary~\ref{cor:omega3} bypasses the use of the Hochschild--Serre spectral sequence, and exploits Poincar\'e duality instead (see (iii) below). We note that some of Krasnov's arguments are closer to our approach (see the proof of \mbox{\cite[Theorem~1.6]{Krasnov}}).

(iii) Given its importance for the applications appearing
in Theorem~\ref{coindexsurfaces} and in
Corollary~\ref{2dim}, we
sketch a short and direct proof of Corollary \ref{cor:omega3},
obtained by distilling the proof of Theorem~\ref{omega2n-1}
for the special case of~$\omega^3_X$.
It avoids the technicalities required for the general case,
and
highlights the role of Poincar\'e duality in our arguments.
Replacing~$X$ with a smooth
compactification, we may assume that~$X$ is proper.
Now the assumption that $X(\R)=\emptyset$ implies that~$X(\C)/G$ is a compact~$\ci$ manifold.
As explained in
\cite[Proposition~1.10]{bw1}
or in Proposition~\ref{Poincareduality}~(1),
Poincaré duality on this manifold provides an isomorphism
$H^4_G(X(\C),\Q/\Z(2))=\Q/\Z$ which, when combined with the cup product
and with the natural identification
$H^4_G(X(\C),\Q/\Z(4))=H^4_G(X(\C),\Q/\Z(2))$, gives rise to a
perfect pairing
\begin{align}
\label{eq:poincare duality h33}
 H^3_G(X(\C),\Z(3))\times H^1_G(X(\C),\Q/\Z(1)) \to \Q/\Z\rlap.
\end{align}
To see that $\omega^3_X=0$, it now suffices to prove that $\omega^3\smile x=0$ for
$x\in H^1_G(X(\C),\Q/\Z(1))$. 
Let $y\in H^2_G(X(\C),\Z(1))_\torsion$ be the image of $x$ by the boundary map of
$$
0\to\Z(1)\to\Q(1)\to\Q/\Z(1)\to 0\rlap,
$$
and let~$\oy\in H^2_G(X(\C),\Z/2)$ be its reduction mod $2$. As $\omega^3=\beta_{\Z(3)}(\oo^2)$,
it suffices to show that~$\oo^2\smile \oy=0$ in $H^4_G(X(\C),\Z/2)$
(see \cite[Lemma~07MC]{SP}). As a consequence
of the Wu formula,
the second Wu class of the compact manifold $X(\C)/G$
can be written as
 $u_{G,2}(\tau_{X(\C)})=\oo^2+\oc_1$ (see Theorem~\ref{Wuvb});
hence $\oy^2 = (\oo^2 + \oc_1) \smile y$.
 It follows that $\oo^2\smile \oy$ is the reduction mod~$2$ of $(c_1\smile y)+y^2\in H^4_G(X(\C),\Z(2))$.
The class $(c_1\smile y)+y^2$ is torsion
since so is $y$;
as $H^4_G(X(\C),\Z(2))$ is torsion-free, this class must therefore vanish. This concludes the proof.
\end{rmks}

\begin{rmk}
Vanishing results for the $\omega^e$, such as Proposition \ref{omegalow} or Theorem \ref{omega2n-1}, can be used to obstruct the algebraicity of $G$-equivariant cohomology classes, as follows. 

Assume that $\omega^e\in J^{\top}_{\Z(e),n-c}$.  Let $X$ be a smooth variety of dimension $n$ over $\R$ with $X(\R)=\varnothing$.  If $Y$ is a smooth variety of dimension $n-c$ over $\R$ and $f\colon Y\to X$ is a proper morphism, then $\omega_X^e\smile f_*1=f_*\omega^e_Y=0$ by the projection formula.  All classes in the image of the $G$-equivariant cycle class map $\cl\colon \CH^c(X)\to H^{2c}_G(X(\C),\Z(c))$ are therefore in the kernel of cup product with~$\omega_X^e$, and a class in 
$H^{2c}_G(X(\C),\Z(c))$ whose cup product with~$\omega_X^e$ does not vanish cannot be algebraic.

Let us illustrate this strategy with an example extracted from \cite[Theorem 4.16]{Steenrodalgebraic} (where a different argument is used).  Let $Q$ be the real anisotropic quadric threefold and let~$E$ be an elliptic curve over $\R$ such that $E(\R)$ is not connected.  Set ${X:=E\times_{\R}Q}$. By \cite[Proposition 4.13]{Steenrodalgebraic}, there exist $a\in H^1_G(E(\C),\Z)$ and $b\in H^1_G(E(\C),\Z(1))$ such that $a\smile b\in H^2_G(E(\C),\Z(1))$ is the class of a real point of~$E$.  As $\oo^6_Q\in H^6_G(Q(\C),\Z/2)$ is nonzero (see Proposition~\ref{prop:omegai quadriques} below) hence equal to the class of a closed point of $Q$,  we deduce that~$a\smile b\smile\oo_X^6\in H^8_G(X(\C),\Z/2)$ is the class of a closed point of $X$ and therefore nonzero.  Consequently, the class $b\smile\omega_X^3\in H^4_G(X(\C),\Z(2))$ is not killed by $\omega^3$, and hence cannot be algebraic (as~$\omega^3\in J^{\top}_{\Z(3),2}$, see Proposition~\ref{omegalow} or Remark \ref{remsomega3} (iii)).
\end{rmk}

\subsection{Relations involving  \texorpdfstring{$\oo^{n+1}$}{bar omega n+1} or \texorpdfstring{$\omega^{n+1}$}{omega n+1}}

The relations obtained in the next two propositions will be used in the proofs of Theorems~\ref{evendim} and \ref{odddim}.

\begin{prop}
\label{reln+1}
There exists a polynomial $\gamma\in H^{n+1}_G(\BU, \Z/2)\subset\Z/2[\oo, (\oc_i)_{i\geq 1}]$ (see~(\ref{cohoBU2})) such that the monomial $\oo^{n+1}$ does not occur in $\gamma$, and such that $\oo^{n+1}+\gamma\in J^{\top}_n$.
\end{prop}

\begin{proof}
When $n$ is odd,  we take $\gamma:=v_{n+1}+\oo^{n+1}$. The vanishing of $v_{n+1}\in H^{n+1}_G(M, \Z/2)$ follows from Theorem \ref{reltopo} since $v_{n+1}=(1)\Sq^{n+1}$ (see (\ref{rightactionwu})).  
The coefficient $N(n,n+1)$ of~$\oo^{n+1}$ in $v_{n+1}$ is nonzero as $n$ is odd by Lemma \ref{lemNnk} (iv).

When $n$ is even, we take $\gamma:=(\oo)\Sq^n+\oo^{n+1}$. The vanishing of $(\oo)\Sq^n\in H^{n+1}_G(M, \Z/2)$ follows from Theorem \ref{reltopo}. It remains to show that the coefficient of $\oo^{n+1}$ in~$(\oo)\Sq^n$ is nonzero.  Combining (\ref{rightactionwu}) and (\ref{chiSqomega}), we compute that $(\oo)\Sq^n=\sum_{l\geq 0}\oo^{2^l}v_{n-2^l+1}$. It follows that the coefficient of interest is equal to $\sum_{l\geq 0}N(n,n-2^l+1)$. As $n$ is even,  Lemma \ref{lemNnk} (i) implies that all terms with $l>0$ vanish and we deduce from Lemma \ref{lemNnk}~(iii) that the term with~$l=0$ is $N(n,n)=1$. This concludes the proof.
\end{proof}

\begin{prop}
\label{reln}
For $n$ even, there exists $\gamma\in H^{n}_G(\BU, \Z(n))\subset\Z[\omega,(c_i)_{i\geq 1}]/(2\omega)$ (see~(\ref{cohoBU})) such that the monomial $\omega^{n}$ does not occur in $\gamma$, and with $\omega^{n+1}+\omega \gamma\in J_{\Z(n+1),n}^{\top}$.
\end{prop}

\begin{proof}
Define $\gamma$ to be such that $\omega\gamma=\beta_{\Z(n+1)}(v_n+\oo^n)$ in $H^{n+1}_G(\BU, \Z(n+1))$.
As the monomial $\oo^n$ appears with coefficient $1$ in $v_n\in \Z/2[\oo, (\oc_i)_{i\geq 1}]$ by Lemma \ref{lemNnk} (iii), the monomial $\omega^n$ does not appear in~$\gamma$.
As $\beta_{\Z(n+1)}(v_n)=\beta_{\Z(n+1)}((1)\Sq^n)\in K_{\Z(n+1),n}$ (use~(\ref{rightactionwu}) and the definition (\ref{defKnZ})) vanishes in $H^{n+1}_G(M, \Z(n+1))$ by Theorem \ref{reltopoZ},  we see that $\omega\gamma=\beta_{\Z(n+1)}(\oo^n)=\omega^{n+1}$ in 
$H^{n+1}_G(M, \Z(n+1))$.
\end{proof}

\section{The case of the \texorpdfstring{$\omega^e$}{omega e}}
\label{secomegai}

In \S\ref{parvanomegai}, we obtained relations of a particular form: vanishing results for the constant cohomology classes $\omega^e$ (Proposition~\ref{omegalow},
Theorem~\ref{omega2n-1}, Corollary~\ref{th:omega2n-1ouvert}).
In this section,
we investigate the optimality of these results, as well as possible extensions.
To this end, we use
tools from \textsection\textsection\ref{secrelations}--\ref{seccomputations},
namely Theorems~\ref{reltopo}
and~\ref{thtopac} and
Corollaries~\ref{Wukappa} and~\ref{th:omega2n-1ouvert}.

\subsection{Nonvanishing of the \texorpdfstring{$\omega^e$}{omega e}}

The most natural task in this direction is the construction of smooth real algebraic varieties of dimension $n$ with no real points (or more generally of stably complex $\ci$ $G$-manifolds of dimension $n$ with no $G$-fixed points) on which~$\omega^e$ (or~$\oo^e$) is nonzero for a value of $e$ that is as large as possible.

\subsubsection{Anisotropic quadrics}
\label{paranisotropic}

The simplest example to consider is the real anisotropic quadric of dimension $n$: the variety with projective equation $\{\sum_{i=0}^{n+1}x_i^2=0\}$ in $\P^{n+1}_\R$.

\begin{prop}
\label{prop:omegai quadriques}
Let $n,e\geq 0$ be integers.  Let $Q$ be the real anisotropic quadric of dimension~$n$.
The following conditions are equivalent:
\begin{enumerate}[(i)]
\item the class $\omega_Q^e \in H^e_G(Q(\C),\Z(e))$ vanishes;
\item the class $\oo_Q^e \in H^e_G(Q(\C),\Z/2)$ vanishes;
\item there exists $k\geq 0$ such that $n<2^k-1\leq e$.
\end{enumerate}
\end{prop}

\begin{proof}
That (i) implies (ii) is obvious. 

Assume that (iii) holds and let $Q'$ be the real anisotropic quadric of dimension $2^{k-1}-1$.
As $Q'(\R)=\varnothing$
and $e>2\dim(Q')$, we have $\omega_{Q'}^e=0$.
Pfister has shown the existence of a rational map $Q\dashrightarrow Q'$ (see \cite[Satz 5]{PfisterStufe}).
Resolving its indeterminacies yields a smooth projective variety $X$ over $\R$,  a birational morphism $f\colon X\to Q$
and a morphism $g\colon X\to Q'$. One has $\omega_Q^e=f_*\omega_X^e=f_*g^*\omega_{Q'}^e=0$ by the
projection formula,  proving (i).
 
If (iii) does not hold,  then there exists $l\geq 0$ such that $2^{l}-1\leq n$ and $e<2^{l+1}-1$.  Let~$Q''$ be the real anisotropic quadric of dimension $2^{l+1}-2$. 
By the weak Lefschetz theorem, the restriction map $H^i(\P^{2^{l+1}-1}(\C),\Z/2)\to H^i(Q''(\C),\Z/2)$ is bijective if $i<2^{l+1}-2$ and injective if $i=2^{l+1}-2$. The Hochschild--Serre spectral sequence~(\ref{HS}) implies that the restriction maps $H^{i}_G(\P^{2^{l+1}-1}(\C),\Z/2)\to H^{i}_G(Q''(\C),\Z/2)$ enjoy the same properties.  As~$\oo^e_{\P^{2^{l+1}-1}_{\R}}\neq 0$ (its restriction to any point $x\in\P^{2^{l+1}-1}(\R)$ is nonzero) and~$e\leq 2^{l+1}-2$, we deduce that $\oo^e_{Q''}\neq 0$. Pfister has shown the existence of a rational map $Q''\dashrightarrow Q$ (see \cite[Satz~5]{PfisterStufe}). Resolving its indeterminacies yields a smooth projective variety~$X'$ over~$\R$,  a birational morphism $f'\colon X'\to Q''$
and a morphism $g'\colon X'\to Q$.  The projection formula shows that $\oo_{Q''}^e=f'_*\oo_{X'}^e=f'_*g'^*\oo_{Q}^e$,
and hence $\oo^e_Q\neq 0$, which proves~(i).
\end{proof}

If~$X$ is a smooth variety over~$\R$ with no real points,
and~$Q$ denotes the real anisotropic quadric of the same dimension as~$X$,
it often happens that
 the class $\omega^e$ (resp.~$\oo^e$) vanishes on~$X$ for a wider range of values of~$e$
 than it does on~$Q$ (\eg consider powers of the anisotropic conic).
We have not, however, been able to find
any example in which this range of values
is smaller (see Remark~\ref{remexemplessc} (iii) below),
thus leaving the next question open.

\begin{quest}
Do there exist $n,e\geq 0$ and a smooth variety $X$ of dimension~$n$ over~$\R$ with no real points such that $\omega^e_X\neq 0$ but $\omega^e_Q=0$, where $Q$ is the real anisotropic quadric of dimension~$n$?
\end{quest}

\subsubsection{Almost complex examples}
\label{parnonalgexamples}

To produce better examples than those obtained in~\S\ref{paranisotropic}, we resort to non-algebraic almost complex~$\ci$ $G$\nobreakdash-manifolds (see Theorem~\ref{nonalgexamples} below).

If $S$ is a $G$-space and $k\geq 0$, we consider the restriction map
\begin{equation}
\begin{alignedat}{5}
\label{decompocan}
\rho_S^k\colon H^k_G(S,\Z/2)\to H^k_G(S^G,\Z/2)&=H^k(S^G\times BG,\Z/2) \\
&=\bigoplus_{p+q=k}H^p(S^G,\Z/2)\otimes H^q(BG,\Z/2)\\
&=H^0(S^G,\Z/2)\oplus\dots\oplus H^k(S^G,\Z/2),
\end{alignedat}
\end{equation}
where the equalities come from the definition of $G$-equivariant cohomology, the K\"unneth formula, and the computation of the cohomology of $G$.

Let $\alpha(m)$ denote the number of ones in the binary expansion of $m$. 

\begin{lem}
\label{doubleapplication}
Fix $n\geq 0$, and set $k:=2n+1-\alpha(n+1)$.  There exist a compact stably complex $\ci$ $G$\nobreakdash-manifold~$S$ of dimension $n+1$ and $\tau\in H^0(S^G,\Z/2)$ such that $(\tau,0,\dots, 0)$ is not in the image of the map $\rho_S^{k}$ defined in (\ref{decompocan}).
\end{lem}

\begin{proof}
Write $n+1=\sum_{i=1}^{\alpha(n+1)} 2^{N_i}$ as a sum of distinct powers of $2$.
Set $X:=\prod_{i=1}^{\alpha(n+1)}\P^{2^{N_i}}_{\R}$. 
Letting $w_i \in H^*(X(\R),\Z/2)$  denote the $i$th Stiefel--Whitney class of the tangent bundle of~$X(\R)$,
we consider the normal Stiefel--Whitney classes $w'_i \in H^*(X(\R),\Z/2)$ for $i\geq 0$,
defined
by the relation ${ww'=1}$,
where
$w=\sum_{i\geq 0}w_i$ and $w'=\sum_{i\geq 0}w_i'$.
  An application of the Whitney formula shows that $w'_{n+1-\alpha(n+1)}(X(\R))\neq 0$.

Let $M$ be a $G$-invariant closed tubular neighborhood of $X(\R)$ in $X(\C)$. This compact stably complex $\ci$ $G$-manifold of dimension $n+1$ with boundary  $G$\nobreakdash-equivariantly retracts onto~$X(\R)$.
 Let $S$ be the double of $M$, constructed as in \S\ref{parstablycxconstr}. 
One has $S=M\cup -M$ and $S^G=M^G\sqcup (-M)^G\simeq X(\R)\sqcup X(\R)$.
Choose $$\tau:=(1,0)\in H^0(S^G,\Z/2)=H^0(M^G,\Z/2)\oplus H^0((-M)^G,\Z/2).$$

Consider the Mayer--Vietoris exact sequence
\begin{equation}
\label{MV}
H^k_G(S,\Z/2)\to H^k_G(M,\Z/2)\oplus H^k_G(-M,\Z/2)\to H^k_G(\partial M,\Z/2).
\end{equation}
The middle term of~(\ref{MV}) can be identified, by retraction, with $H^k_G(S^G,\Z/2)$, and hence
the left arrow of~(\ref{MV}) with $\rho_S^k$. Consequently, we only have to show that $(\tau,0,\dots, 0)$ is not 
in the kernel of the right arrow of (\ref{MV}).  
We henceforth identify $H^j_G(M,\Z/2)\isoto H^j_G(M^G,\Z/2)$ with $H^0(M^G,\Z/2)\oplus\dots\oplus H^j(M^G,\Z/2)$ by means of $\rho^j_M$. As the rightmost term of (\ref{MV}) can be identified, by retraction, with $H^k_G(M\setminus M^G,\Z/2)$, we
need to show that the restriction map $H^k_G(M,\Z/2)\to H^k_G(M\setminus M^G,\Z/2)$ does not annihilate the class $(1,0,\dots,0)$. 

To this end, we consider the exact sequence of relative cohomology
\begin{equation}
\label{withsupport}
H^{k-n-1}_G(M^G,\Z/2)=H^k_{G}(M,M \setminus M^G,\Z/2)\to H^k_G(M,\Z/2)\to H^k_G(M\setminus M^G,\Z/2),
\end{equation}
where the left equality is the $G$-equivariant Thom isomorphism. Using the above identifications, the left map $H^{k-n-1}_G(M^G,\Z/2)\to H^k_G(M,\Z/2) \isoto H^k_G(M^G,\Z/2)$ of (\ref{withsupport}) is the cup product with the $G$-equivariant mod $2$ Euler class of the normal bundle $N_{M^G/M}$ of $M^G$ in $M$, which is nothing but its top $G$-equivariant Stiefel--Whitney class $w_{G,n+1}(N_{M^G/M})$.  Multiplication by~$i$ identifies  $N_{M^G/M}$ with the real vector bundle $T(M^G)$ endowed with the $G$\nobreakdash-action given by multiplication by $-1$.  It therefore follows from \cite[Theorem~2.1]{krasnovequivariant} that $w_{G,n+1}(N_{M^G/M})$ can be identified with $w(T(M^G))=w(X(\R))$.

We now conclude the proof.  As $w'_{n+1-\alpha(n+1)}(X(\R))\neq 0$ and as $k-n-1=n-\alpha(n+1)$, the left map of (\ref{withsupport}), which we have computed to be the cup product with $w(X(\R))$, cannot have $(1,0,\dots,0)$ in its image.
It follows that $(1,0,\dots, 0)$ is not annihilated by the right map of (\ref{withsupport}),  as desired.
\end{proof}

We finally reach the goal of this paragraph.

\begin{thm}
\label{nonalgexamples}
For $n\geq 0$, there exists a compact almost complex $\ci$ $G$\nobreakdash-manifold~$M$ of dimension $n$ with $M^G=\varnothing$ such that $\oo^{2n-\alpha(n+1)}_M\in H^{2n-\alpha(n+1)}_G(M,\Z/2)$ is nonzero. 
\end{thm}

\begin{proof}
By Theorem \ref{thtopac}, it suffices to construct a compact stably complex $\ci$ $G$\nobreakdash-manifold of dimension $n$ with the required properties.

Let $S$ and $\tau\in H^0(S^G,\Z/2)$ be as in Lemma \ref{doubleapplication}. Choose a $\ci$ function $\phi\colon S\to \C$ that takes positive real values at the points of $S^G$ where $\tau=0$ and negative real values at those where $\tau=1$.  Replace $\phi$ with $\phi+\overline{\phi\circ\sigma}$ to make it $G$-equivariant, and perturb it to make it transversal to $0$ (to see that this is possible, use \cite[Theorem~1.4]{Bierstone} to put $\phi$ in general position with respect to $0\in\C$ in the sense of \cite[Definition~1.2]{Bierstone}, and note that this implies the desired transversality property since $G$ acts with trivial stabilizers on the zero locus of $\phi$).
Endow $M:=\{\phi=0\}\subset S$ with a structure of compact stably complex~$\ci$ $G$-manifold of dimension $n$ as explained in~\S\ref{parstablycxconstr}. One has $M^G=\varnothing$.

Let $\ci_S$ and $\ci_{S^G}$ be the sheaves of complex-valued $\ci$ functions on $S$ and $S^G$.
Consider the commutative diagram
\begin{equation}
\begin{aligned}
\label{diagexpS}
\xymatrix
@R=3ex
@C=0.33cm{
&&H^0_G(S^G,\Z/2)\ar^a[d] \\
H^0_G(S^G,(\ci_S)^*)\ar_{b'}[rd]\ar^d[r]&H^0_G(S^G,(\ci_{S^G})^*)\ar^e[ur]\ar^c[r]&H^1_G(S^G,\Z(1))\ar^{b}[d] \\
&H^1_G(S,S^G,(\ci_S)^*)\ar^{c'}[r]&H^2_G(S,S^G,\Z(1)),
}
\end{aligned}
\end{equation}
where $a$ is a boundary map of the short exact sequence of $G$-modules
\begin{equation}
\label{multpar2}
0\to \Z(1)\xrightarrow{2}\Z(1)\to\Z/2\to 0,
\end{equation}
where $b$ and $b'$ are boundary maps in relative cohomology long exact sequences, where~$c$ and~$c'$ are induced by $G$-equivariant exponential short exact sequences,  where $d$ is a restriction map,
and where~$e$ is the sign map (noting that all functions in $H^0_G(S^G,(\ci_{S^G})^*)$ are real-valued).  

Let us verify that the upper triangle of (\ref{diagexpS}) does commute.  Introduce the morphism $f\colon H^0_G(S^G,\Z/2)\to H^0_G(S^G,(\ci_{S^G})^*)$ induced by the morphism $\Z/2\to(\ci_{S^G})^*$ with image~$\{1,-1\}$.  One has $e\circ f=\Id$. The existence of an appropriate morphism from~(\ref{multpar2}) to the exponential sequence defining $c$ shows that $c\circ f=a$.  In addition,  as a function in~$\Ker(e)$ has a real logarithm, one has $\Ker(e)\subset\Ker(c)$. Now, for $x\in  H^0_G(S^G,(\ci_{S^G})^*)$, one computes 
$c(x)=cfe(x)+c(x-fe(x))=ae(x)$ as $x-fe(x)\in \Ker(e)\subset\Ker(c)$. 

Replacing $S^G$ with $S\setminus M$ in the bottom part of (\ref{diagexpS}) yields a diagram
\begin{equation}
\label{Thomspace}
H^0_G(S\setminus M,(\ci_S)^*)\to H^1_G(S,S\setminus M,(\ci_S)^*)\to H^2_G(S,S\setminus M,\Z(1)).
\end{equation}
Let $\beta\in H^2_G(S,S\setminus M,\Z(1))$ be the image of $\phi|_{S\setminus M}\in H^0_G(S\setminus M,(\ci_S)^*)$ in (\ref{Thomspace}).  (One can check that $\beta$ is the $G$-equivariant Thom class of the normal bundle of $M$ in $S$ defined in \cite[Proposition 5]{kahnchern}.)
Let $\gamma\in H^2_G(S,S^G,\Z(1))$ be the image of $\phi|_{S^G}\in H^0_G(S^G,(\ci_S)^*)$ in (\ref{diagexpS}). The compatibility of (\ref{diagexpS}) and (\ref{Thomspace}) shows that $\gamma$ is the image of $\beta$. In addition, the commutativity of (\ref{diagexpS}) shows that $\gamma=c'\circ b'(\phi|_{S^G})=b\circ a\circ e\circ d(\phi|_{S^G})=b\circ a(\tau)$.

At this point, set $k:=2n+1-\alpha(n+1)$ and assume for contradiction that $\oo^{k-1}_M=0$.  Let $i\colon M\hookrightarrow S$ be the inclusion.  Define $G$-equivariant push-forward maps 
\begin{equation}
\label{weirdGysin}
i_*\colon H^{*}_G(M,\Z/2)\to H^{*+2}_G(S,S^G,\Z/2)
\end{equation}
as follows.  Choose a $G$-equivariant open tubular neighborhood $U$ of $M$ in $S$. Define (\ref{weirdGysin}) to be the composition of the pull-back $H^{*}_G(M,\Z/2)\to H^{*}_G(U,\Z/2)$ by a $G$-equivariant retraction,  of the cup product $H^{*}_G(U,\Z/2)\to H^{*+2}_G(U,U\setminus M,\Z/2)$ with the restriction of~$\beta$, of the excision isomorphism $H^{*+2}_G(U,U\setminus M,\Z/2)=H^{*+2}_G(S,S\setminus M,\Z/2)$, and of the natural map
$H^{*+2}_G(S,S\setminus M,\Z/2)\to H^{*+2}_G(S,S^G,\Z/2)$. One computes that
$$0=i_*\oo_M^{k-1}=\gamma\smile\oo_S^{k-1}=b\circ a(\tau)\smile\oo_S^{k-1}=\tilde{b}(a(\tau)\smile\oo_S^{k-1}),$$
where $\tilde{b}\colon H^k_G(S^G,\Z/2)\to H^{k+1}_G(S,S^G,\Z/2)$ is a boundary map in a long exact sequence of relative cohomology.
The vanishing of $\tilde{b}(a(\tau)\smile\oo_S^{k-1})$ shows the existence of a class $\delta\in H^k_G(S,\Z/2)$ such that $\delta|_{S^G}=a(\tau)\smile\oo_S^{k-1}$. As the reduction of $a$ mod $2$, which is the boundary map associated with the short exact sequence $0\to \Z/2\to\Z/4(1)\to\Z/2\to 0$, is given by the cup product with $\oo$, we deduce that $\delta|_{S^G}=\tau\smile\oo_S^{k}$. It follows that
\begin{equation}
\label{rhos}
\rho_S^k(\delta)=\rho_{S^G}^k(\delta|_{S^G})=\rho_{S^G}^k(\tau\smile\oo_S^{k})=(\tau,0,\dots,0),
\end{equation}
where the third equality follows from the description of the last equality in (\ref{decompocan}). 

Equation (\ref{rhos}) contradicts our choice of $S$ and $\tau$ and concludes the proof.
\end{proof}

It is natural to wonder whether the examples constructed in Proposition~\ref{prop:omegai quadriques}
(see also Remark~\ref{rmk:ane pas soif}~(i))
 and Theorem~\ref{nonalgexamples} are (jointly) optimal. This gives rise to the following questions.

\begin{quests}
\label{questvanomega}
Let $M$ be a  stably complex $\ci$ $G$\nobreakdash-manifold of dimension~$n$ with~$M^G=\varnothing$.
If $n$ is not of the form $2^k-1$,  does 
$\oo_M^{2n+1-\alpha(n+1)}$ (or even $\omega_M^{2n+1-\alpha(n+1)}$) vanish?
\end{quests}

\begin{rmks}
\label{remexemplessc}
(i) 
The starting point of the construction in Lemma \ref{doubleapplication} is a compact~$\ci$ manifold of dimension $n+1$ (chosen to be a product of real projective spaces) whose characteristic class $w'_{n+1-\alpha(n+1)}$ does not vanish.  A theorem of Massey \cite[Theorem~I]{Massey} asserts that, in contrast, its characteristic classes $w'_i$ must vanish for ${i>n+1-\alpha(n+1)}$.  This indication that the construction might be optimal gives grounds for asking Questions~\ref{questvanomega}.

(ii) 
The smallest dimension for which our results leave Questions \ref{questvanomega} open is $n=10$.

(iii) 
Already when $n=5$, we do not known if the almost complex $G$-manifold of dimension $n$ constructed in Theorem \ref{nonalgexamples} can be chosen to be the set of complex points of a smooth proper variety of dimension $n$ over~$\R$ with no real points.
\end{rmks}

\subsection{Singular varieties}

Can our results concerning the vanishing of powers of $\omega$ (such as Proposition \ref{omegalow} 	and Theorem \ref{omega2n-1}) be extended to possibly singular real algebraic varieties with no real points? The next two propositions show that some can, while some cannot.

\begin{prop}
\label{propsing}
Let $X$ be a variety of dimension $n$ over $\R$ with $X(\R)=\varnothing$.  If $n$ is not of the form~$2^k-1$, one has $\omega_X^{2n}=0$ in $H^{2n}_G(X(\C),\Z(2n))$.
\end{prop}

\begin{proof}
Let $\pi\colon \tX\to X$ be a resolution of singularities.  An analysis of the dimension of the fibers of $\pi$ shows that, for $q>0$, the $G$-equivariant sheaf $\RR^q\pi_*\Z(2n)$ on $X(\C)$ is supported in dimension $\leq n-1-q/2$.  As $X(\R)=\varnothing$, it follows from the first spectral sequence of \cite[Th\'eor\`eme~5.2.1]{tohoku} that $H^p_G(X(\C), \RR^q\pi_*\Z(2n))=H^p(X(\C)/G,\sF)$, where $\sF$ is the sheaf on $X(\C)/G$ whose pull-back
to~$X(\C)$ is the $G$-equivariant sheaf $\RR^q\pi_*\Z(2n)$. Consequently, these groups vanish when $q>0$ and $p+q\geq 2n-1$ for dimension reasons.  We deduce from this, thanks to the Leray spectral sequence of $\pi$, that the map 
\begin{equation}
\label{map1}
H^{2n}_G(X(\C),\pi_*\Z(2n))\to H^{2n}_G(\tX(\C),\Z(2n))
\end{equation}
 is injective. In addition, the natural morphism $\Z(2n)\to\pi_*\Z(2n)$ of $G$-equivariant sheaves on $X(\C)$ is injective with cokernel $\sK$ supported in dimension $\leq n-1$. Arguing as above shows that $H^{2n-1}_G(X(\C),\sK)=0$,  hence that the map 
\begin{equation}
\label{map2}
H^{2n}_G(X(\C),\Z(2n))\to H^{2n}_G(X(\C),\pi_*\Z(2n))
\end{equation}
 is also injective. As the composition of (\ref{map2}) and (\ref{map1}) is the pull-back map, we see that if $\omega^{2n}$ were nonzero on $X$, it would be nonzero on $\tX$, contradicting Corollary~\ref{th:omega2n-1ouvert}.
\end{proof}

\begin{prop}
\label{singularomega3}
There exists an integral projective algebraic surface $X$ over $\R$ with $X(\R)=\varnothing$ such that the class $\oo^3_X\in H^3_G(X(\C),\Z/2)$ is nonzero.
\end{prop}

\begin{proof}
Let $\Gamma\subset\P^2_{\R}$ be the real conic with no real points. Let $p\in\Gamma$ be a closed point. Let $\pi\colon\tX\to \Gamma\times\Gamma$ be the blow-up of $p\times p$ with projections $pr_i\colon\tX\to\Gamma$ for $i\in\{1,2\}$ (note that $\pi_\C\colon\tX_\C\to\Gamma_\C\times\Gamma_\C$ is the blow-up of four complex points). The strict transform of $(\Gamma\times p)\cup (p\times\Gamma)$ in~$\tX$ is the disjoint union of $\Gamma\times p$ and $p\times \Gamma$ (both isomorphic to~$\Gamma_\C$). Use \cite[Th\'eor\`eme~5.4]{Ferrand} to glue~$\tX$ along these two curves by means of the obvious isomorphism and thus get a morphism $\mu\colon\tX\to X$ to an integral proper surface $X$ over $\R$. 

 Let $\tau\colon\tX\to\tX$ be the involution exchanging the two copies of $\Gamma$ and let $\iota\colon \Gamma_{\C}\to X$ be the inclusion of the glued locus.
Any ample line bundle on $\tX$ whose isomorphism class is $\tau$\nobreakdash-invariant descends to a line bundle on~$X$, which is ample by \cite[Proposition~2.6.2]{ega31}. This shows that $X$ is projective. It remains to verify the hypotheses \ref{i} and \ref{ii} of Lemma~\ref{lemomega3} below.

The exact sequence $0\to\sO_X\to\mu_*\sO_{\tX}\to \iota_*\sO_{\Gamma_\C}\to 0$ and the vanishing of $H^2(\tX,\sO_{\tX})$ and $H^1(\Gamma_{\C},\sO_{\Gamma_\C})$ imply that $H^2(X,\sO_X)=0$.
One has a commutative diagram
\begin{equation}
\begin{aligned}
\label{PicBr}
\xymatrix
@R=3ex
@C=0.33cm{
\Pic(\tX)\ar[r]&\Pic(\tX_\C)^G \ar[r] & \Br(\R)\mathrlap{{}=\Z/2}\ar@{=}[d] \\
\Pic(X)\ar[r]\ar[u]&\Pic(X_\C)^G\ar[r]\ar[u]&\Br(\R)\mathrlap{{}=\Z/2\rlap,}
}
\end{aligned}
\end{equation}
by \cite[8.1/4]{neronmodels}. Let $\lambda\in\Pic(\Gamma_{\C})^G\simeq\Z$ be a generator. The group $\Pic(\tX_{\C})^G\simeq \Z^4$ is spanned by $pr_1^*\lambda$, by $pr_2^*\lambda$, and by two exceptional divisors that are defined over $\R$. The first two generators are sent to the nontrivial class of $\Br(\R)$ and the other two to the trivial class.  Consequently, the $\tau$-invariant classes in $\Pic(\tX_{\C})^G$ are all sent to zero in $\Br(\R)$.  It thus follows from (\ref{PicBr}) that the map $\Pic(X_\C)^G\to\Br(\R)$ is zero, and hence that \ref{i} holds.

The short exact sequence $0\to \Z(1)\to\mu_*\Z(1)\to \iota_*\Z(1)\to 0$ of $G$-equivariant sheaves on $X(\C)$ yields a long exact sequence
\begin{equation}
\label{longGeq}
H^2_G(\tX(\C),\Z(1))\to H^2_G((\Gamma\times p)(\C),\Z(1))\to H^3_G(X(\C),\Z(1))\to H^3_G(\tX(\C),\Z(1)).
\end{equation}
Since $H^2_G((\Gamma\times p)(\C),\Z(1))=H^2(\Gamma(\C),\Z)=\Z$ and since the image in this group of $\cl(pr^*_1 K_{\Gamma})\in H^2_G(\tX(\C),\Z(1))$ is twice a generator, one gets from (\ref{longGeq}) an exact sequence:
\begin{equation}
\label{shorterGeq}
\Z/2\to H^3_G(X(\C),\Z(1))\to H^3_G(\tX(\C),\Z(1)).
\end{equation}
The Leray spectral sequence for $\pi$ implies that $H^3_G((\Gamma\times\Gamma)(\C),\Z(1))\isoto H^3_G(\tX(\C),\Z(1))$. 
In the Hochschild--Serre spectral sequence (\ref{HS}), the only terms that can contribute to $H^3_G((\Gamma\times\Gamma)(\C),\Z(1))$ are $H^1(G, H^2((\Gamma\times\Gamma)(\C),\Z(1)))=H^1(G,\Z^2)=0$ and $H^3(G,\Z(1))$, whose image in $H^3_G(\Gamma(\C),\Z(1))=H^3(\Gamma(\C)/G,\Z(1))$ vanishes for dimension reasons, so that its image in $H^3_G((\Gamma\times\Gamma)(\C),\Z(1))$ also vanishes.  The rightmost term of (\ref{shorterGeq}) therefore vanishes, proving~\ref{ii}.
\end{proof}

\begin{lem}
\label{lemomega3}
Let $X$ be an integral proper variety over $\R$.  Consider the properties
\begin{enumerate}[(i)]
\item
\label{i} $H^2(X,\sO_X)=0$\textrm{ and }$\Pic(X)\to\Pic(X_{\C})^G$ is surjective;
\item 
\label{ii}$H^3_G(X(\C),\Z(3))$ is $2$-torsion.
\end{enumerate}
If \ref{i} holds, then the class $\omega^3_X\in H^3_G(X(\C),\Z(3))$ is nonzero.  If moreover \ref{ii} holds, then so is $\oo^3_X\in H^3_G(X(\C),\Z/2)$.
\end{lem}

\begin{proof}
The exponential exact sequence $0\to \Z(1)\xrightarrow{2\pi i}\sO_{X(\C)}\xrightarrow{\exp}\sO_{X(\C)}^*\to 0$ on $X(\C)$ is $G$\nobreakdash-equi\-variant and (applied on $X$ and on the point) yields a commutative exact diagram
\begin{equation}
\begin{aligned}
\label{expseq}
\xymatrix
@R=3ex
@C=0.27cm{
0=H^2_G(X(\C),\sO_{X(\C)})\ar[r]&H^2_G(X(\C),\sO_{X(\C)}^*) \ar[r] & H^3_G(X(\C),\Z(1))& \\
&H^2(G,\C^*)\ar[r]\ar[u]&H^3(G,\Z(1))\ar[u]\ar[r]& H^3(G,\C)=0,
}
\end{aligned}
\end{equation}
in which one sees that $H^2_G(X(\C),\sO_{X(\C)})=H^2(X(\C),\sO_{X(\C)})^G=0$ by combining \mbox{\cite[\S A.3]{Tight}}, \cite[Expos\'e XII, Corollaire 4.3]{SGA1} and \ref{i}. 

Serre's GAGA theorem and \v{C}ech computations of ($G$-equivariant) cohomology give identifications $\Pic(X_{\C})=H^1(X(\C),\sO_{X(\C)}^*)$ and $\Pic(X)=H^1_G(X(\C),\sO_{X(\C)}^*)$ (see \cite[\S A.2, \S A.4]{Tight}). We therefore deduce from \ref{i} and the Hochschild--Serre spectral sequence~(\ref{HS})
$$E_2^{p,q}=H^p(G,H^q(X(\C),\sO_{X(\C)}^*))\Rightarrow H^{p+q}_G(X(\C),\sO_{X(\C)}^*)$$
that the left vertical map $H^2(G,\C^*)\to H^2_G(X(\C),\sO_{X(\C)}^*)$ in (\ref{expseq}) is injective. 

 It now follows from (\ref{expseq}) that the pull-back map $H^3(G,\Z(1))\to H^3_G(X(\C),\Z(1))$ is injective, hence that $\omega^3_X$ is nonzero.  If \ref{ii} holds,  the long exact sequence of $G$\nobreakdash-equivariant cohomology associated with $0\to \Z(3)\xrightarrow{2}\Z(3)\to\Z/2\to 0$ implies that so is $\oo^3_X$.
\end{proof}

We leave the following question open when $n \geq 4$.

\begin{quest}
Fix an integer $n$ not of the form $2^k-1$.  Does there exist a variety~$X$ of dimension $n$ over $\R$, with $X(\R)=\varnothing$, such that $\oo^{2n-1}_X$ is nonzero?
\end{quest}

\subsection{\texorpdfstring{$\Q/\Z$}{Q/Z} coefficients}
\label{parQZ}

For $e\geq 0$,  let $\omega_{\Q/\Z}^e\in H^e(G,\Q/\Z(e+1))\simeq \Z/2$ be the nonzero element. (Beware that despite the notation, it is not an $e$th power for the cup product.)
 If $S$ is a $G$-space, we still denote by $\omega_{\Q/\Z}^e$ the induced class in $H^e_G(S,\Q/\Z(e+1))$.

The class $\omega^{e+1}$ is the image of $\omega_{\Q/\Z}^e$ by the boundary map of the short exact sequence of $G$-modules $0\to \Z(e+1)\to\Q(e+1)\to\Q/\Z(e+1)\to 0$. 
In situations where one has a vanishing result for $\omega^{e+1}$ (for instance in the setting of Proposition \ref{omegalow} or of Theorem~\ref{omega2n-1}),  it is therefore natural to ask whether this is explained by a stronger vanishing result for~$\omega_{\Q/\Z}^e$.
In the next two propositions, we show that it is sometimes the case,  and sometimes not. 

\begin{prop}
There exists a smooth projective surface $X$ over $\R$ with $X(\R)=\varnothing$ such that $\omega_{\Q/\Z}^2\in  H^2_G(X(\C),\Q/\Z(1))$ is nonzero.
\end{prop}

\begin{proof}
Let $X$ be a projective real $\KT$ surface with $X(\R)=\varnothing$ carrying a line bundle $L$ of degree $2$. One may for instance let $\pi\colon X\to\P^2_{\R}$ be the double cover with equation $\{w^2+x^6+y^6+z^6=0\}$ and take $L:=\pi^*\sO_{\P^2_\R}(1)$.

Set $x:=\cl(L)\in H^2_G(X(\C),\Z(1))$ and let $\ox\in H^2_G(X(\C),\Z/2)$ be its reduction mod~$2$.  As a general section of $L$ is a smooth geometrically irreducible curve of genus $2$, it follows from \cite[Theorem 3.6]{bw1} that $\omega^2\smile x\in H^4_G(X(\C),\Z(3))$ is nonzero. Since the reduction mod~$2$ morphism
$H^4_G(X(\C),\Z(3))\to H^4_G(X(\C),\Z/2)$ is an isomorphism (by Poincar\'e duality), we further deduce that $\oo^2\smile\ox$ is nonzero, hence that $\oo^2\neq 0$ in $H^2_G(X(\C),\Z/2)$. 

The Hochschild--Serre spectral sequence (\ref{HS}) shows that $H^1_G(X(\C),\Q/\Z(1))=0$.  
The inclusion of $G$-modules $\iota\colon\Z/2\hookrightarrow\Q/\Z(1)$ therefore induces an injective morphism $H^2_G(X(\C),\Z/2)\to H^2_G(X(\C),\Q/\Z(1))$.
As $\omega_{\Q/\Z}^2$ is the image of $\oo^2$ by the map $H^2(G,\Z/2)\xrightarrow{\iota}H^2(G,\Q/\Z(1))$, we deduce that $\omega_{\Q/\Z}^2\in H^2_G(X(\C),\Q/\Z(1))$ is nonzero.
\end{proof}

In the next proposition and its proof, we let $\iota\colon\Z/2\hookrightarrow\Z/4(1)$ be the inclusion and we still denote by $\iota$ the morphisms it induces in cohomology.

\begin{prop}
\label{o8QZ}
Let $M$ be a stably complex $\ci$ $G$-manifold~$M$ of dimension~$5$ with $M^G=\varnothing$.  Then $\iota(\oo^8_M)\in H^8_G(M,\Z/4(1))$ and $\omega^8_{\Q/\Z}\in H^8_G(M,\Q/\Z(1))$ vanish.
\end{prop}

\begin{proof}
The second statement is a consequence of the first, which we now prove. Arguing as in Case \ref{Case2mod2} of the proof of Theorem \ref{reltopo}, we may assume that $M$ is compact.

Let $\tau\colon M\to \BU\times\Z$ be a $G$-equivariant map classifying $\tau_M$, and let $\dA$ act on the right on $H^*_G(\BU,\Z/2)$ and $H^*_G(M,\Z/2)$ by means of $-\kappa_5$ and $-[TM]=-\tau^*\kappa$ respectively (where $\kappa$ and $\kappa_5$ are as in \S\ref{parintrelations}).  Using (\ref{rightactionwu}) and Corollary \ref{Wukappa}, one can show that
\begin{equation}
\label{identitedim5bis}
(\oo^4)\Sq^4+(\oo^2)\Sq^4\Sq^2=\oo^8+\oo^6\oc_1+\oo^2(\oc_1^3+\oc_3)
\end{equation}
in $H^8_G(\BU,\Z/2)$.  The class $\oo^6\oc_1+\oo^2(\oc_1^3+\oc_3)$ is the image of $\oo^5\oc_1+\oo(\oc_1^3+\oc_3)$ by the boundary map of $0\to\Z/2\xrightarrow{\iota}\Z/4(1)\to\Z/2\to 0$ (as follows from the commutativity of the
diagram~\eqref{cd:bocktwist} tensored with~$\Z(1)$), and hence is killed by $\iota$.  Moreover, the class~$(\oo^4)\Sq^4$ vanishes in $H^8_G(M,\Z/2)$ by Theorem \ref{reltopo}.  It therefore follows from (\ref{identitedim5bis}) that
$\iota(\oo^8)=\iota((\oo^2)\Sq^4\Sq^2)\textrm{ in }H^8_G(M,\Z/4(1))$.

It remains to show that the class $\iota((\oo^2)\Sq^4\Sq^2)$ vanishes. To do so, fix $z\in H^2_G(M,\Z/4)$ and let $\bar{z}\in H^2_G(M,\Z/2)$ be its reduction mod $2$.  One computes that
\begin{equation*}
\label{sansiota}
(\oo^2)\Sq^4\Sq^2\smile \bar z=\oo^2\smile \Sq^4\Sq^2(\oz)=\oo^2\smile\Sq^2(\oz^3)=(\oo^2)\Sq^2\smile\oz^3=\oo^2\oc_1\smile\oz^3=0
\end{equation*}
in $H^{10}_G(M,\Z/2)$, where the first and third equalities follow from Proposition \ref{droitegauche}, the second equality stems from the Cartan formula (\ref{Cartan}) since $\Sq^1(\oz)=0$, and the fourth equality is a computation based on (\ref{rightactionwu}) and Corollary \ref{Wukappa}. As for the last vanishing, note that~$\oo^2\oc_1\smile\oz^3$ lifts to a $2$-torsion class in $H^{10}_G(M,\Z/4(1))$ and that the reduction mod $2$ map 
$\Z/4=H^{10}_G(M,\Z/4(1))\to H^{10}_G(M,\Z/2)=\Z/2$ vanishes on $2$-torsion classes.

We deduce that $\iota((\oo^2)\Sq^4\Sq^2)\smile z=\iota((\oo^2)\Sq^4\Sq^2\smile \bar z)=0$ for all $z\in H^2_G(M,\Z/4)$, and Poincar\'e duality (see Proposition \ref{Poincareduality} (1)) shows that $\iota((\oo^2)\Sq^4\Sq^2)=0$, as desired.
\end{proof}

\section{The coindex of real varieties with no real points}
\label{seccoindex}

Recall that the \textit{coindex} $\coind(S)$ of a $G$-space $S$ with $S^G=\varnothing$ is the smallest nonnegative integer~$m$ such that there exists a $G$-equivariant map $f\colon S\to\bS^m$
(or $\infty$ if there is no such~$m$),  where the sphere $\bS^m$ is endowed with the antipodal involution.
A natural source of $G$\nobreakdash-spaces with no fixed point are sets of complex points of smooth varieties over $\R$ with no real points.  In this section, we study the coindex of such $G$-spaces (and more generally of stably complex~$\ci$ $G$-manifolds with no $G$-fixed points) with the hope that the (stable) complex structure may force the coindex to be smaller than expected.
Our guiding question~is:

\begin{quest}
What is the highest possible value of $\coind(X(\C))$, where $X$ is a smooth variety of dimension $n$ over $\R$ with $X(\R)=\varnothing$?
\end{quest}

We shall apply
Theorem~\ref{omega2n-1}
and establish
Theorems~\ref{coindexnintro}
and~\ref{coindexsurfacesintro}
from the introduction.

\subsection{The primary obstruction}

As pointed out in \cite[p.~419]{CF1}, the coindex of a $G$\nobreakdash-space $S$ with $S^G=\varnothing$ can be studied using obstruction theory applied to the space~$S/G$. Assuming that $S/G$ has a structure of CW complex, as is the case in the situations we consider, 
constructing a $G$\nobreakdash-equivariant map $f\colon S\to\bS^m$ inductively over the $k$-skeleton of $S/G$ for increasing values of~$k$ requires the vanishing of successive obstruction classes living in $H^{k+1}_G(S,\pi_k(\bS^m))$, where~$G$ acts on $\pi_k(\bS^m)$ through the antipodal involution (see \eg \cite[Theorem 34.2]{Steenrodbundles} where, by convention, $\pi_0$ means degree $0$ reduced homology).

The first obstruction encountered lives in $H^{m+1}_G(S,\pi_m(\bS^m))=H^{m+1}_G(S,\Z(m+1))$. This primary obstruction does not depend on any choice, and is functorial in the $G$\nobreakdash-space~$S$ (see \cite[Theorem~35.7]{Steenrodbundles}). It is therefore induced, for all $G$-spaces $S$, by a class in $H^{m+1}_G(EG,\Z(m+1))=H^{m+1}(G,\Z(m+1))=\{0,\omega^{m+1}\}$. We deduce that this class is either always zero, or always equal to $\omega^{m+1}_S$.  It cannot be always zero,  as its vanishing on~$S=\bS^{m+1}$ would imply the existence of a $G$-equivariant map $\bS^{m+1}\to\bS^m$, contradicting the Borsuk--Ulam theorem \cite[Satz II]{Borsuk}. We deduce that this class is equal to $\omega^{m+1}_S$.  

\begin{thm}
\label{coindexn}
Let $M$ be a stably complex~$\ci$ $G$-manifold of dimension~$n$ with $M^G=\varnothing$.
\begin{enumerate}[(i)]
\item One has $\coind(M)\leq 2n$, and $\coind(M)\leq 2n-1$ if $n$ is not of the form $2^k-1$.
\item Suppose that $M$ has no compact connected component. Then $\coind(M)\leq 2n-1$, and $\coind(M)\leq 2n-2$ if $n$ is not of the form $2^k-1$.
\end{enumerate}
\end{thm}

\begin{proof}
The first assertions of both (i) and (ii) hold because all the obstructions live in cohomology groups that vanish. Similarly, the second assertion of (i) (\resp of (ii)) holds because the only obstruction that lives in a cohomology group that might not vanish is equal to $\omega^{2n}_M$ (\resp to $\omega^{2n-1}_M$). This obstruction therefore vanishes by Theorem~\ref{omega2n-1}. 
\end{proof}

The following algebraic variant of Theorem \ref{coindexn} allows for singularities.

\begin{thm}
\label{coindexn2}
Let $X$ be a variety of dimension $n$ over $\R$ with $X(\R)=\varnothing$.  
Then $\coind(X(\C))\leq 2n$. If either $X$ has no proper irreducible component, or $n$ is not of the form $2^k-1$, then $\coind(X(\C))\leq 2n-1$.
\end{thm}

\begin{proof}
Run the proof of Theorem \ref{coindexn}, using Proposition \ref{propsing} instead of Theorem \ref{omega2n-1}.
\end{proof}

\begin{rmk}
(i)
We do not know if there exists $n\geq 0$ such that all stably complex~$\ci$ $G$-manifolds $M$ of dimension $n$ with $M^G=\varnothing$ have coindex~$\leq 2n-2$.

(ii)
The hypothesis on $n$ in the second half of Theorem \ref{coindexn} (i) cannot be dispensed with, as the obstruction~$\omega^{2n}_M$  need not vanish if $n=2^k-1$, see Proposition~\ref{prop:omegai quadriques}.

(iii)
There is no algebraic variant with singularities of the second half of Theorem \ref{coindexn}~(ii). Indeed,  let $X$ be the surface constructed in Proposition \ref{singularomega3}. Let $X^0\subset X$ be the complement of a smooth closed point.  As $\oo^3_X\neq 0$ by Proposition \ref{singularomega3}, one has $\oo^3_{X^0}\neq 0$ by purity. So the primary obstruction $\omega^3_{X^0}$ to $X^0(\C)$ having coindex $\leq 2n-2$ does not vanish.
\end{rmk}

\subsection{The secondary obstruction for surfaces}
\label{subsec:secondary}

To understand to what extent Theorem~\ref{coindexn} (i) can be improved when $n=2$, we  must describe a secondary obstruction.  This is the content of the next proposition.
For any $G$\nobreakdash-space~$S$ and any $i\geq 0$,
we slightly abuse notation and 
write
$\Nr_{\Z(i)}\colon  H^i(S,\Z(i-1))\to H^i_G(S,\Z(i))$
for the norm map defined as the composition
of the norm map $H^i(S,\Z(i)) \to H^i_G(S,\Z(i))$
appearing in (\ref{rc})
with the canonical (but non-$G$\nobreakdash-equivariant) isomorphism
$H^i(S,\Z(i-1))=H^i(S,\Z(i))$.

\begin{prop}
\label{prop2obs}
Let $M$ be a stably complex
$\ci$ $G$-manifold of dimension $2$ with $M^G=\varnothing$. Then $\coind(M)\leq 2$ if and only if  there exists $a\in H^2(M,\Z(1))$ with $a^2=0$ and $\Nr_{\Z(2)}(a)=\omega_M^2$.
\end{prop}

\begin{proof}
Applying (\ref{rc}) to $S=\bS^2$ yields an exact sequence
$$H^2(\bS^2,\Z(1))\xrightarrow{\Nr_{\Z(2)}}H^2_G(\bS^2,\Z(2))\to H^3_G(\bS^2,\Z(1))=0$$
where the last group vanishes for dimension reasons. It follows that the generator~$\omega^2_{\bS^2}$ of $H^2_G(\bS^2,\Z(2))\simeq\Z/2$ lifts to the generator $b:=[\bS^2]$ of $H^2(\bS^2,\Z(1))$.
If $f\colon M\to\bS^2$ is $G$-equivariant, then $a:=f^*b$ has the required properties, as $b^2\in H^4(\bS^2,\Z(2))=0$.

Conversely, let $a\in H^2(M,\Z(1))$ be such that $a^2=0$ and $\Nr_{\Z(2)}(a)=\omega_M^2$. Write the following portion of the real-complex exact sequence (\ref{rc}) for $M$:
\begin{equation}
\label{realcxX}
H^2_G(M,\Z(1))\to H^2(M,\Z(1))\xrightarrow{\Nr_{\Z(2)}}H^2_G(M,\Z(2))\xrightarrow{\omega}H^3_G(M,\Z(1)).
\end{equation}
By (\ref{realcxX}), the existence of $a$ with $\Nr_{\Z(2)}(a)=\omega^2_M$ implies that $\omega^3_M=0$.  
Fix a CW complex structure on $M$ with $k$-skeleton $M_{\leq k}$. 
As the primary obstruction~$\omega^3_M$ to $M$ having coindex~$\leq 2$ vanishes, there exists a $G$-equivariant map $f_{\leq 3}\colon M_{\leq 3}\to\bS^2$.  Define $a':=f_{\leq 3} ^*b$ in $H^2(M_{\leq 3},\Z(1))= H^2(M,\Z(1))$.  One has $\Nr_{\Z(2)}(a')=f_{\leq 3} ^*\Nr_{\Z(2)}(b)=f_{\leq 3} ^*\omega^2_{\bS^2}=\omega^2_{M}$.  
By~(\ref{realcxX}), there exists $c\in H^2_G(M,\Z(1))$ such that $a'-a$ is induced by~$c$.  After modifying~$f_{\leq 3}$ by the primary difference $c$ (as in \cite[Theorem~37.5]{Steenrodbundles}), we may assume that $a'=a$ (by \cite[Lemma 37.9]{Steenrodbundles}).
The obstruction to  $G$\nobreakdash-equivariantly extending~$f_{\leq 3}|_{M_{\leq 2}}$ to all of $M$ lives in $H^4_G(M,\pi_3(\bS^2))=H^4_G(M,\Z)$ because the antipody of $\bS^2$ acts trivially on $\pi_3(\bS^2)\simeq \Z$.
 We claim that this obstruction vanishes, which will conclude the proof.  As the natural map $H^4_G(M,\Z)\to H^4(M,\Z)$ is injective by Poincar\'e duality, it suffices to show that the obstruction to (non\nobreakdash-$G$\nobreakdash-equivariantly) extending $f_{\leq 3}|_{M_{\leq 2}}$ to all of $M$, which lives in $H^4(M,\Z)$,  vanishes. But this obstruction has been shown by Steenrod \cite[Theorem~24.1]{Steenrodspheres} to be equal to $(a')^2=a^2=0$.
\end{proof}

\begin{rmk}
It is also possible to describe secondary obstructions in higher dimension, thanks to the work of Liao \cite[Remark~1, p.~184]{Liao}. 
\end{rmk}

\subsection{Two lemmas on \texorpdfstring{$G$}{G}-equivariant cohomology}

The next two lemmas will be useful in our study of the coindex of real algebraic surfaces in \S\ref{parcoindcs}.

\begin{lem}
\label{normomega2}
Let $S$ be a $G$-space and fix $i\geq 1$.  If $a\in H^i(S,\Z(i-1))$ is such that $\Nr_{\Z(i)}(a)=\omega^i_S$, then $a\in H^i(S,\Z(i-1))^G$.
\end{lem}

\begin{proof}
The composition $H^i(S,\Z(i-1))\xrightarrow{\Nr_{\Z(i)}}H^i_G(S,\Z(i))\xrightarrow{\phi} H^i(S,\Z(i))=H^i(S,\Z(i-1))$ of the norm map,
 the natural map~$\phi$, and the canonical non-$G$\nobreakdash-equivariant
isomorphism
 $H^i(S,\Z(i-1))=H^i(S,\Z(i))$
already used
at the beginning of~\textsection\ref{subsec:secondary}
 sends $a$ to $a-\sigma(a)$.
  We deduce that $a-\sigma(a)=\phi(\Nr_{\Z(i)}(a))=\phi(\omega^i_S)=0$.
\end{proof}

\begin{lem}
\label{Geqauto}
Let $S$ be a $G$-space and fix $i\geq 0$. If $a\in H^i(S,\Z)$, then~$a\smile \sigma(a)$ is in the image of the natural map $H^{2i}_G(S,\Z(i))\to H^{2i}(S,\Z)$.
\end{lem}

\begin{proof}
Let $K:=K(\Z,i)$ be an Eilenberg--MacLane space and 
let $G$ act on $K\times K$ by exchanging the two factors. 
The K\"unneth exact sequence (see \cite[Theorem 3B.6]{Hatcher})
$$0\to\bigoplus_{r+s=q}H_r(K,\Z)\otimes H_s(K,\Z) \to H_{q}(K\times K,\Z)\to \bigoplus_{r+s=q-1} \Tor^{\Z}_1(H_r(K,\Z), H_s(K,\Z))\to 0$$
and the universal coefficient theorem (see \cite[Theorem 3.2]{Hatcher})
$$0\to \Ext^1_{\Z}(H_{q-1}(K\times K,\Z),\Z)\to H^q(K\times K,\Z)\to\Hom(H_q(K\times K,\Z),\Z)\to 0$$
imply, in view of the Hurewicz theorem, that $H^q(K\times K,\Z)$ is an extension of $G$-modules of the form $A[G]$ when~$0<q<2i$.  We deduce that $H^p(G, H^q(K\times K,\Z(i)))=0$ for $p>0$ and $0<q<2i$. 
The Hochschild--Serre spectral sequence (\ref{HS})
$$E_2^{p,q}=H^p(G,H^q(K\times K,\Z(i)))\Rightarrow H^{p+q}_G(K\times K,\Z(i))$$
therefore yields an exact sequence
\begin{equation}
\label{sesHS}
H^{2i}_G(K\times K,\Z(i))\to H^{2i}(K\times K,\Z(i))^G\to H^{2i+1}(G, \Z(i))\to H^{2i+1}_G(K\times K,\Z(i)).
\end{equation}
As the restriction to a $G$-fixed point of $K\times K$ yields a retraction of the right arrow of~(\ref{sesHS}) which is therefore injective, the left arrow of (\ref{sesHS}) is surjective. 
Letting $\tau\in H^i(K,\Z)$ denote the universal class,  it follows that $\tau\times\tau\in H^{2i}(K\times K,\Z(i))^G$ lifts to a class $\rho\in H^{2i}_G(K\times K,\Z(i))$.

Let $f\colon S\to K$ be a map classifying $a$.  Then $g:=(f,f\circ\sigma)\colon S\to K\times K$ is $G$-equivariant and the image of $g^*\rho$ by the natural map $H^{2i}_G(S,\Z(i))\to H^{2i}(S,\Z)$ is~${f^*\tau\smile\sigma^* f^*\tau=a\smile \sigma(a)}$.
\end{proof}

\subsection{The coindex of real algebraic curves and surfaces}
\label{parcoindcs}

We now focus on $G$-spaces of the form $X(\C)$, where $X$ is a low-dimensional smooth proper variety over $\R$.

\begin{lem}
\label{coind0ou1}
Let $X$ be an irreducible smooth proper variety over $\R$ with $X(\R)=\varnothing$.
\begin{enumerate}[(i)]
\item $\coind(X(\C))=0$ if and only if $X$ is not geometrically irreducible.
\item $\coind(X(\C))=1$ if and only if the natural map $\Pic^0(X)\to \Pic^0(X_\C)^G$ is not onto.
\end{enumerate}
\end{lem}

\begin{proof}
Assertion (i) is obvious and we now consider (ii). 
If $X$ is not geometrically irreducible, then $\Pic^0(X)\to \Pic^0(X_\C)^G$ is onto. Consequently, to prove (ii), we may assume that $X$ is geometrically irreducible and hence that $\coind(X(\C))\geq 1$.  Obstruction theory then shows that $\coind(X(\C))=1$ if and only if~$\omega^2_X=0$ (as $\bS^1$ is a $K(\Z,1)$, the primary obstruction is the only one).  Assertion~(ii) now follows from \cite[Corollary~3.4]{vanhameldivisors} (where the exponent $G$ is missing in the statement, and whose proof does not use the projectivity hypothesis).
\end{proof}

It is easy to entirely describe the coindex of curves.

\begin{prop}
\label{coindexcurves}
Let $X$ be an irreducible smooth proper curve over $\R$ with $X(\R)=\varnothing$.
\begin{enumerate}[(i)]
\item $\coind(X(\C))=0$ if and only if $X$ is not geometrically irreducible.
\item $\coind(X(\C))=1$ if and only if $X$ is geometrically irreducible of odd genus.
\item $\coind(X(\C))=2$ if and only if $X$ is geometrically irreducible of even genus.
\end{enumerate}
\end{prop}

\begin{proof}
This follows from Lemma \ref{coind0ou1} and \cite[Proposition 2.2 (2)]{grossharris} (which goes back to Geyer, see \cite[p.~91]{geyer}).
\end{proof}

We also fully understand the coindex of surfaces thanks to Theorem \ref{coindexn},  Proposition~\ref{prop2obs} and Lemma \ref{coind0ou1}. These results imply that the coindex of a surface with no real points belongs to $\{0,1,2,3\}$. We now show that all these values arise.

\begin{thm}
\label{coindexsurfaces}
The possible values of $\coind(X(\C))$, when~$X$ ranges over the irreducible smooth projective surfaces over $\R$ with $X(\R)=\varnothing$, are $0$, $1$, $2$ and $3$.
\end{thm}

\begin{proof}
One has $\coind(X(\C))\leq 3$ by Theorem \ref{coindexn} (i).
Let $C_0=\P^1_\C$, let $C_1$ be a genus~$1$ curve with no real point and let~$C_2$ be the real conic with no real point.
For $m \in \{0,1,2\}$,
the surface $X:=\P^1_{\R}\times_{\R}C_m$ satisfies $\coind(X(\C))=m$,
by Proposition~\ref{coindexcurves}.
To find $X$ with $\coind(X(\C))=3$,
we apply Proposition \ref{Enriques} below.
\end{proof}

\begin{prop}
\label{Enriques}
Let $X$ be an Enriques surface over $\R$ such that $X(\R)=\varnothing$. Then $\coind(X(\C))=3$.
\end{prop}

\begin{proof}
We apply Theorem~\ref{coindexn} (i) and Proposition \ref{prop2obs}, and assume for contradiction that there exists $a\in H^2(X(\C),\Z(1))$ with $a^2=0$ and $\Nr_{\Z(2)}(a)=\omega_X^2$.
By Lemma \ref{normomega2}, one has $a\in H^2(X(\C),\Z(1))^G$. The cycle class map $\Pic(X_\C)\to H^2(X(\C),\Z(1))$ is bijective as $H^1(X,\sO_X)= H^2(X,\sO_X)=0,$ so $a$ corresponds to a line bundle $L\in\Pic(X_\C)^G$. 

Let $A$ be an ample line bundle on $X$. If $l\gg 0$,  we deduce from \cite[VIII, Proposition 16.1 (iii)]{BPvdV} that 
$\P (H^0(X_{\C},L\otimes A_{\C}^{\otimes 2l}))$ has even dimension. As there are no nontrivial real forms of even-dimensional complex projective spaces (this follows from \cite[Theorem 1.5]{ArtinBS} since $\Br(\R)=\Z/2$),  the natural $G$-action on $\P (H^0(X_{\C},L\otimes A_{\C}^{\otimes 2l}))$ has a fixed point.  Consequently, the line bundle $L\otimes A_{\C}^{\otimes 2l}$ is defined over~$\R$, hence so is $L$.

Using Krasnov's cycle class map to $G$\nobreakdash-equivariant cohomology (see \cite[\S 1.6.1]{bw1}), we deduce that $a$ lifts to $H^2_G(X(\C),\Z(1))$,
so that $\Nr_{\Z(2)}(a)=0$,
by (\ref{rc}). This contradicts our assumption that $\Nr_{\Z(2)}(a)=\omega_X^2$. Indeed,  the Hochschild--Serre spectral sequence (\ref{HS}) for $S=X(\C)$ and $\sF=\Z(2)$ shows that $\omega_X^2$ is nonzero.
\end{proof}

Proposition \ref{Enriques} dashes the hope of proving Corollary \ref{cor:omega3} in a geometric manner by constructing a $G$-equivariant map $X(\C)\to\bS^2$.  It also shows that Theorems \ref{coindexn} (i) and \ref{coindexn2} are optimal for surfaces.  Surprisingly, the bound provided by Theorem~\ref{coindexn2} in the surface case can always be improved after an appropriate modification.

\begin{thm}
\label{coindexblowup}
Let $X$ be a surface over $\R$ with $X(\R)=\varnothing$.  There exists a projective birational morphism $\pi\colon Y\to X$ such that $\coind(Y(\C))\leq 2$.
\end{thm}

\begin{proof}
By applying Chow's lemma \cite[Th\'eor\`eme 5.6.1]{EGA2} and then compactifying and blowing up,
we may assume that~$X$ is smooth and projective.
By Lemma \ref{coind0ou1} (i), we may further assume that $X$ is geometrically irreducible. 
Let $\pi\colon Y\to X$ be the blow-up of $X$ at four closed points.
Let $b\in H^2(X(\C),\Z(1))$ be the class of an ample line bundle on $X$, and let $c_1,\dots, c_4\in H^2(Y(\C),\Z(1))$ be the classes of the exceptional divisors of $\pi$.  Krasnov's equivariant cycle class map (see \cite[\S 1.6.1]{bw1}) provides
lifts of  $b$ (\resp of $c_1,\dots,c_4$) in $H^2_G(X(\C),\Z(1))$ (\resp in $H^2_G(Y(\C),\Z(1))$), hence
these classes are annihilated by $\Nr_{\Z(2)}$.

Since~$\omega^3_X=0$ by Corollary \ref{cor:omega3}, the real-complex exact sequence~(\ref{rc}) shows the existence of a class $d\in H^2(X(\C),\Z(1))$ with $\Nr_{\Z(2)}(d)=\omega^2_X$.
After adding to $d$ an appropriate multiple of $b$, we may assume that the integer $d^2\in H^4(X(\C),\Z(2))=\Z$ is nonnegative.  In view of Lemmas~\ref{normomega2} and \ref{Geqauto}, the class $d^2$ is in the image of the natural morphism $H^4_G(X(\C),\Z(2))\to H^4(X(\C),\Z(2))$. In other words, the integer $d^2$ is even.
By Lagrange's theorem, one can find $\lambda_1,\dots,\lambda_4\in\Z$ such that $\sum_{i=1}^4\lambda_i^2=\frac{d^2}{2}$. Define $a:=\pi^*d+\sum_{i=1}^4\lambda_ic_i$.  As $c_i^2=-2$ and $c_i\smile\pi^*d=0$ for $1\leq i\leq 4$, and as $c_i\smile c_j=0$ for $i\neq j$, one computes that $a^2=0$. In addition,  $\Nr_{\Z(2)}(a)=\pi^*\Nr_{\Z(2)}(d)=\pi^*\omega^2_X=\omega^2_{Y}$. Applying Proposition \ref{prop2obs} now shows that $\coind(Y(\C))\leq 2$.
\end{proof}

\section{The level of real function fields}
\label{seclevel}

This final section gathers applications to sums of squares problems.
We apply
results from~\textsection\ref{seccomputations}---specifically,
 Propositions~\ref{reln+1}
and~\ref{reln}---and establish Theorems~\ref{uniruledintro},
\ref{levelsurfaceintro} and~\ref{th:intro17}
announced in the introduction.

\subsection{Real closed fields}

In this whole section, we work over an arbitrary real closed field~$R$, which we fix. We set $C:=R(\sqrt{-1})$ and $G:=\Gal(C/R)\simeq \Z/2$.
We need to allow non-archimedean real closed fields for the proof of Theorem~\ref{thlowerdegree}
 to go through over~$\R$ (for the reduction to smooth hypersurfaces).
However, we insist that the results  presented below are already new and interesting when $R=\R$ and $C=\C$.

We refer to \cite[\S 1]{bw1} for generalities on the cohomology and $G$-equivariant cohomology of algebraic varieties over $R$.
Let $X$ be a variety over $R$. If $A$ is a $G$-module, one can define $G$-equivariant cohomology groups $H^k_G(X(C),A)$ using semi-algebraic cohomology. The cohomology classes~$\omega\in H^1_G(X(C),\Z(1))$ and $\oo\in H^1_G(X(C),\Z/2)$ are well-defined (see \eg \cite[\S\S 1.1.1--1.1.2]{bw1}). If $X$ is smooth, one defines Chern classes 
${c_i\in H^{2i}_G(X(C),\Z(i))}$ as the images of the Chern classes in Chow theory by the $G$\nobreakdash-equiv\-ari\-ant cycle class map (see \cite[\S 1.6.1]{bw1}). Let $\oc_i\in H^{2i}_G(X(C),\Z/2)$ denote their reductions mod~$2$. 

Statements concerning algebraic varieties over $C$ that are invariant under extensions of algebraically closed fields are true if they are true for $C=\C$.  We will frequently appeal to such arguments under the name of Lefschetz principle.
In a similar vein,  any purely cohomological statement which is valid over $\R$ can be extended to $R$ by a spreading out argument (see \cite[\S 7]{delfsknebuschonthehomology} for an example of application of this method).
It follows that any relation between~$\oo$ and the~$\oc_i$ (or $\omega$ and the~$c_i$) that holds over the reals also holds over~$R$.  We will freely use this fact in the remainder of~\textsection\ref{seclevel}.

\subsection{Levels of real function fields}

Recall the theorem \cite[Theorem~2]{Pfister2} of Pfister, which we discussed in \S\ref{parintrosquares} over the reals.

\begin{thm}
\label{thPfister}
Let $X$ be an irreducible smooth variety of dimension $n$ over $R$.  If $s(R(X))$ is finite, then $s(R(X))\leq 2^n$. 
\end{thm}

There, we explained that the $2^n$ bound is optimal when $n\leq 2$ and we raised,  following Pfister, the question whether it is optimal for all $n$.

\begin{quest}
\label{qPfister}
For $n\geq 0$, does there exist an irreducible smooth variety $X$ of dimension~$n$ over $R$ such that $s(R(X))=2^{n}$?
\end{quest}

A powerful tool to study Question \ref{qPfister} is the following consequence of the Milnor conjecture proved by Voevodsky \cite{Voevodsky}.  Recall that a cohomology class of a variety over $R$ or $C$
is said to have \textit{coniveau} $\geq c$ if it vanishes after restriction to the complement of a Zariski
closed subset of codimension $\geq c$. (We use this definition in several contexts: possibly $G$-equivariant cohomology of the space of $C$-points or of the space of $R$\nobreakdash-points.)

\begin{thm}
\label{lowlevelcrit}
Let $X$ be an irreducible smooth variety over $R$ and fix $k\geq 1$. The following assertions are equivalent.
\begin{enumerate}[(i)]
\item One has $s(R(X))\leq 2^{k-1}$.
\item The class $\omega^k_X$ has coniveau $\geq 1$.
\item The class $\oo^k_X$ has coniveau $\geq 1$.
\end{enumerate}
\end{thm}

\begin{proof}
As recalled in \cite[§1.1.1]{bw1},
for any variety~$U$ over~$R$,
there are canonical isomorphisms
$H_G^k(U(C),\Z/2) = H^k_\et(U,\Z/2)$
and
$H_G^k(U(C),\Z(k)) \otimes_\Z \Z_2 = H^k_\et(U,\Z_2(k))$.
The latter induces an injection
$H^k_G(U(C),\Z(k))[2] \hookrightarrow H^k_\et(U,\Z_2(k))$.
From these observations, it follows that the statement of Theorem~\ref{lowlevelcrit}
is equivalent to \cite[Proposition 3.3]{Hilbert17}.
\end{proof}

The appearance of powers of $\omega$ in Theorem \ref{lowlevelcrit} explains why the results of the previous sections have a bearing on questions concerning the level of real function fields.

\subsection{Improving Pfister's bound}
\label{subsec:improvingpfister}

In this section, we improve the bound obtained by Pfister in Theorem \ref{thPfister}, under appropriate geometric hypotheses.

\subsubsection{Geometric hypotheses}
 
We first present the geometric properties that will play a role in our results (they will be applied to a variety $Y$ of the  
form $X_C$ for some smooth proper variety $X$ of dimension $n$ over $R$).
If $Y$ is a smooth variety over $C$ and $A$ is an abelian group, we denote by $\sH_Y^q(A)$ the Zariski sheaf on $Y$ associated with the presheaf $U\mapsto H^q(U(C),A)$.  The \textit{unramified cohomology} group $H^q_{\nr}(Y,A)$ is by definition the group $H^0(Y,\sH^q_Y(A))$.  The validity of the Gersten conjecture for semialgebraic cohomology (see the discussion in \cite[\S 5.1]{bw1}) implies that these groups are birational invariants
of smooth projective varieties over $C$ (see \cite[Theorem 4.1.1]{CTbir} for the argument in the entirely similar situation of \'etale cohomology with torsion coefficients).

\begin{prop}
\label{prophyp}
Let $Y$ be a smooth proper variety of dimension $n$ over~$C$.  Each of the following properties implies the next one.
\begin{enumerate}[(i)]
\item The variety $Y$ is uniruled.
\item The group $\CH_0(Y)$ is supported in codimension~$1$ in the sense
of \cite{blochsrinivas},
i.e.\  there exists a closed subvariety $Z\subset Y$ of codimension $\geq 1$ such that for any algebraically closed field~$C'$ containing~$C$, the push-forward map $\CH_0(Z_{C'})\to \CH_0(Y_{C'})$ is onto.
\item The unramified cohomology group $H^n_{\nr}(Y(C),\Q)$ vanishes; equivalently, the unramified cohomology group $H^n_{\nr}(Y(C),\Z)$ vanishes.
\item All cohomology classes in $H^n(Y(C),\Q)$ have coniveau $\geq 1$; equivalently, all cohomology classes in $H^n(Y(C),\Z)$ have coniveau $\geq 1$.
\item The group $H^n(Y,\sO_Y)$ vanishes.
\end{enumerate}
\end{prop}

\begin{proof}
As all these properties are birational invariants, we may assume, by Chow's lemma \cite[Th\'eor\`eme 5.6.1]{EGA2} and resolution of singularities, that $Y$ is projective.
The two conditions appearing in (iii) (\resp in (iv)) are equivalent as $H^n_{\nr}(Y(C),\Z)$ is torsion-free (see \cite[Th\'eor\`eme~3.1]{ctvoisin} if $C=\C$, or \cite[Proposition~5.1~(ii)]{bw1} in general).

To see that (i) implies (ii),  let $Z$ be the support of any ample divisor in $Y$: it will meet all members of a family of rational curves covering $Y_{C'}$.  
The decomposition of the diagonal argument given in \cite[Proposition~3.3~(i)]{ctvoisin} over $\C$ works over $C$ and shows that (ii) implies (iii).
If~(iii) holds,  all classes in $H^n(Y(C),\Q)$ vanish Zariski locally,  so (iv) also holds.
 The last implication (iv)$\Rightarrow$(v) holds over $\C$ by Hodge theory, hence over $C$ by the Lefschetz principle.
\end{proof}

\begin{rmk}
\label{remhyp}
Conjecturally, properties (ii), (iii), (iv) and (v) are all equivalent. Indeed, the generalized Hodge conjecture predicts that (v) implies (iv) and the generalized Bloch conjecture that (iv) implies (ii). When $n=2$, the implications (v)$\Rightarrow$(iv) and (iv)$\Rightarrow$(iii) are unconditional: the former is a consequence of the Lefschetz $(1,1)$ theorem (and of the Lefschetz principle) while the latter results from purity.  The implication (iii)$\Rightarrow$(ii) is not known even when $n=2$,  in which case it amounts to the usual Bloch conjecture.
\end{rmk}

The guiding question of this section, to which Theorems \ref{evendim}, \ref{odddim} and \ref{conicbundles} and Corollaries~\ref{2dim}, ~\ref{coreven} and \ref{cor3} give partial answers, is the following.

\begin{quest}
Let $X$ be an irreducible smooth proper variety of dimension $n\geq 1$ over $R$ with $X(R)=\varnothing$. Assume that $Y:=X_C$ satisfies one of the properties of Proposition~\ref{prophyp}. Does it follow that $s(R(X))\leq 2^{n-1}$?
\end{quest}

\subsubsection{Even-dimensional varieties}

We first consider the case of varieties of even dimension.

\begin{thm}
\label{evendim}
Let $X$ be an irreducible smooth proper variety of dimension $n \geq 1$ over~$R$ with $X(R)=\varnothing$. If $n$ is even and all classes in $H^n(X(C),\Q)$ have coniveau $\geq 1$, then $s(R(X))\leq 2^{n-1}$. 
\end{thm}

\begin{proof}
By Proposition \ref{reln}, there exists $\gamma\in H^{n}_G(\BU, \Z(n))\subset\Z[\omega,(c_i)_{i\geq 1}]/(2\omega)$ (see~(\ref{cohoBU})) such that $\omega^{n}$ does not occur in $\gamma$ and such that $\omega_X^{n+1}=\omega_X \gamma$ in $H^{n+1}_G(X(C), \Z(n+1))$.  As $c_i \in H^{2i}_G(X(C),\Z(i))$ has coniveau~$\geq i$ for all~$i$ and as $\omega^{n}$ does not appear in $\gamma$, the class $\gamma\in H^{n}_G(X(C), \Z(n))$ has coniveau $\geq 1$ (in fact $\geq 2$).
The real-complex exact sequence~(\ref{rc}) and the identity $\omega_X^{n+1}=\omega_X\gamma$ imply that there exists $\delta\in H^{n}(X(C),\Z)$ such that
\begin{equation}
\label{identiteeven}
\omega_X^n=\gamma+\Nr_{\Z(n)}(\delta)\textrm{ in }H^n_G(X(C),\Z(n)).
\end{equation}
By our hypothesis and the equivalence in Proposition~\ref{prophyp}~(iv),  the class $\delta$ has coniveau~$\geq 1$.
By~(\ref{identiteeven}), we deduce that~$\omega_X^n$ has coniveau~$\geq 1$.
Theorem~\ref{lowlevelcrit}~(ii)$\Rightarrow$(i) now shows that~$s(R(X))\leq 2^{n-1}$.
\end{proof}

\begin{cor}
\label{2dim}
Let $X$ be an irreducible smooth proper surface over $R$ such that $X(R)=\varnothing$. If~$H^2(X,\sO_X)=0$, then $s(R(X))\leq 2$.
\end{cor}

\begin{proof}
Combining Theorem \ref{evendim} with the Lefschetz $(1,1)$ theorem,
applied in the same way as in Remark~\ref{remhyp},
proves the corollary.
\end{proof}

\begin{cor}
\label{coreven}
Let $X$ be an irreducible smooth proper variety of dimension $n\geq 1$ over~$R$ with $X(R)=\varnothing$. If $n$ is even and $X_C$ is uniruled,  then $s(R(X))\leq 2^{n-1}$. 
\end{cor}

\begin{proof}
This follows at once from Theorem \ref{evendim} and Proposition~\ref{prophyp} (i)$\Rightarrow$(iv).
\end{proof}

\begin{rmks}
\label{evendimrmks}
(i)
The hypothesis $H^2(X,\sO_X)=0$ cannot be omitted from Corollary~\ref{2dim}:
without it,
Colliot-Thélène~\cite{ctnoetherlefschetz} shows how
the Noether--Lefschetz theorem---which exploits
the nontriviality of $H^2(X,\sO_X)$ via
Hodge theory---can be used
to construct
real surfaces~$X$ with $s(\R(X))=4$.
This can be reformulated using the point of view adopted here;
we explain how, in the case of a
quartic surface $X \subset \P^3_\R$
such that~$X_\C$ is very general,
with $X(\R)=\emptyset$. (One could equally well consider
a sextic double cover
of the plane, as in~\cite{ctnoetherlefschetz}.)
The image in~$H^2(X(\C),\Z/2)$ of the class $[\sO_X(1)] \in \Pic(X)$ is nonzero,
hence the image of this class in $H^2_G(X(\C),\Z/2)$ cannot coincide with $\bar\omega_X^2$.
In addition, as $H^1(X(\C),\Z/2)=0$, the Hochschild--Serre spectral sequence implies
that $\bar\omega_X^2\neq 0$.
Finally, as~$X_\C$ is very general, we have $\Pic(X_\C)=\Z[\sO_{X_{\C}}(1)]$ by the Noether--Lefschetz theorem,
hence $\Pic(X)=\Z[\sO_{X}(1)]$.
Therefore~$\bar\omega_X^2$ is not algebraic.
By Theorem~\ref{lowlevelcrit}, we conclude that $s(\R(X))>2$,
hence that~$s(\R(X))=4$ since~$X(\R)=\emptyset$ (see \textsection\ref{parintrosquares}).

(ii)
In contrast with Corollary \ref{2dim}, when $n\geq 4$, we do not know if one can replace, in Theorem \ref{evendim}, the coniveau  hypothesis with
the assumption that $H^n(X,\sO_X)=0$. In view of Remark \ref{remhyp}, this would follow from the generalized Hodge conjecture.

(iii)
One can check that the relation used  in the proof of Theorem \ref{evendim} when $n=4$ is $\omega^5+\omega(c_1^2+c_2)=0$, valid for all smooth varieties of dimension $4$ over $R$ with no real points.
\end{rmks}

\subsubsection{Odd-dimensional varieties}
\label{secodd}

Here is a counterpart to Theorem \ref{evendim} in odd dimensions.
In its statement, we denote by $N^p H^{n+1}(X(C),\Z)$ the coniveau
filtration on $H^{n+1}(X(C),\Z)$ (see \cite[\textsection5.1]{bw1}).

\begin{thm}
\label{odddim}
Let $X$ be an irreducible smooth proper variety of dimension $n$ over $R$ with $X(R)=\varnothing$. Suppose that $n\geq 3$ is odd,  that $H^n_{\nr}(X(C),\Q)=0$, and that
the $2$\nobreakdash-torsion subgroup of $H^{n+1}(X(C),\Z)/N^2 H^{n+1}(X(C),\Z)$
is trivial.
Then $s(R(X))\leq 2^{n-1}$. 
\end{thm}

\begin{proof}
By Proposition \ref{reln+1}, there exists $\gamma\in H^{n+1}_G(\BU, \Z/2)\subset\Z/2[\oo, (\oc_i)_{i\geq 1}]$ (see~(\ref{cohoBU2})) such that the monomial $\oo^{n+1}$ does not occur in $\gamma$ and such that the equality $\oo_X^{n+1}=\gamma$ holds in $H^{n+1}_G(X(C), \Z/2)$.
One can uniquely decompose~$\gamma$ as $\gamma = \gamma_0 + \gamma_1$ where~$\gamma_0$ (resp.~$\gamma_1$)
is the sum of those
monomials $\oo^e \prod_{i\geq 1} \oc_i^{\mkern1mue_i}$ appearing in~$\gamma$ such that $e + \sum_{i\geq 1} ie_i$ is even (resp.~odd).
For $j \in \{0,1\}$,
taking the sum of the corresponding monomials
 $\omega^e \prod_{i\geq 1} c_i^{\mkern1mue_i}$
provides a lift~$\gamma'_j \in H^{n+1}_G(\BU,\Z(j))\subset\Z[\omega,(c_i)_{i\geq 1}]/(2\omega)$ (see~(\ref{cohoBU}))
 of~$\gamma_j$.

We still denote by~$\gamma'_j$ the image of $\gamma'_j$ in $H^{n+1}_G(X(C),\Z(j))$,
and we denote by $(\gamma'_j)_\C$ its image
by the map $H^{n+1}_G(X(C),\Z(j)) \to H^{n+1}(X(C),\Z)$ induced by the
natural $G$\nobreakdash-equivariant map $\Z(j) \to \Z[G]$.
For later use, we make the following two observations,
which readily follow from the assumption that
 the monomial~$\oo^{n+1}$ does not appear in~$\gamma$
and from the assumption that $n\geq 3$:
the class $\gamma'_0$, in $H^{n+1}_G(X(C),\Z)$, has coniveau~$\geq 2$;
the class $(\gamma'_1)_\C$, in $H^{n+1}(X(C),\Z)$, has coniveau~$\geq 2$.

The morphism of exact sequences of $G$\nobreakdash-modules
\begin{align}
\begin{aligned}
\xymatrix@R=3ex{
0 \ar[r] & \Z(1) \ar[r]^2 & \Z(1) \ar[r] & \Z/2 \ar[r] & 0 \\
0 \ar[r] & \Z(1) \ar[r] \ar@{=}[u] & \Z[G] \ar[r] \ar[u] & \Z \ar[u] \ar[r] & 0
}
\end{aligned}
\end{align}
induces a commutative diagram with exact rows
\begin{align*}
\xymatrix@R=1.5ex@C=1em{
& \gamma'_1 \ar@{}|{\vinbas}[d]\\
H^{n+1}_G(X(C),\Z(1)) \ar[r] & H^{n+1}_G(X(C),\Z(1)) \ar[r] & H^{n+1}_G(X(C),\Z/2) \ar[r] & H^{n+2}_G(X(C),\Z(1))\\\\
H^{n+1}_G(X(C),\Z(1)) \ar@{=}[uu] \ar[r] & H^{n+1}(X(C),\Z) \ar[r]^{N_\Z} \ar[uu]_{N_{\Z(1)}} & H^{n+1}_G(X(C),\Z) \ar[uu] \ar[r] & H^{n+2}_G(X(C),\Z(1)) \ar@{=}[uu] \\
& \zeta  \ar@{}|{\vinhaut}[u] & \omega_X^{n+1}+\gamma'_0 \ar@{}|{\vinhaut}[u]
}
\end{align*}
(whose bottom row is~\eqref{rc} for $\sF=\Z$).
As $\oo_X^{n+1}=\gamma$ in $H^{n+1}_G(X(C),\Z/2)$,
the classes $\gamma'_1$ and $\omega_X^{n+1}+\gamma_0'$ have the same image in $H^{n+1}_G(X(C),\Z/2)$.
By a chase in the above diagram,
we deduce the existence of $\zeta \in H^{n+1}(X(C),\Z)$
such that
\begin{align}
\label{eq:nzzeta}
N_\Z(\zeta)=\omega_X^{n+1}+\gamma'_0 \;\;\text{ in } H^{n+1}_G(X(C),\Z)
\end{align}
and
\begin{align}
N_{\Z(1)}(\zeta)=\gamma'_1 \;\;\text{ in } H^{n+1}_G(X(C),\Z(1))\rlap.
\end{align}
It follows that
$\zeta+\sigma(\zeta)=(\gamma_0')_\C$
and $\zeta-\sigma(\zeta)=(\gamma_1')_\C$,
hence $2\zeta = (\gamma_0')_\C + (\gamma_1')_\C$.
In particular, the class~$2\zeta$ has coniveau~$\geq 2$.
By the last assumption of Theorem~\ref{odddim}, we conclude that~$\zeta$ has coniveau~$\geq 2$.
This implies, by~\eqref{eq:nzzeta}, that~$\omega_X^{n+1}$ has coniveau~$\geq 2$.

By the equivalence in Proposition~\ref{prophyp} (iii), one has $H^n_{\nr}(X(C),\Z)=0$.  We may therefore apply \cite[Proposition 3.5 (ii)$\Rightarrow$(i)]{Hilbert17} (or rather its variant in semialgebraic cohomology obtained by using the comparison theorem between semialgebraic cohomology and $2$-adic \'etale cohomology, see \cite[\S 1.1.1]{bw1})
 to deduce that $\omega^n_X$ has coniveau~$\geq 1$ (there, the hypothesis that $\CH_0(X_C)$ is supported in codimension $\geq 1$ is only used through the consequence that $H^n_{\nr}(X(C),\Z)=0$, and the projectivity hypothesis is not used).  Applying Theorem~\ref{lowlevelcrit}~(ii)$\Rightarrow$(i) now shows that $s(R(X))\leq 2^{n-1}$.
\end{proof}

\begin{rmks}
\label{remamplehyp}
(i) The last assumption
of Theorem~\ref{odddim}
is satisfied
if
$H^{n+1}(X(C),\Z)=N^2H^{n+1}(X(C),\Z)$.
This stronger condition can, in turn,
be ensured
by verifying that $X_C$ contains a smooth ample hyperplane section $Y\subset X_C$ such that all classes in $H^{n-1}(Y(C),\Z)$ have coniveau~$\geq 1$. Indeed, the Gysin morphism $H^{n-1}(Y(C),\Z)\to H^{n+1}(X(C),\Z)$ is surjective as a consequence of the weak Lefschetz theorem (either reduce to the case $C=\C$ by the Lefschetz principle,
 or apply \cite[Proposition~1.14]{bw1}).
The coniveau hypothesis on $Y$ can then be verified by making use of Proposition~\ref{prophyp}.

(ii) When $n=3$ and $C=\C$, the torsion subgroup of
$H^{n+1}(X(\C),\Z)/N^2H^{n+1}(X(\C),\Z)$
measures the defect of the integral Hodge conjecture on~$X_\C$: it coincides with the group denoted~$Z^4(X_\C)$
in~\cite{ctvoisin}, as a consequence of the well-known fact that smooth proper complex threefolds
satisfy the rational Hodge conjecture.
\end{rmks}

\begin{cor}
\label{cor3}
Let $X$ be an irreducible smooth proper threefold over $R$ with $X(R)=\varnothing$. If~$X_C$ is uniruled,  then $s(R(X))\leq 4$. 
\end{cor}

\begin{proof}
We check that the hypotheses of Theorem~\ref{odddim} are satisfied.
Proposition~\ref{prophyp}~(i)$\Rightarrow$(iii) shows that $H^3_{\nr}(X(C),\Q)=0$.
By Remark~\ref{remamplehyp}~(ii)
and by Voisin's theorem~\cite[Theorem~2]{voisinthreefolds}
according to which smooth proper uniruled complex threefolds satisfy the integral Hodge conjecture,
the quotient $H^{4}(X(C),\Z)/N^2 H^{4}(X(C),\Z)$
is torsion-free if $C=\C$, and hence in general by the Lefschetz principle.
\end{proof}

\begin{rmks}
\label{remsodd}
(i)
The relation that is used (through Theorem \ref{odddim}) in the proof of Corollary \ref{cor3} is $\oo^4+\oo^2\oc_1+\oc_1^2+\oc_2=0$, valid for all smooth threefolds over $R$ with no real points (see Proposition \ref{lowrelations} (iii)).

(ii)
Let $X$ be an
irreducible smooth proper variety of dimension $n \geq 1$ over $R$ with $X(R)=\emptyset$ such that~$X_C$ is uniruled. We have seen in Corollaries \ref{coreven} and \ref{cor3} that $s(R(X))\leq 2^{n-1}$
if~$n$ is even or $n=3$.  This raises the question whether the same conclusion may also hold for odd $n\geq 5$.
We do not know the answer. Combining Theorem~\ref{odddim} and Proposition~\ref{prophyp} (i)$\Rightarrow$(iii) shows that it would suffice to give a positive answer to the following question.
\end{rmks}

\begin{quest}
\label{q:torsionfree}
Let~$Y$ be a smooth projective uniruled variety of dimension~$n$ over~$C$.
Is the quotient $H^{n+1}(Y(C),\Z)/N^2 H^{n+1}(Y(C),\Z)$ torsion-free?
\end{quest}

We refer to Lemma \ref{dn} for a positive answer to Question \ref{q:torsionfree} in a significant particular case. The group $H^{n+1}(Y(C),\Z)/N^2 H^{n+1}(Y(C),\Z)$ can be checked to be a birational invariant
among smooth proper varieties (by applying resolution of singularities, the formula for the cohomology
of a blow-up \cite[Theorem~7.31]{voisinbookhodge} and the contravariant functoriality of the coniveau filtration
\cite[Lemma~2.1]{arapurakang} if $C=\C$, and then in general by the Lefschetz principle).
As we have seen
in the proof of Corollary~\ref{cor3}, Question~\ref{q:torsionfree} has a positive answer for $n \leq 3$,
by Voisin's work.
When $n \geq 4$, Question~\ref{q:torsionfree} is open
even under the stronger assumption that~$Y$ is rationally connected.
When~$Y$ is rationally connected, it is also not known whether
the group appearing in Question~\ref{q:torsionfree} can be nontrivial.

\subsubsection{Conic bundles}

As a first step towards the question raised in Remark~\ref{remsodd}~(ii),
we give,
 in the next theorem, a positive answer to it in the special case of conic bundles over an arbitrary base.
Following the suggestion contained in this remark,
it would be possible to establish Theorem~\ref{conicbundles}
by giving a positive answer to Question~\ref{q:torsionfree}
when~$Y$ has a conic bundle structure,
thus exploiting the full force of Theorems~\ref{evendim} and~\ref{odddim}.
We have opted to give a different, much more direct proof below.

\begin{thm}
\label{conicbundles}
Let $f\colon X\to B$ be a morphism of irreducible smooth proper varieties over~$R$ whose generic fiber is a conic.  Assume that $n:=\dim(X)\geq 2$ and that $X(R)=\varnothing$. Then $s(R(X))\leq 2^{n-1}$. 
\end{thm}

\begin{proof}

If $n=2$, the result follows
from Corollary~\ref{coreven}.
We henceforth assume that~$n\geq 3$.
Let $U\subset B$ be a dense affine open subset over which~$f$ is smooth. Let $f_U:X_U\to U$ be the restriction of $f$ over $U$.  One computes
 that $(f_U)_*\Z(n)=\Z(n)$, that $\RR^2(f_U)_*\Z(n)=\Z(n-1)$,  and that $\RR^i(f_U)_*\Z(n)=0$ for all other values of $i$. The Leray spectral sequence for~$f_U$ therefore induces an exact sequence
\begin{equation}
\label{Leraysequence}
H^{n-3}_G(U(C),\Z(n-1))\xrightarrow{\chi} H^{n}_G(U(C),\Z(n))\xrightarrow{f_U^*} H^{n}_G(X_U(C),\Z(n)),
\end{equation}
where $\chi$ is a differential on page $3$ of the spectral sequence.  Consider the composition
\begin{equation}
\label{yetanothersequence}
H^{n}_G(U(C),\Z(n))\isoto H^n_G(U(R),\Z(n))=\bigoplus_{i\textrm{ even }}H^i(U(R),\Z/2)
\end{equation}
of the restriction map, which is an isomorphism as~$U$ is affine and $n>\dim(U)$ (see \cite[Lemma~1.16]{bw1}), with the canonical decomposition of \cite[(1.31)]{bw1}. It follows from \cite[Theorem 2.1]{scheidererpurity} that $H^i(U(R),\Z/2)$ has coniveau $\geq 1$ when $i\geq 1$
(see also the proof of \cite[Lemma~1.2~(i)]{vanhameltorsion} for a more direct argument).
Thus, after shrinking $U$, we may assume that the images of~$\omega^n_U$ and $\chi(\omega^{n-3}_U)$ in $\bigoplus_{i\textrm{ even }}H^i(U(R),\Z/2)$ both belong to the summand $H^0(U(R),\Z/2)$.

We claim that these two images are equal.
To verify this, we need only check that their restrictions
to any $x\in U(R)$ are equal.  The analogue of the exact sequence (\ref{Leraysequence}) obtained after restriction to $x$ reads
\begin{equation}
\label{HSsequence}
H^{n-3}(G,\Z(n-1))\xrightarrow{\chi_x} H^{n}(G,\Z(n))\to H^{n}_G(X_x(C),\Z(n))=0,
\end{equation}
where $X_x=f^{-1}(x)$ and where the group on the right vanishes because $X_x(R)=\varnothing$ and~$n\geq 3$. As $\omega^{n-3}$ and $\omega^n$ are generators of $H^{n-3}(G,\Z(n-1))$ and $H^{n}(G,\Z(n))\simeq\Z/2$ respectively, we deduce from the exactness of (\ref{HSsequence}) that $\omega^n=\chi_x(\omega^{n-3})$, as desired.

As~(\ref{yetanothersequence}) is an isomorphism, we deduce, from the claim, that $\omega^n_U=\chi(\omega^{n-3}_U)$, and it then follows from~(\ref{Leraysequence}) that~$\omega^n_{X_U}=f_U^*(\omega_U^n)=0$. Consequently, the class $\omega^n_X$ has coniveau $\geq 1$ and Theorem~\ref{lowlevelcrit}~(ii)$\Rightarrow$(i) implies that $s(R(X))\leq 2^{n-1}$.
\end{proof}

\begin{rmk}
The proofs of Theorems \ref{evendim}, \ref{odddim} and \ref{conicbundles} rely on an analysis of the coniveau of classes of the form $\omega^e$. It should be interesting to investigate such questions in a systematic manner,  extending the study of the vanishing of the $\omega^e$ undertaken in \S\ref{parvanomegai} and Section \ref{secomegai}.
\end{rmk}

\subsubsection{Highest known levels}

To put the results of \textsection\ref{subsec:improvingpfister} into perspective,
it is natural to look for irreducible varieties~$X$ over~$R$
with $X(R)=\emptyset$ and~$s(R(X))$  as large as possible.
Pfister showed in \cite[Satz~5]{PfisterStufe} that if $Q$ is the real
anisotropic quadric of dimension~$n$,
then~$s(R(Q))$ is the largest power of~$2$ less than or equal to~$n+1$.
This already demonstrates that there is no upper bound for levels of function
fields of real algebraic varieties with no real points.
One can do slightly better:

\begin{prop}
For every $n \geq 1$, there exists an irreducible variety~$X$ over~$R$ of dimension~$n$
such that~$s(R(X))$
is
the largest power of~$2$ less than or equal to~$n+2$.
\end{prop}

\begin{proof}
Starting from $f(x_0,x_1) \in R(x_0,x_1)$ which is a sum of~$4$ but not  $3$ squares
in $R(x_0,x_1)$ (\eg the Motzkin polynomial, see~\cite{cep}, or,
when $R=\R$, the polynomials from \cite{ctnoetherlefschetz}),
let~$X$ denote the affine variety with equation $f(x_0,x_1)+x_2^2 + \dots + x_n^2 = 0$.
It follows from \cite[Chapter~IX, Corollary~2.3 (applied $n-2$ times) and Chapter~XI, Theorem~2.7]{lamthealgebraic}
that~$s(R(X))$ is
the largest power of~$2$ less than or equal to~$n+2$.
\end{proof}

This falls a long way short of answering Question~\ref{qPfister}.
The following observation suggests a positive answer to it for $n=3$.  Fix $d \geq 6$
such that $\congru{d}{2}{4}$.
Let~$X \subset \P^4_\R$
be a very general degree~$d$ hypersurface
 with $X(\R)=\emptyset$.
The equality $s(\R(X))=8$
would result if one knew that the group $\CH_1(X)$ is generated by the class of a plane section of~$X$
(a question briefly discussed in \cite[Question~9.12]{bw2}).
Indeed, if $s(\R(X))\leq 4$,
applying Theorem~\ref{lowlevelcrit}
and \cite[Proposition 3.5 (i)$\Rightarrow$(ii)]{Hilbert17}
would imply
that $\omega^4_X$ has coniveau~$\geq 2$,
\ie that $\omega^4_X$ belongs to the image of the cycle class map
$\CH_1(X) \to H^4_G(X(\C),\Z(2))=H^4_G(X(\C),\Z(4))$. However, by computing this group explicitly,
one can verify that~$\omega^4_X$ is not a multiple of the cycle class of a plane section.

\subsection{Application to Hilbert's 17th problem}
\label{parH17}

Let $f\in R[x_1,\dots,x_n]$ be a nonnegative polynomial, i.e.\ a polynomial that only takes nonnegative values
on~$R^n$. By Artin's solution \cite[Satz 4]{Artin17} to Hilbert's 17th problem, one can write~$f$ as a sum of squares in $R(x_1,\dots,x_n)$. Pfister obtained in \cite[Theorem~1]{Pfister2} a quantitative refinement of Artin's result: the rational function $f$ is a sum of~$2^n$ squares in $R(x_1,\dots, x_n)$. One may ask, as Pfister did in \cite[Problem 1]{PfisterICM}, whether this bound is optimal.

\begin{quest}
\label{qPfister2}
For $n\geq 1$, does there exist a nonnegative $f\in R[x_1,\dots,x_n]$ that is not a sum of $2^n-1$ squares in $R(x_1,\dots, x_n)$?
\end{quest}

A positive answer to Question \ref{qPfister2} would yield a positive answer to Question \ref{qPfister} (choose~$X$ with function field $R\big(x_1,\dots,x_n,\sqrt{-f}\mkern1mu\big)$ and apply
\cite[Chapter~XI, Theorem~2.7]{lamthealgebraic}). 
One may therefore view the former as a concrete motivation to study the latter.  Conversely,  studying the level of real function fields provides insights into Question \ref{qPfister2}. This point of view was used in \cite{Hilbert17} to show that low degree polynomials cannot be used to answer Question \ref{qPfister2} positively:

\begin{thm}[{\cite[Theorem~0.1]{Hilbert17}}]
\label{thH17}
Let $n\geq 1$. Let $f\in R[x_1,\dots,x_n]$ be nonnegative of degree $d$.  If either
\begin{enumerate}[(i)]
\item $d\leq 2n-2$, or
\item $d=2n$ and either $n$ is even, or $n=3$, or $n=5$,
\end{enumerate}
then $f$ is a sum of~$2^n-1$ squares in $R(x_1,\dots, x_n)$. 
\end{thm}

\begin{rmk}
(i)
It was pointed out to us by David Leep that case (i) of Theorem \ref{thH17} could also be deduced from the more elementary methods of \cite{Leep, LeepBecher}. 

(ii) 
Case (ii) of Theorem \ref{thH17} (where $d=2n$) can be shown to hold for all $n\geq 2$. By \cite[Proposition 6.3]{Hilbert17}, it suffices to answer \cite[Question 5.1]{Hilbert17}. This can be done by adapting the arguments of Lemma \ref{dn} below, once the results of \cite{Diamond} are extended from hypersurfaces of projective spaces to double covers of projective spaces.
\end{rmk}

The proof of Theorem \ref{thH17} given in \cite{Hilbert17} relies, among other tools, on substantial cohomological computations on double covers of projective space (see \mbox{\cite[\S 4]{Hilbert17}}). One of our motivations for writing the present article was to replace these concrete computations by more conceptual arguments (allowing, for instance, to replace $\P^n_R$ by other real algebraic varieties).
To illustrate this point of view, we now show how to use Theorems \ref{evendim} and \ref{odddim} to significantly improve the required number of squares, at the expense of a slightly stronger hypothesis on the degree.

\begin{thm}
\label{thlowerdegree}
Let $f\in R[x_1,\dots,x_n]$ be nonnegative of degree~$d$. If~$2\leq d\leq n$ and $(d,n)\neq(2,2)$, then $f$ is a sum of $2^{n-1}$ squares in $R(x_1,\dots, x_n)$.
\end{thm}

\begin{proof}
Let $F\in R[X_0,\dots,X_n]$ be the homogenization of $f$. The exact same specialization argument as in \cite[\S 6B]{Hilbert17} allows us to assume that the zero locus $X\subset \P^n_{R}$ of $F$ is smooth. The nonnegativity of $f$ then implies that $X(R)=\varnothing$. That $d\leq n$ implies that~$X_C$ is Fano, hence uniruled (see \cite[Theorem~3.4]{Debarre}).

We claim that $s(R(X))\leq 2^{n-2}$. If $n$ is odd, this follows from Corollary~\ref{coreven}.
If $n$ is even, then $n \geq 4$ and
this will be a consequence of Theorem~\ref{odddim} and 
Proposition~\ref{prophyp}~(i)$\Rightarrow$(iii) once we have verified that $H^{n}(X(C),\Z)$ has coniveau $\geq 2$. If $d\leq n-1$,  a general hyperplane section of~$X_C$ is again Fano, hence uniruled, and we can 
apply Remark \ref{remamplehyp}~(i). If $d=n$, we apply Lemma \ref{dn} below with $Y=X_C$.

If $K$ is a field of characteristic not $2$, we let $\{x\}\in H^1(K,\Z/2)\simeq K^*/(K^*)^2$ denote the class of $x\in K^*$. Consider $\alpha:=\{f\}\smile\{-1\}^{n-1}\in H^{n}(R(x_1,\dots, x_n),\Z/2)$. The residue of $\alpha$ along the divisor $X\subset\P^n_R$ is equal to $\{-1\}^{n-1}\in H^{n-1}(R(X),\Z/2)$ (see \cite[Proposition 1.3]{CTO}) and hence vanishes, by Theorem \ref{lowlevelcrit} (i)$\Rightarrow$(iii) (noting that~$\{-1\}$ corresponds to $\oo$ by the comparison theorem between \'etale and $G$-equivariant semi-algebraic cohomology \cite[Corollary~15.3.1]{scheiderer}). 
It follows that the class $\alpha$ is an unramified cohomology class.  
As unramified cohomology is a stable birational invariant \cite[Proposition~1.2]{CTO}, there exists $\beta\in H^n(R,\Z/2)$ such that $\alpha$ is induced by $\beta$.

Fix a real closed extension $R'$ of $R(x_1,\dots,x_n)$.  Artin has shown in \cite[Satz 4]{Artin17} that $f$ is a sum of squares in $R(x_1,\dots,x_n)$, hence a square in $R'$. The image of $\alpha$ in $H^n(R',\Z/2)$, which is also the image of $\beta$ in $H^n(R',\Z/2)$, therefore vanishes. As the natural morphism $H^{n}(R,\Z/2)\to H^{n}(R',\Z/2)$ is an isomorphism (both groups are isomorphic to~$\Z/2$ with~$\{-1\}^n$ as a generator), the class $\beta$ vanishes. Hence $\alpha=0$,
which, as a consequence of the Milnor conjecture proved by Voevodsky, is equivalent to $f$ being a sum of~$2^{n-1}$ squares in $R(x_1,\dots,x_n)$
(see \cite[Proposition 2.1]{Henselian}).
\end{proof}

The next lemma follows from Diamond's work \cite{Diamond}.

\begin{lem}
\label{dn}
Fix $n\geq 4$ even. Let $Y\subset\P^n_C$ be a smooth hypersurface of degree $n$. Then all classes in $H^n(Y(C),\Z)$ have coniveau $\geq 2$.
\end{lem}

\begin{proof}
Arguing as in \cite[Lemmas 5.4 and 5.5]{Hilbert17}, we may assume that $C=\C$ and~$Y$ is general. 
Let $G$ be the Grassmannian of lines in $\P^n_{\C}$, with incidence variety $I\subset \P^n_\C\times G$ and projections $q\colon I\to \P^n_{\C}$ and $p\colon I\to G$.
Let $F_Y\subset G$ be the variety of lines of $Y$, with incidence variety $I_Y\subset Y\times G$ and projections $q_Y\colon I_Y\to Y$ and~$p_Y\colon I_Y\to G$. By \cite[Theorem 8]{BvdV}, the variety~$F_Y$ is smooth of the expected dimension $n-3$.

Consider the natural exact sequence $0\to \sE\to\sO_G^{\oplus(n+1)}\to\sQ\to 0$, where
$\sE$ is the rank~$2$ tautological bundle on~$G$. Then $\P(\sE)=I$ with structural morphism $p\colon\P(\sE)\to G$ and~$c_1(\sO_{\P(\sE)}(1))=q^*\alpha$, where~$\alpha:=c_1(\sO_{\P^n(\C)}(1))$. Letting $s_i$ denote the $i$-th Segre class \cite[\S 3.1]{Fulton}, we get $p_*q^*\alpha^{\frac{n}{2}-1}=s_{\frac{n}{2}-2}(\sE)=c_{\frac{n}{2}-2}(\sQ)$ (see \mbox{\cite[\S 3.2]{Fulton}}). As~$c_{\frac{n}{2}-2}(\sQ)$ is part of a $\Z$-basis of $H^{n-4}(G(\C),\Z)$ (see \cite[\S 14.7]{Fulton}), and as the restriction morphism $H^{n-4}(G(\C),\Z)\to H^{n-4}(F_Y(\C),\Z)$ is bijective (see \cite[Theorem 1.1]{Diamond}), the class $(p_*q^*\alpha^{\frac{n}{2}-1})|_{F_Y}=(p_Y)_*(q_Y)^*(\alpha|_Y)^{\frac{n}{2}-1}$ is part of a $\Z$-basis of $H^{n-4}(F_Y(\C),\Z)$. 
By Poincar\'e duality, there is $\beta\in H^{n-2}(F_Y(\C),\Z)$ with $\deg((p_Y)_*(q_Y)^*(\alpha|_Y)^{\frac{n}{2}-1}\smile\beta)=1$.

Set $\gamma:=(q_Y)_*(p_Y)^*\beta$. By the projection formula, one has $\deg((\alpha|_Y)^{\frac{n}{2}-1}\smile\gamma)=1$. This shows that $\gamma$ is a generator of $H^n(Y(\C),\Z)\simeq \Z$. As $\beta$ has coniveau $\geq 1$ (it vanishes on any affine open subset of $F_Y$) and $p$ is not dominant (by dimension), the class $\gamma$ has coniveau $\geq 2$. This concludes the proof of the lemma.
\end{proof}

\bibliographystyle{myamsalpha}
\bibliography{niveau}
\end{document}